\documentclass[11pt, twoside]{article}

\usepackage[english]{babel}

\usepackage{bbm}
\usepackage{amssymb}
\usepackage{amsfonts}
\usepackage{amsmath}
\usepackage{amsthm}
\usepackage{color}
\usepackage{mathrsfs}
\usepackage{txfonts}
\usepackage{bbm}
\usepackage{enumerate}
\usepackage{anysize}
\usepackage{indentfirst}
\usepackage{latexsym}
\usepackage{tabularx}

\usepackage[colorlinks=true,
linkcolor=blue,
citecolor=red,
urlcolor=magenta,
]{hyperref}

\allowdisplaybreaks

\newtheorem{theorem}{Theorem}[section]
\newtheorem{lemma}[theorem]{Lemma}

\newtheorem{proposition}[theorem]{Proposition}

\theoremstyle{definition}

\newtheorem{remark}[theorem]{Remark}
\newtheorem{definition}[theorem]{Definition}

\newcounter{assum}

\renewcommand{\appendix}{\par
\setcounter{section}{0}%
\setcounter{subsection}{0}%
\setcounter{subsubsection}{0}%
\gdef\thesection{\@Alph\c@section}%
\gdef\thesubsection{\@Alph\c@section.\@arabic\c@subsection}%
\gdef\theHsection{\@Alph\c@section.}%
\gdef\theHsubsection{\@Alph\c@section.\@arabic\c@subsection}%
\csname appendixmore\endcsname
}

\numberwithin{equation}{section}

\begin{document}

\arraycolsep=1pt

\title{\bf\Large
Fractional Gagliardo--Nirenberg Inequalities: Pointwise Estimates,
Sharp Asymptotics, and Optimal Target Spaces
\footnotetext{\hspace{-0.35cm} 2020
\emph{Mathematics Subject Classification}. Primary 26D10; Secondary
46E35, 42B25, 42B35.
\endgraf \emph{Key words and phrases}.
Fractional Gagliardo--Nirenberg inequality, pointwise estimate,
sharp asymptotics, optimal target space, Calder\'on--Lozanovski\u{\i} space,
ball Banach function space.
\endgraf
This project is partially supported by the National
Natural Science Foundation of China (Grant Nos. 12431006,
12371093, and 124B2004), the Beijing Natural Science Foundation (Grant No. 1262011), and the Fundamental Research Funds for the
Central Universities (Grant No. 2253200028).
}}
\author{Pingxu Hu, Yinqin Li,
Dachun Yang and Wen Yuan}
\date{\today}
\maketitle

\vspace{-0.8cm}

\begin{center}
\begin{minipage}{13cm}
{\small {\bf Abstract}\quad
We establish two pointwise estimates for fractional difference operators,
tracking explicitly the dependence of the constants on the smoothness
index $s\in(0,1)$.
Using these, within the framework of ball Banach function spaces
we obtain two fractional Gagliardo--Nirenberg inequalities, including the
BMO endpoint case. Furthermore, we establish endpoint asymptotic results as $s\to0^+$
and $s\to1^-$, proving that the asymptotic factors appearing in these
inequalities have optimal order.
Under the additional assumption that the underlying function space is rearrangement
invariant, we show that the optimal Gagliardo--Nirenberg target
spaces are precisely those given by
the Calder\'on--Lozanovski\u{\i} space. This completely characterizes the rearrangement
invariant target spaces for which the corresponding Gagliardo--Nirenberg
inequalities hold, thereby answering
an open question posed by K. Le\'snik, T. Roskovec, and F. Soudsk\'y.
These results can be
applied to various function spaces; in particular, they are
completely new in the off-diagonal and BMO cases.}
\end{minipage}
\end{center}

%
\tableofcontents
%

\section{Introduction}

Gagliardo--Nirenberg interpolation inequalities are a fundamental tool in the
study of Sobolev spaces and partial differential equations. They quantify the
interpolation between different orders of differentiability and different levels
of integrability. Since the works of Gagliardo \cite{Gag59} and Nirenberg
\cite{Nir59}, these inequalities have been extended in several directions.
Fractional Sobolev and nonlocal versions have been studied, for example, in
\cite{BM18,DPV12,V23}; extensions to Lorentz, Orlicz, Morrey type, and
Musielak--Orlicz settings can be found in
\cite{HSS20,MNSS12,MNSS13,SW13}; rearrangement invariant Banach function
space versions, together with related optimality questions, were investigated in
\cite{FFRS19,LRS23,LRS25}. Besides their important role in deriving a
priori estimates for solutions to
partial differential equations, such inequalities also appear in the study of
Sobolev chain rules, multilinear operators, endpoint interpolation phenomena,
and refined embeddings in fractional and limiting Sobolev spaces; see, for
example, \cite{BM18,DPV12,MTT02,V23}.

A particularly effective way to obtain such inequalities is through pointwise
estimates involving maximal operators.
For instance, Maz'ya and Shaposhnikova
\cite[Theorem 1]{MS99} proved the following pointwise Gagliardo--Nirenberg
inequality: for any positive integers $j,k$ with $j<k$, any
$f\in W^{k,1}_{\rm loc}$, and almost every $x\in \mathbb{R}^n$,
\begin{align*}
\left| \nabla^j f(x) \right|
\lesssim\left[ M f(x) \right]^{1-\frac{j}{k}}
\left[ M\left(|\nabla^k f|\right)(x) \right]^{\frac{j}{k}},
\end{align*}
where $M$ denotes the classical Hardy--Littlewood maximal operator.
The strength of such pointwise estimates is that they provide a
unified way to derive norm inequalities on a wide class of function spaces:
once a pointwise bound is available, norm inequalities can
be obtained by combining the boundedness of maximal operators and interpolation or lattice
properties of function spaces under consideration. This idea can be traced back, in
particular, to pointwise multiplicative inequalities of Ka{\l}amajska
\cite{Kal94}, and has been used to derive Gagliardo--Nirenberg inequalities in
weighted Orlicz spaces and rearrangement invariant Banach function spaces;
see, for example, \cite{FFRS19,KPP06a,KPP06b,LRS23,LRS25}.
In the integer order setting, the Calder\'on--Lozanovski\u{\i}
space naturally appears in this approach. More
precisely, as proved in \cite{FFRS19} (see also \cite[Theorem 2.1]{LRS23}),
if $j,k$ are positive integers with $j<k$ and if $X$, $Y$ are
suitable rearrangement invariant Banach function spaces, then, for any
$f\in W^{k,1}_{\rm loc}$,
\begin{align}\label{eq-GNinter}
\left\| \nabla^j f \right\|_{X^{1-\frac{j}{k}}Y^{\frac{j}{k}}}
\lesssim\left\| f \right\|_{X}^{1-\frac{j}{k}}
\left\| \nabla^k f \right\|_{Y}^{\frac{j}{k}},
\end{align}
where $X^{1-\frac{j}{k}}Y^{\frac{j}{k}}$ denotes the associated
Calder\'on--Lozanovski\u{\i} space; see Definition \ref{def-CL}.
Moreover, Le\'snik et al.\ \cite[Corollary 2.4]{LRS23} showed that,
in \eqref{eq-GNinter}, the Calder\'on--Lozanovski\u{\i} target space is optimal
among rearrangement invariant Banach function spaces under the \emph{additional
inclusion} assumption $X\subset Y$.

On the other hand, fractional counterparts of pointwise Gagliardo--Nirenberg inequalities have
also attracted considerable attention. Maz'ya and Shaposhnikova
\cite[Lemma]{MS02i} proved that, if $s\in(0,1)$, $p\in[1,\infty)$, and
$f\in W^{1,p}_{\rm loc}$, then, for almost every $x\in \mathbb{R}^n$,
\begin{align}\label{ms1}
\left[
\int_{\mathbb{R}^n}
\frac{|f(x)-f(y)|^p}{|x-y|^{n+sp}}\,dy
\right]^{\frac{1}{p}}
\lesssim\left[ M\left(|f-f(x)|^p\right)(x) \right]^{\frac{1-s}{p}}
\left[ M\left(|\nabla f|^p\right)(x) \right]^{\frac{s}{p}};
\end{align}
see also \cite[Lemma 3.2]{Spe20}. Pointwise and norm inequalities involving sharp maximal functions and
${\rm BMO}$ norms were further developed in
\cite{Lok13,MRR13,Miy20,Str06,V23}. The replacement of an
$L^\infty$ term by a ${\rm BMO}$ term is a meaningful
endpoint improvement because $L^\infty\subsetneqq {\rm BMO}$ and ${\rm BMO}$
norms naturally occur in critical estimates for PDEs. For instance, BMO-based
logarithmic Sobolev and bilinear estimates have been used in regularity and
blow-up criteria for the Euler and Navier--Stokes equations; see
\cite{KT00n,KT00e}.

Since fractional inequalities contain an additional parameter $s\in(0,1)$, it is
also important to understand their limiting behaviour as $s\to0^+$ and
$s\to1^-$. This is closely connected with the Bourgain--Brezis--Mironescu
(for short, BBM) and Maz'ya--Shaposhnikova (for short, MS) limiting formulae
for fractional Sobolev seminorms; see, for example,
\cite{BBM01,BBM02,DGPYYZ24,DM23,KMX05,MS02}.
Thus, it is a natural question to explore a fractional Gagliardo--Nirenberg theory,
which provides inequalities for each $s\in(0,1)$ and also explain the
dependence of the constants on $s$ and
their endpoint behaviour.

The main purpose of this article is to develop such a fractional Gagliardo--Nirenberg theory
within the framework of ball Banach function spaces.
We first establish two pointwise estimates for fractional difference operators, where the associated constants depend explicitly on the smoothness index $s$; see Theorem \ref{thm-ps}.
One estimate is formulated
in terms of the sharp maximal function of $f$ and the Hardy--Littlewood
maximal function of  $|\nabla^k f|$,
while the other extends the inequality \eqref{ms1} of
Maz'ya and Shaposhnikova. Applying these estimates,
we obtain the Calder\'on--Lozanovski\u{\i} norm inequalities
in Theorem \ref{thmGNS}, including the BMO endpoint estimate.
These inequalities establish a uniform dependence on the smoothness index $s$.
We then show BBM and MS type endpoint asymptotic results (see Theorems \ref{thmBBM} and \ref{thmMS}), which prove that this uniform dependence has optimal order as $s\to1^-$ and $s\to0^+$. Subsequently,
for fixed $s\in(0,1)$, Theorem \ref{thm-fs} provides an improved admissible parameter range,
answering the question raised in \cite[Remark 4.14(iii)]{DLYYZ23} in the non-endpoint regime.
Finally, in the rearrangement invariant setting, we show that the optimal target spaces are precisely those given by the Calder\'on--Lozanovski\u{\i} construction.
This optimality argument, based on a separated bump construction,
characterizes the admissible target spaces and answers the open question regarding
Lorentz target norms raised by Le\'snik,  Roskovec, and Soudsk\'y
in \cite[Question 2.6]{LRS23}.
These results can be applied to various function spaces;
in particular, they are completely new in the off-diagonal
and BMO cases (see Section \ref{sec-app}).

To state the main results, we first introduce some notation.
Let the \emph{notation $L^0$} denote the set of all measurable functions.
For any $k\in \mathbb{N}$ and
$h\in \mathbb{R}^n$, the \emph{$k$-th order difference}
$\Delta^k_h f$ of any $f\in L^0$ is defined by setting
\begin{align}\label{eq-sho-hd}
\Delta_{h}^k
f(\cdot):=
\sum_{j=0}^{k}(-1)^{k-j}\binom{k}{j}f(\cdot+jh),
\end{align}
where $\binom{k}{j}:=\frac{k!}{j!(k-j)!}$.
For any $k\in \mathbb{N}$, $s\in (0,1)$, $q\in [1,\infty)$,
and $f\in L^0$, let
\begin{align*}
\mathfrak{D}^{s,k}_{q}(f)(\cdot):=
\left[\int_{\mathbb{R}^n}\frac{|\Delta^k_h f(\cdot)|^q}{|h|^{n+skq}}\,dh\right]^{\frac{1}{q}}.
\end{align*}
For any $p \in [1, \infty)$ and $f\in L^{p}_{\rm loc}$,
the \emph{powered Hardy--Littlewood maximal function} $M_p(f)$ is defined by setting,
for any $x\in \mathbb{R}^n$,
\begin{align*}
M_p(f)(x):=\sup_{Q}
\left[\frac{1}{|Q|}\int_Q|f(y)|^p\,dy\right]^{\frac{1}{p}}
{\bf 1}_{Q}(x),
\end{align*}
and the \emph{powered sharp maximal function} $M_{p}^{\sharp}(f)$ is defined by setting,
for any $x\in \mathbb{R}^n$,
\begin{align*}
M_{p}^{\sharp}(f)(x)
:=\sup_{Q}\left[
\inf_{c\in \mathbb{R}}
\frac{1}{|Q|}\int_Q |f(y)-c|^p\,dy
\right]^{\frac{1}{p}}{\bf 1}_{Q}(x),
\end{align*}
where the suprema are taken over all cubes $Q$ in $\mathbb{R}^n$.
When $p=1$, $M(f):=M_1(f)$ is the classical
\emph{Hardy--Littlewood maximal function} and
$M^{\sharp}(f):=M_{1}^{\sharp}(f)$ is the classical
\emph{sharp maximal function}.

The first main result of this article is the following pointwise estimate
with explicit dependence on the smoothness index. It is the
starting point of all subsequent norm inequalities. Part {\rm(i)} is
an estimate with a sharp maximal function, whereas part {\rm(ii)} is the
other one of the form that recovers the classical Maz'ya--Shaposhnikova estimate.

\begin{theorem}\label{thm-ps}
Let $k\in \mathbb{N}$ and $p_1,p_2,q\in[1,\infty)$.
\begin{enumerate}[{\rm (i)}]
\item Assume that $0<s_0<\widetilde{s_0}\le 1$ satisfy
$$n\left(\frac{1-s_0}{p_1}+\frac{s_0}{p_2}-\frac{1}{q}\right)< s_0 k \quad
\text{and}
\quad n\left(\frac{1-\widetilde{s_0}}{p_1}+\frac{\widetilde{s_0}}{p_2}-\frac{1}{q}\right)< \widetilde{s_0} k.$$
Then, for any $s\in (s_0, \widetilde{s_0})$ and
$f\in W^{k,1}_{\rm loc}$ and for almost every $x\in \mathbb{R}^n$,
\begin{align}\label{eq-tss}
\mathfrak{D}^{s,k}_q (f)(x)
\lesssim
\frac{1}{[(s-s_0)(\widetilde{s_0}-s)]^{\frac{1}{q}}}
\left[M^{\sharp}_{p_1}(f)(x)\right]^{1-s}\left[M_{p_2}\left(\left|\nabla^k f\right|\right)(x)\right]^{s},
\end{align}
where the implicit positive constant depends only on
$n$, $k$, $p_1$, $p_2$, $q$, $s_0$, and $\widetilde{s_0}$.

\item  Assume that $p_1\ge q$ and $n(\frac{1}{p_2}-\frac{1}{q})<k$.
Then, for any $s\in (0,1)$ and
$f\in W^{k,1}_{\rm loc}$ and for almost every $x\in \mathbb{R}^n$,
\begin{align}\label{eq-tss2}
\mathfrak{D}^{s,k}_q (f)(x)
\lesssim
\frac{1}{[s(1-s)]^{\frac{1}{q}}}
\left[M_{p_1}(f-f(x))(x)\right]^{1-s}\left[M_{p_2}\left(\left|\nabla^k f\right|\right)(x)\right]^{s},
\end{align}
where
the implicit positive constant depends only on $n$, $k$, $p_1$, $p_2$, and $q$.
\end{enumerate}

\end{theorem}

\begin{remark}\label{rem-1.2}
\begin{enumerate}[{\rm (i)}]
\item When $q=p_1=p_2=:p$, Theorem \ref{thm-ps}{\rm(ii)} gives that, for any
$s\in(0,1)$, any $f\in W^{k,1}_{\rm loc}$, and almost every
$x\in\mathbb{R}^n$,
\begin{align}\label{eq-111}
\mathfrak{D}^{s,k}_{p}(f)(x)
&\lesssim [s(1-s)]^{-\frac{1}{p}}
\left[M_p(f-f(x))(x)\right]^{1-s}
\left[M_p\left(\left|\nabla^k f\right|\right)(x)\right]^{s},
\end{align}
where the implicit positive constant is independent of $s$, $f$, and $x$.
In particular, when $k=1$, estimate \eqref{eq-111} recovers the
pointwise estimate of Maz'ya and Shaposhnikova in \eqref{ms1}, including
the asymptotic factor $[s(1-s)]^{-\frac1p}$ for $p\in[1,\infty)$.
The cases with other parameters $q,p_1,p_2$, as well as the
higher order case $k\ge2$, appear to be new.

\item From the point of view of parameters, Theorem \ref{thm-ps}{\rm(ii)}
is closely related to the formal endpoint $s_0=0$ of the assumption
$n(\frac{1-s_0}{p_1}+\frac{s_0}{p_2}-\frac1q)<s_0k$
appearing in Theorem \ref{thm-ps}{\rm(i)}. Indeed, letting $s_0=0$ in this
assumption, we obtain
$n(\frac1{p_1}-\frac1q)<0,$
which is equivalent to $p_1>q$. Thus, this formal endpoint gives the range
of $p_1,q$ in Theorem \ref{thm-ps}{\rm(ii)} except for the
endpoint case $p_1=q$. Even so, we still cannot replace $M_{p_1}(f-f(x))(x)$ in
\eqref{eq-tss2} by the sharp maximal function $M_{p_1}^{\sharp}(f)(x)$, as
in \eqref{eq-tss}, with a constant independent of $s$; see Proposition
\ref{pro-counter}.
\end{enumerate}
\end{remark}

To obtain the pointwise estimates in Theorem \ref{thm-ps} for fractional
difference operators with general parameters, we refine the localization
arguments used in \cite{DGPYYZ24,HLYY-ineq} for weighted norm inequalities
and adapt them to the present pointwise setting.
The proof proceeds through three main steps.
First, we use shifted dyadic grids
to prove a telescopic estimate for the difference operator; see Lemma
\ref{lem-point}. Then we combine this estimate with suitable Poincar\'e
inequalities, and sum over different scales of dyadic cubes, keeping track of the
precise dependence on the smoothness index.
Finally, in the present case (namely the pointwise setting),
a new difficulty appears: the point
$x$ under consideration may not belong to the cubes over which the
averages in the telescopic estimate are taken;
see, for instance, \eqref{eq-Hqj}. Hence, these averages
cannot be controlled directly by maximal operators. We
overcome this possible mismatch by a further use of shifted dyadic grids, together
with quantitative covering estimates; see \eqref{eq-dom},
\eqref{eq-Qvol}, and \eqref{eq-tt}.

As a consequence, we obtain the following norm inequality. It retains
the sharp endpoint dependence on the smoothness index
$s$. The use of the sharp maximal function also allows us to include
a $\operatorname{BMO}$ norm for the lower order term. The terminology
used in the following theorem is presented in
Section \ref{sec-poin}.

\begin{theorem}\label{thmGNS}
Let $X$ and $Y$ be ball Banach function spaces whose lower generalized Boyd
indices are respectively $p_X$ and $p_Y$ (see Definition \ref{def-Boyd}).
Let $k\in \mathbb{N}$ and $q\in [1,\infty)$. Assume that $p_Y\in (1,\infty]$ and
$n(\frac{1}{p_Y}-\frac{1}{q})< k$.
\begin{enumerate}[{\rm(i)}]
\item  If $p_X\in (1,\infty]$, then
there exists $s_0\in (0,1)$ such that,
for any $s\in (s_0,1)$ and $f \in X\cap \dot{W}^{k,Y}$,
\begin{align}\label{eq-gns1}
\left\|
\mathfrak{D}^{s,k}_q (f)
\right\|_{X^{1-s}Y^{s}}
&\lesssim
[(s-s_0)(1-s)]^{-\frac{1}{q}}
\left\| f\right\|_X^{1-s}
\left\|\nabla^k f\right\|_{Y}^{s},
\end{align}
where the implicit positive constant is independent of $s$ and $f$. Moreover, if $q\in[1,p_X)$,
then, for any $s\in (0,1)$ and $f \in X\cap \dot{W}^{k,Y}$,
\begin{align}\label{eq-gns2}
\left\|
\mathfrak{D}^{s,k}_q (f)
\right\|_{X^{1-s}Y^{s}}
&\lesssim
[s(1-s)]^{-\frac{1}{q}}
\left\| f\right\|_X^{1-s}
\left\|\nabla^k f\right\|_{Y}^{s},
\end{align}
where the implicit positive constant is independent of $s$ and $f$.

\item
There exists $s_0\in (0,1)$ such that,
for any $s\in (s_0,1)$ and $f \in \operatorname{BMO}
\cap\dot{W}^{k,Y}$,
\begin{align*}
\left\|
\mathfrak{D}^{s,k}_q (f)
\right\|_{Y^{\frac1s}}
&\lesssim
[(s-s_0)(1-s)]^{-\frac{1}{q}}
\|f\|_{\operatorname{BMO}}^{1-s}
\left\|\nabla^k f\right\|_{Y}^{s},
\end{align*}
where the implicit positive constant is independent of $s$ and $f$.
\end{enumerate}
\end{theorem}

\begin{remark}
\begin{enumerate}[{\rm (i)}]

\item
The range of $q$ in Theorem \ref{thmGNS} is sharp in some sense. More
precisely, the estimates \eqref{eq-gns1} and \eqref{eq-gns2} may fail when
$n(\frac{1}{p_Y}-\frac{1}{q})\ge k.$
Moreover, even when
$n(\frac{1}{p_Y}-\frac{1}{q})<k$, estimate \eqref{eq-gns2} may fail if
$q\ge p_X$; see Section \ref{sec-sharp} for details.

\item Observe that \cite[Theorem 4.1]{HLYY-ineq} with $\Omega=\mathbb{R}^n$
and the non-endpoint case $p_0\in(1,\infty)$ is a special case of Theorem \ref{thmGNS}
in the diagonal case $X=Y$ and, in this case,
the assumption on the underlying function space of Theorem \ref{thmGNS}
is weaker; see Remark \ref{rem-HLYY-assumption} for the detailed comparison.
However, the domain and the endpoint cases treated in
\cite[Theorem 4.1]{HLYY-ineq} are not covered
by Theorem \ref{thmGNS}. Beyond the diagonal setting, the off-diagonal
case $X\ne Y$ and the BMO endpoint estimate in Theorem
\ref{thmGNS}{\rm(ii)} appear to be new. Moreover, when the derivative
space is an ordinary Lebesgue space, this BMO endpoint estimate is
further sharpened in Theorem \ref{thm-L} by exploiting
the symmetry in the order of integration.

\item The restrictions $p_X,p_Y>1$ in Theorem \ref{thmGNS} mainly come
from the use of pointwise estimates involving maximal operators, whose
boundedness may fail at the endpoints.
To reach endpoint spaces, one would need different pointwise tools. In
the integer order setting,
\cite[Theorem 1.2]{LRS25} provided such a new sparse pointwise estimate and used it to
handle rearrangement invariant Gagliardo--Nirenberg inequalities. Whether
an analogous pointwise principle is available for the fractional
inequalities considered here remains an interesting problem.
\end{enumerate}
\end{remark}

The orders of the asymptotic factors as $s\to0^+$ and $s\to1^-$ in Theorem \ref{thmGNS}
are optimal. To show this, we establish endpoint asymptotic estimates,
namely BBM type formulae as $s\to1^-$ and
MS type estimates as $s\to0^+$, on Calder\'on--Lozanovski\u{\i} spaces
$X^{1-s}Y^s$. To do so, we give some necessary notation.
For any $k\in\mathbb{N}$, $q\in[1,\infty)$, $f\in W^{k,1}_{\rm loc}$, and
$x\in\mathbb{R}^n$, define
\begin{align}\label{eq-Dkq}
\mathbb{D}^{k}_q(f)(x):=
\left[\int_{\mathbb{S}^{n-1}}
\left|
\sum_{\alpha\in\mathbb{Z}_+^n,\,|\alpha|=k}
\frac{k!}{\alpha!}\partial^\alpha f(x)\xi^\alpha
\right|^q\,d\mathcal{H}^{n-1}(\xi)
\right]^{\frac1q}
\end{align}
and
\begin{align}\label{eq-Knq}
K_{n,q}:=
\left[
\frac{2\pi^{\frac{n-1}{2}}\Gamma(\frac{q+1}{2})}
{q\Gamma(\frac{q+n}{2})}
\right]^{\frac1q},
\end{align}
where $\Gamma$ denotes the Gamma function.

\begin{theorem}\label{thmBBM}
Let $k\in\mathbb{N}$ and both $X$ and $Y$ be ball Banach function spaces having
absolutely continuous norms (see Definition \ref{def-ac}). Assume that $p_X,p_Y\in(1,\infty)$. Let
$q\in[1,\infty)$ satisfy $n(\frac1{p_Y}-\frac1q)<k$. Then, for any
$f\in X\cap W^{k,Y}$,
\begin{align*}
\lim_{s\to1^-}(1-s)^{\frac1q}
\left\|\mathfrak{D}^{s,k}_q(f)\right\|_{X^{1-s}Y^s}
=(kq)^{-\frac1q}\left\|\mathbb{D}^k_q(f)\right\|_Y.
\end{align*}
Moreover, if $k=1$, then, for any $f\in X\cap W^{1,Y}$,
\begin{align*}
\lim_{s\to1^-}(1-s)^{\frac1q}
\left\|\mathfrak{D}^{s}_q(f)\right\|_{X^{1-s}Y^s}
=K_{n,q}\left\|\nabla f\right\|_Y,
\end{align*}
where $K_{n,q}$ is as in \eqref{eq-Knq}.
\end{theorem}

\begin{remark}\label{rem-imp}
\begin{enumerate}[{\rm(i)}]
\item
We point out that
\begin{align}\label{eq-ssim}
\mathbb{D}^{k}_q(f)(x)\sim |\nabla^k f(x)|,
\end{align}
where the positive equivalence constants depend only on $n$, $k$, and
$q$. In particular, when $k=1$,
\begin{align*}
\mathbb{D}^{1}_q(f)(x)
=
|\nabla f(x)|
\left[
\frac{2\pi^{\frac{n-1}{2}}\Gamma(\frac{q+1}{2})}
{\Gamma(\frac{q+n}{2})}
\right]^{\frac1q};
\end{align*}
see \cite[Lemma 4.7]{HLYYZ-bsvy}. Hence, Theorem \ref{thmBBM} and
\eqref{eq-ssim} prove that, for every admissible $f$ with
$\nabla^k f\not\equiv0$,
\begin{align*}
\left\|\mathfrak D_q^{s,k}(f)\right\|_{X^{1-s}Y^s}
\sim
(1-s)^{-\frac1q}\left\|\nabla^k f\right\|_Y
\end{align*}
as $s\to1^-$. Consequently, the factors
$[(s-s_0)(1-s)]^{-\frac1q}$ in \eqref{eq-gns1} and
$[s(1-s)]^{-\frac1q}$ in \eqref{eq-gns2} have optimal order as
$s\to1^-$.

\item Observe that \cite[Theorem 5.5]{HLYY-ineq} with $\Omega=\mathbb{R}^n$
and the non-endpoint case $p\in(1,\infty)$ is a special case of Theorem \ref{thmBBM}
in the diagonal case $X=Y$ and, in this case,
the assumption on the underlying function space of Theorem \ref{thmBBM}
is weaker; see Remark \ref{rem-HLYY-assumption} for the detailed comparison.
However, the domain and the endpoint cases treated in
\cite[Theorem 5.5]{HLYY-ineq} are not covered
by Theorem \ref{thmBBM}. Beyond the diagonal setting, the off-diagonal
case $X\ne Y$ in Theorem \ref{thmBBM} appear to be new.

\item We also clarify a
minor difference between the higher order quantity
\eqref{eq-Dkq} used here and the one
in \cite[Theorem 5.5]{HLYY-ineq}. In the present article, the limiting
quantity is defined with the factor $\frac{k!}{\alpha!}$ in the
multi-index summation. This factor appears when the
expansion
\begin{align*}
\sum_{j_1=1}^n\cdots\sum_{j_k=1}^n
\xi_{j_1}\cdots\xi_{j_k}
\partial_{j_1}\cdots\partial_{j_k}f
\end{align*}
is rewritten in multi-index notation. Indeed,
for each $\alpha\in\mathbb{Z}_+^n$ with $|\alpha|=k$, let
\begin{align*}
I_\alpha
:=
\left\{
(j_1,\ldots,j_k)\in\{1,\ldots,n\}^k:
\#\{\ell\in\{1,\ldots,k\}:j_\ell=i\}=\alpha_i
\text{ for every }i\in\{1,\ldots,n\}
\right\}.
\end{align*}
Then
$\# I_\alpha
=
\frac{k!}{\alpha_1!\cdots\alpha_n!}
=
\frac{k!}{\alpha!}.$
Consequently,
\begin{align*}
\sum_{j_1=1}^n\cdots\sum_{j_k=1}^n
\xi_{j_1}\cdots\xi_{j_k}
\partial_{j_1}\cdots\partial_{j_k}f
=
\sum_{\alpha\in\mathbb{Z}_+^n,\,|\alpha|=k}
\frac{k!}{\alpha!}
\partial^\alpha f\xi^\alpha.
\end{align*}
Thus, the
multi-index version of \cite[Lemma 5.7]{HLYY-ineq} should contain this
combinatorial factor, and we use this corrected normalization here.
\end{enumerate}
\end{remark}

The following theorem gives the MS type estimates as $s\to0^+$ and
shows that the factor $s^{\frac1q}$ in \eqref{eq-gns2} has the optimal order.
Detailed notation and definitions appear in Subsection \ref{secMS}.

\begin{theorem}
\label{thmMS}
Let $k\in\mathbb N$, $q\in [1,\infty)$, and $X$ and $Y$ be ball Banach function spaces.
Assume that $X^{\frac{1}{q}}$ and $Y^{\frac{1}{q}}$ are ball Banach function spaces and that the centered ball average operators $\{\mathcal{B}_r\}_{r\in (0,\infty)}$
are uniformly bounded on $X^{\frac{1}{q}}$ and $Y^{\frac{1}{q}}$.
Then, for any $f\in [\bigcup_{\sigma\in (0,k)}{F}^{\sigma,k}_{X,q}]\cap
[\bigcup_{\sigma\in (0,k)}{F}^{\sigma,k}_{Y,q}]$ with compact support,
\begin{align}\label{eq-XX3}
\left(\frac{1}{kq}\left|\mathbb{S}^{n-1}\right|\right)^{\frac{1}{q}}
\left\|f\right\|_{X}
&\le \liminf_{s \to 0^+}s^{\frac{1}{q}}
\left\|
\mathfrak{D}^{s,k}_q (f)
\right\|_{X^{1-s}Y^{s}}\notag\\
&\le \limsup_{s\to 0^+}s^{\frac{1}{q}}
\left\|
\mathfrak{D}^{s,k}_q (f)
\right\|_{X^{1-s}Y^{s}}
\lesssim \left\|f\right\|_X,
\end{align}
where the implicit positive constant is independent of $f$.
\end{theorem}

\begin{remark}
\begin{enumerate}[{\rm(i)}]
\item In the case $X:=Y$, $k:=1$, and $q\in[1,\infty)$,
Theorem \ref{thmMS} improves
\cite[Theorem 2.16(i)]{PYYZ24} within the setting of ball Banach function
spaces. Indeed, the latter assumes that
$p,q\in(0,\infty)$ with $q\le p$, that $X^{\frac1p}$ is a
ball Banach  space, and that the Hardy--Littlewood maximal operator
$M$ is bounded on $(X^{\frac1p})'$. Since $\frac{p}{q}\ge1$,
$
X^{\frac1q}
=
(X^{\frac1p})^{\frac{p}{q}}$
is also a ball Banach function space. Moreover, by
\cite[Lemma 3.11]{DGPYYZ24},
we find that, when $q\in [1,\infty)$, the centered ball average operators are uniformly bounded on $X^{\frac{1}{p}}$. Since $q\le p$, this, together with H\"{o}lder's inequality and the lattice property, yields the corresponding uniform boundedness on $X^{\frac{1}{q}}$. Thus, the assumptions of
\cite[Theorem 2.16(i)]{PYYZ24} imply those of Theorem
\ref{thmMS}.

\item In the Lebesgue space, Maz'ya and Shaposhnikova
\cite{MS02} gave an exact limiting formula as $s\to0^+$ for fractional
Sobolev seminorms. Such an exact formula, however, cannot be extended to
general ball Banach function spaces. Indeed, Pan et al.
\cite[Remarks 2.13(ii) and 2.21 and Example 2.20]{PYYZ24} proved that,
for certain weighted Lebesgue spaces, no positive constant independent of the
function under consideration can make the corresponding Maz'ya--Shaposhnikova
type limit formula hold for all functions.
\end{enumerate}
\end{remark}

Compared with the corresponding BBM formula and MS estimate in a fixed
ball Banach function space, the present setting involves two additional
considerations. First, the target space $X^{1-s}Y^s$ itself varies
along with the smoothness index $s$. We therefore
need to identify its limiting space as $s\to1^-$ and $s\to0^+$;
see Lemma \ref{lem-equa}. Second, the extrapolation argument used
for the BBM upper estimate in \cite{DGPYYZ24,HLYY-ineq} is not suitable
for the present off-diagonal setting because it would not preserve the sharp
range of indices. Instead, we use the pointwise estimates in Theorem
\ref{thm-ps} and their norm consequences in Theorem \ref{thmGNS}.

Unlike Theorem \ref{thmGNS}, which is formulated so as to keep track of
the endpoint behavior as $s\to0^+$ or $s\to1^-$, the smoothness index
$s\in(0,1)$ is fixed in the next theorem. This allows the admissible parameter range
to depend on $s$ itself and leads to the sharp assumption
\begin{align*}
n\left(\frac{1-s}{p_X}+\frac{s}{p_Y}-\frac1q\right)<sk.
\end{align*}
Under this assumption, the Calder\'on--Lozanovski\u{\i} product
$X^{1-s}Y^s$ is not merely an admissible target space. It is the optimal
rearrangement invariant target space.
Hence, the following theorem gives a
characterization of the rearrangement invariant target spaces
for which the corresponding fractional Gagliardo--Nirenberg inequalities hold.

\begin{theorem}\label{thm-fs}
Let $k\in \mathbb{N}$, $s\in(0,1)$, and $q\in[1,\infty)$.
\begin{enumerate}[{\rm(i)}]
\item Let $X$ and $Y$ be ball Banach function spaces whose lower generalized
Boyd indices are respectively $p_X$ and $p_Y$ (see Definition \ref{def-Boyd}).
Assume that $p_X,p_Y\in(1,\infty)$ and
$n(\frac{1-s}{p_X}+\frac{s}{p_Y}-\frac1q)<sk$. Then, for any
$f\in X\cap\dot W^{k,Y}$,
\begin{align}\label{eq-XYss}
\left\|\mathfrak D_q^{s,k}(f)\right\|_{X^{1-s}Y^s}
\lesssim
\|f\|_X^{1-s}\left\|\nabla^k f\right\|_Y^s,
\end{align}
where the implicit positive constant is independent of $f$. The target space
is optimal in the following sense: If $X,Y$, and $B$ are rearrangement
invariant Banach function spaces, then the inequality
\begin{align}\label{eq-OPP-ri}
\left\|\mathfrak D_q^{s,k}(f)\right\|_B
\lesssim
\|f\|_X^{1-s}\left\|\nabla^k f\right\|_Y^s
\end{align}
holds for any $f\in X\cap\dot W^{k,Y}$ if and only if
$X^{1-s}Y^s\hookrightarrow B$.

\item Let $Y$ be a ball Banach function space whose lower generalized Boyd
index is $p_Y$ (see Definition \ref{def-Boyd}). Assume that
$p_Y\in(1,\infty)$ and $n(\frac{s}{p_Y}-\frac1q)<sk$. Then, for any
$f\in\operatorname{BMO}\cap\dot W^{k,Y}$,
\begin{align}\label{eq-BMO-fs}
\left\|\mathfrak D_q^{s,k}(f)\right\|_{Y^{\frac1s}}
\lesssim
\|f\|_{\operatorname{BMO}}^{1-s}
\left\|\nabla^k f\right\|_Y^s,
\end{align}
where the implicit positive constant is independent of $f$. The
target $Y^{\frac1s}$ is optimal. To be precise, if $Y$ and $B$ are rearrangement invariant
Banach function spaces, then the inequality
\begin{align}\label{eq-OPP-bmo}
\left\|\mathfrak D_q^{s,k}(f)\right\|_B
\lesssim
\|f\|_{\operatorname{BMO}}^{1-s}
\left\|\nabla^k f\right\|_Y^s
\end{align}
holds for any $f\in\operatorname{BMO}\cap\dot W^{k,Y}$ if and only if
$Y^{\frac1s}\hookrightarrow B$.
\end{enumerate}
\end{theorem}

\begin{remark}
\begin{enumerate}[{\rm (i)}]
\item When $p_1,p_2\in (1,\infty)$, $X=L^{p_1}$, $Y=L^{p_2}$,
$s\in (0,1)$, $p\in (1,\infty)$ satisfy
$\frac{1}{p}=\frac{1-s}{p_1}+\frac{s}{p_2}$, and $q=p$,
Theorem \ref{thm-fs}{\rm(i)} reduces to the classical fractional
Gagliardo--Nirenberg inequality on Lebesgue spaces; see, for instance,
\cite[Theorem 1.1]{BM18}. Beyond this classical Lebesgue space case, the
off-diagonal estimates in Theorem \ref{thm-fs}, together with the optimal
target assertions, appear to be new for many important function spaces
discussed in Section \ref{sec-app}.

\item
The two ranges of $q$ in Theorem \ref{thm-fs} are sharp in some sense. More
precisely, the estimate in Theorem \ref{thm-fs}{\rm(i)} may fail when
$n(
\frac{1-s}{p_X}+\frac{s}{p_Y}-\frac{1}{q}
)\ge sk$,
and the estimate in Theorem \ref{thm-fs}{\rm(ii)} may fail when
$n(
\frac{s}{p_Y}-\frac{1}{q}
)\ge sk;$
see Section \ref{sec-sharp}.

\item
Let $X=Y^{q_1}$ with $q_1\in [1,\infty)$, let $q\in [1,q_1]$, and let $s\in (0,1)$
satisfy $\frac{1}{q}=\frac{1-s}{q_1}+s$.
Then $X^{1-s}Y^s=Y^q$ and $p_X=q_1 p_Y$. Thus, $\frac{1-s}{p_X}+\frac{s}{p_Y}-\frac{1}{q}=\frac{1}{q}(\frac{1}{p_Y}-1)$. Combining this observation with
Theorem \ref{thm-fs} for $k=1$, we
conclude that, if $p_Y\in (1,\infty),$ then, for any
$f \in Y^{q_1}\cap \dot{W}^{1,Y}$,
\begin{align*}
\left\|
\mathfrak{D}^{s}_q (f)
\right\|_{Y^q}
\lesssim
\left\|f\right\|_{Y^{q_1}}^{1-s}
\left\|\nabla f\right\|_{Y}^{s}
\end{align*}
with the implicit positive constant independent of $f$.
This gives an answer to the question raised in
\cite[Remark 4.14(iii)]{DLYYZ23} in the non-endpoint regime.
\end{enumerate}
\end{remark}

The inequalities in Theorem \ref{thm-fs} are obtained by combining the
pointwise estimates established above with the boundedness of powered
maximal operators on suitable convexifications of $X$ and $Y$. For the
converse direction in the optimality statement,
as in the proof of \cite[Corrollary 2.4]{LRS23}, we first reduce
the optimality problem to a norm
estimate for simple functions; see \eqref{eq-uv} and \eqref{eq-gn1}.
Then, different from the proof of \cite[Corrollary 2.4]{LRS23}, we
introduce a new family of test
functions formed by sums of localized bumps whose supports are
sufficiently far apart; see \eqref{eq-fR}. This separation allows us to
estimate the required norms directly by rearrangement invariance and,
in particular, avoids the \emph{additional assumption} $X\subset Y$ imposed in
\cite[Corollary 2.4]{LRS23}. The same construction applies to the
integer order setting and yields Theorem \ref{thmOP}{\rm(ii)}, which
also removes the \emph{additional assumption} $X\subset Y$ and extends the optimality result
of \cite[Corollary 2.4]{LRS23} from $j=1$ and $k=2$ to arbitrary
integers $1\le j<k$. In particular, taking $j=1$ and $k=2$ and applying
the result to Lorentz spaces gives an answer to
\cite[Question 2.6]{LRS23}.

The remainder of this article is organized as follows. In Section
\ref{sec-poin}, we show the pointwise estimates in Theorem \ref{thm-ps}
and derive the norm inequalities in Theorem \ref{thmGNS}. Section
\ref{sec-asy} is devoted to proving Theorems \ref{thmBBM} and \ref{thmMS}. In Section \ref{secOP}, we show Theorem
\ref{thm-fs}, including the inequalities for each $s$ and the optimality of
the Calder\'on--Lozanovski\u{\i} target spaces.
In Section \ref{sec-sharp},
we prove that the assumptions involving $q$ in
Theorems \ref{thmGNS} and \ref{thm-fs} are sharp.
Finally, in Section \ref{sec-app}, we apply the abstract results to several
specific function spaces. Besides recovering known results in some
classical cases, we obtain several new off-diagonal inequalities, identify
optimal Lorentz target spaces, and show a sharpened BMO endpoint estimate
when the derivative space is an ordinary Lebesgue space.

We conclude this introduction with some notational conventions.
Throughout the article, let $\mathbb{N}:=\{1,2,\ldots\}$ and
$\mathbb{Z}_+ :=\mathbb{N}\cup\{0\}$.
For any $s\in \mathbb{R}$, the \emph{notation $\lfloor s \rfloor$} 
denotes the largest integer not exceeding $s$.
If $E$ is a subset of ${\mathbb{R}^n}$, we denote by
${\bf 1}_E$ its characteristic function.
We use ${\bf 0}$ to denote the origin of ${\mathbb{R}^n}$ and
$\mathbb{S}^{n-1}$ to denote the unit sphere in ${\mathbb{R}^n}$.
For any $x\in {\mathbb{R}^n}$ and $r\in (0,\infty)$, define
$B(x,r):=\{y\in {\mathbb{R}^n}:|x-y|<r\}$.
For any $\lambda\in (0,\infty)$ and any ball or cube $E$ in
$\mathbb{R}^n$, $\lambda E$ denotes the ball or cube with the same center
as $E$ and $\lambda$ times its radius or edge length.
For any measurable function $f$ on $\mathbb{R}^n$, its support is denoted
by $\operatorname{supp}(f)$.
We use $C_{\rm c}^{\infty}$ to denote the space of all infinitely
differentiable functions on $\mathbb{R}^n$ with compact support.
For any $k\in \mathbb{N}$ and $p\in [1,\infty)$, we use
$W^{k,p}_{\rm loc}$ to denote the set of all $f\in L^p_{\rm loc}$ such
that, for any bounded open set $U\subset \mathbb{R}^n$, $f\in W^{k,p}(U)$.
For any $k\in\mathbb{N}$ and any Banach space $X$ of measurable functions
on $\mathbb{R}^n$, if $f\in L^1_{\rm loc}$ is such that $\nabla^k f$
exists, we write $\|\nabla^k f\|_X:=\|\,|\nabla^k f|\,\|_X$.
For any measurable set $E\subset{\mathbb{R}^n}$ with
$|E|\in(0,\infty)$, define
$$\fint_E f(x)\,dx:=\frac{1}{|E|}\int_E f(x)\,dx.$$
In addition, $C$ denotes a positive constant that is independent of the
main parameters involved, but may vary from line to line, whereas
$C_{\alpha,\dots}$ denotes a positive constant depending on the indicated
parameters $\alpha,\dots$.
The notation $f\lesssim g$ means $f\leq Cg$ and, if
$f\lesssim g\lesssim f$, then we write $f\sim g$.
If $f\leq Cg$ and either $g=h$ or $g\leq h$, we write
$f\lesssim g=h$ or $f\lesssim g\leq h$, respectively.
Moreover, the notation $s\to 0^+$ (resp. $s\to 1^-$)
means $s\in (0,1)$ and $s\to 0$ (resp. $s\to 1$).
Finally, throughout the proofs, we retain the notation introduced in the
corresponding theorem or related statement.

\section{Proofs of Theorems \ref{thm-ps} and \ref{thmGNS}}
\label{sec-poin}

We first prove Theorem \ref{thm-ps}.
To this end, we split the domain of integration in the definition of
$\mathfrak{D}^{s,k}_q$ into two parts and estimate them separately as follows.
These estimates are also of independent interest.

\begin{proposition}\label{proPoin}
Let $k\in \mathbb{N}$, $p_1,p_2,q\in[1,\infty)$, and $f\in W^{k,1}_{\rm loc}$.
\begin{enumerate}
\item[{\rm(i)}] If $s_0\in(0,1)$ satisfies
$n(\frac{1-s_0}{p_1}+\frac{s_0}{p_2}-\frac{1}{q})< s_0 k$,
then, for any $r\in (0,\infty)$ and $s\in (s_0,1)$ and for almost every $x\in \mathbb{R}^n$,
\begin{align}\label{eq-Pout}
\left[\int_{|h|\ge r}\frac{|\Delta_h^k f(x)|^q}{|h|^{n+skq}}\,dh\right]^{\frac{1}{q}}
&\lesssim
\frac{r^{k(s_0-s)}}{(s-s_0)^{\frac{1}{q}}}
\left[M^{\sharp}_{p_1}(f)(x)\right]^{1-s_0}\left[M_{p_2}\left(\left|\nabla^k f\right|\right)(x)\right]^{s_0},
\end{align}
where
the implicit positive constant depends only on $n$, $k$, $p_1$, $p_2$, $q$, and $s_0$.

\item[{\rm(ii)}] If $\widetilde{s_0}\in(0,1]$ satisfies
$n(\frac{1-\widetilde{s_0}}{p_1}+\frac{\widetilde{s_0}}{p_2}-\frac{1}{q})< \widetilde{s_0} k$,
then, for any $r\in (0,\infty)$ and $s\in (0,\widetilde{s_0})$ and for almost every $x\in \mathbb{R}^n$,
\begin{align*}
\left[\int_{|h|\le  r}\frac{|\Delta_h^k f(x)|^q}{|h|^{n+s kq}}\,dh\right]^{\frac{1}{q}}
&\lesssim
\frac{r^{k(\widetilde{s_0}-s)}}{(\widetilde{s_0}-s)^{\frac{1}{q}}}
\left[M^{\sharp}_{p_1}(f)(x)\right]^{1-\widetilde{s_0}}\left[M_{p_2}\left(\left|\nabla^k f\right|\right)(x)\right]^{\widetilde{s_0}},
\end{align*}
where the implicit positive constant depends only on $n$, $k$, $p_1$, $p_2$, $q$, and $\widetilde{s_0}$.

\item[{\rm(iii)}] If $p\in [q,\infty)$, then, for any $r\in (0,\infty)$ and $s\in (0,1)$ and for almost every $x\in \mathbb{R}^n$,
\begin{align}\label{eq-Pout2}
\left[\int_{|h|\ge r}\frac{|\Delta_h^k f(x)|^q}{|h|^{n+skq}}\,dh\right]^{\frac{1}{q}}
\lesssim\frac{r^{-sk}}{s^{\frac{1}{q}}}
M_{p}(f-f(x))(x),
\end{align}
where
the implicit positive constant depends only on $n$, $k$, $p$, and $q$.
\end{enumerate}
\end{proposition}

With Proposition \ref{proPoin} in hand,
Theorem \ref{thm-ps} follows easily as shown below.
After that, we will provide the detailed proof of
Proposition \ref{proPoin}.

\begin{proof}[Proof of Theorem \ref{thm-ps}]
Let $f\in W^{k,1}_{\rm loc}$.
We first show (i). By both (i) and (ii) of Proposition \ref{proPoin}, we find that, for any $r\in (0,\infty)$ and $s\in (s_0, \widetilde{s_0})$ and for almost every $x\in \mathbb{R}^n$,
\begin{align*}
\mathfrak{D}^{s,k}_q (f)(x)
&\le
\left[\int_{|h|\ge r}\frac{|\Delta_h^k f(x)|^q}{|h|^{n+skq}}\,dh\right]^{\frac{1}{q}}
+\left[\int_{|h|\le  r}\frac{|\Delta_h^k f(x)|^q}{|h|^{n+s kq}}\,dh\right]^{\frac{1}{q}} \\
& \lesssim
\frac{r^{k(s_0-s)}}{(s-s_0)^{\frac{1}{q}}}
\left[M^{\sharp}_{p_1}(f)(x)\right]^{1-s_0}\left[M_{p_2}\left(\left|\nabla^k f\right|\right)(x)\right]^{s_0}\\
&\quad +
\frac{r^{k(\widetilde{s_0}-s)}}{(\widetilde{s_0}-s)^{\frac{1}{q}}}
\left[M^{\sharp}_{p_1}(f)(x)\right]^{1-\widetilde{s_0}}\left[M_{p_2}\left(\left|\nabla^k f\right|\right)(x)\right]^{\widetilde{s_0}}.
\end{align*}
If one of the two maximal functions vanishes, then the desired estimate
obviously holds. Otherwise, choose $r$ such that the two terms on the right-hand side are equal, i.\,e.,
\begin{align*}
r
=
\left(
\frac{M_{p_1}^{\sharp}(f)(x)}
{M_{p_2}(|\nabla^k f|)(x)}
\right)^{1/k}
\left(
\frac{\widetilde{s}_0-s}{s-s_0}
\right)^{
\frac{1}
{kq(\widetilde{s}_0-s_0)}}.
\end{align*}
Then we have
\begin{align}\label{eq-full-pointwise}
\mathfrak{D}^{s,k}_q (f)(x)
\lesssim
(s-s_0)^{
-\frac{\widetilde{s}_0-s}
{q(\widetilde{s}_0-s_0)}
}
(\widetilde{s}_0-s)^{
-\frac{s-s_0}
{q(\widetilde{s}_0-s_0)}
}
\left[M^{\sharp}_{p_1}(f)(x)\right]^{1-s}\left[M_{p_2}\left(\left|\nabla^k f\right|\right)(x)\right]^{s}.
\end{align}
Note that, for any $s\in (s_0, \widetilde{s_0})$,
\begin{align*}
(s-s_0)^{
-\frac{\widetilde{s}_0-s}
{q(\widetilde{s}_0-s_0)}
}
(\widetilde{s}_0-s)^{
-\frac{s-s_0}
{q(\widetilde{s}_0-s_0)}
}
\lesssim \frac{1}{[(s-s_0)(\widetilde{s_0}-s)]^{\frac{1}{q}}},
\end{align*}
where the implicit positive constant depends only on $q$, $s_0$, and $\widetilde{s_0}$.
From this and \eqref{eq-full-pointwise}, we deduce \eqref{eq-tss} and hence (i).

Next, we prove (ii).
Applying both (ii) with $\widetilde{s_0}:=1$ and (iii) with $p:=p_1$ of Proposition \ref{proPoin},
we conclude that,
for any $r\in (0,\infty)$ and $s\in (0, 1)$ and for almost every $x\in \mathbb{R}^n$,
\begin{align*}
\mathfrak{D}^{s,k}_q (f)(x)
& \lesssim
\frac{r^{-sk}}{s^{\frac{1}{q}}}
M_{p_1}(f-f(x))(x)
+
\frac{r^{k(1-s)}}{(1-s)^{\frac{1}{q}}}
M_{p_2}\left(\left|\nabla^k f\right|\right)(x).
\end{align*}
Similar to the proof of (i), choosing $r$ such that the two terms on the
right-hand side are equal, we obtain
\eqref{eq-tss2}.
This
completes the proof of (ii) and hence Theorem \ref{thm-ps}.
\end{proof}

Now, we turn to show Proposition \ref{proPoin}. We need some lemmas.
For any $s\in \mathbb{Z}_+$,
let $\mathcal{P}_{s} $ denote the set of
all polynomials of degree not greater than $s$ on
$\mathbb{R}^n$.
For any
$s\in \mathbb{Z}_+$, any ball or cube $\Omega$ in $\mathbb{R}^n$,
and any $f\in L^{1}_{\rm loc} $,
let
$P^{(s)}_{\Omega}(f)$ denote the \emph{minimizing polynomial} in
$\mathcal{P}_{s} $ on $\Omega$ such that, for any
$\alpha\in\mathbb{Z}_+^n$ with $|\alpha|\leq s$,
\begin{align*}
\int_{\Omega}\left[f(x)-P^{(s)}_{\Omega}(f)(x)\right]x^{\alpha}\,dx=0.
\end{align*}
Then, from the definition,
we can easily deduce that, for any $s\in\mathbb{Z}_+$ and any
ball or cube $\Omega$ in $\mathbb{R}^n$,
the minimizing polynomial is unique and has the following
\emph{minimizing property} in terms of norms: for any $f\in L^{1}_{\rm loc}$,
\begin{align*}
E_{s}(f,\Omega):= \inf_{P\in \mathcal{P}_{s}}
\fint_{\Omega}|f(x)-P (x)|\,dx
\sim
\fint_{\Omega}\left|f(x)-P^{(s)}_\Omega(f)(x)\right|\,dx,
\end{align*}
where the positive equivalence constants depend only on $n$ and $s$
(see also \cite[p.\,83]{TG80} and \cite[Lemma 4.1]{Lu95}).
Combining this and \cite[Lemma 4.7]{HLYY-ineq},
we obtain the following variant of \cite[Lemma 4.7]{HLYY-ineq}, which will be used in the
proof of Theorem \ref{thm-ps}.

\begin{lemma}\label{lem:Poinca}
Let $k\in \mathbb{N}$ and $f\in L^1_{\rm loc} $.
Then there exists a positive constant $C_{n,k}$, depending only on
$n$ and $k$, such that,
for almost every $x\in\mathbb{R}^n$ and for any $r\in(0,\infty)$
and any ball $B_1\subset B:=B(x,r)\subset 3B_1$,
\begin{align*}
\left|f(x)-P^{(k-1)}_{B_1}(f)(x)\right|
\le C_{n,k}
\sum_{j\in\mathbb{Z}_+}E_{k-1}\left(f,B\left(x,2^{-j}r\right)\right).
\end{align*}
\end{lemma}

On the other hand,
for any $\alpha\in\{0,\frac{1}{3},\frac{2}{3}\}^n$,
recall that the \emph{shifted
dyadic grid} $\mathcal{D}^\alpha$ is defined by setting
\begin{align*}
\mathcal{D}^\alpha:=\left\{2^j\left[m+[0,1)^n+(-1)^j\alpha\right]:
j\in\mathbb{Z},\ m\in\mathbb{Z}^n\right\}.
\end{align*}
Let $\mathbb{D}:=\bigcup_{\alpha\in\{0,\frac13,\frac23\}^n}\mathcal{D}^\alpha$.
The following properties
of shifted dyadic grids are also needed
(see, for instance, \cite[p.\,479]{MTT02}).

\begin{lemma}\label{lemCub}
\begin{enumerate}[{\rm(i)}]
\item
For any $Q,P\in\mathcal{D}^\alpha$ with
$\alpha\in\{0,\frac{1}{3},\frac{2}{3}\}^n$,
$Q\cap P\in\{\emptyset,Q,P\}$.
\item
For any cube $Q$ in $\mathbb{R}^n$, there exists
$P\in\mathbb{D}$ such that $Q\subset P$ and $l(P)\in(\frac32l(Q),3l(Q)] $.
\end{enumerate}
\end{lemma}

Based on this and Lemma \ref{lem:Poinca},
we obtain the following telescope type estimate
of difference operators,
which plays a key role in the proof of Proposition \ref{proPoin}.

\begin{lemma}\label{lem-point}
Let $k\in \mathbb{N}$, $q\in [1,\infty)$, and $\epsilon\in (0,1)$.
Then there exist positive constants $c_0$ and $c_1$, depending only on
$n$ and $k$, and a positive constant $C_{n,k,q}$, depending only on
$n$, $k$, and $q$, such that, for any $f\in W^{k,1}_{\rm loc}$,
$h\in\mathbb{R}^n\setminus\{{\bf0}\}$, and for almost every
$x\in \mathbb{R}^n$,
\begin{align}\label{lem-point-e1}
|\Delta_h^kf(x)|^q
\le
\frac{C_{n,k,q}}{\epsilon^q}
\sum_{i=0}^{k}\sum_{j\in\mathbb{Z}_+}2^{j\epsilon q}
\sum_{\genfrac{}{}{0pt}{}{Q\in\mathbb{D}}{2^jl(Q)
\in[c_0|h|,c_1|h|]}}
[E_{k-1}(f,Q)]^q F_i(x,h),
\end{align}
where, for each fixed $j\in\mathbb{Z}_+$ and $Q\in\mathbb{D}$ in the
last summation,
\begin{align*}
F_i(x,h):=\begin{cases}
{\bf 1}_{E_0}(x,x+h) & \text{if } i=0,\\
{\bf 1}_{E_i}(x,x+ih) & \text{if } i\in \{1,\dots,k\}
\end{cases}
\end{align*}
with
\begin{align*}
E_i:=\begin{cases}
Q\times 2^jQ & \text{if } i=0,\\
(2^jQ)\times Q & \text{if } i\in \{1,\dots,k\}.
\end{cases}
\end{align*}
\end{lemma}

\begin{proof}
Let $f\in W^{k,1}_{\rm loc}$ and
$h\in\mathbb{R}^n\setminus\{{\bf0}\}$.
Since $\Delta_h^kP=0$ for any $P\in\mathcal{P}_{k-1}$, from
\eqref{eq-sho-hd}, it follows that, for any $x\in\mathbb{R}^n$,
\begin{align*}
|\Delta_h^kf(x)|^q
&=
\left|\Delta_h^k\left(f-P^{(k-1)}_{B_1}(f)\right)(x)\right|^q
\lesssim\sum_{i=0}^k
\left|f(x+ih)-P^{(k-1)}_{B_1}(f)(x+ih)\right|^q,
\end{align*}
where $B_1:=B(x+\frac{kh}{2},k|h|).$
Fix $i\in\{0,\ldots,k\}$. Observe that
$B_1\subset B(x+ih,2k|h|)\subset 3B_1$.
Applying Lemma \ref{lem:Poinca} with the above $B_1$ and
$B:=B(x+ih,2k|h|)$, we find that, for almost every
$x\in\mathbb{R}^n$,
\begin{align}\label{eq-pigeonhole1}
\left|f(x+ih)-P^{(k-1)}_{B_1}(f)(x+ih)\right|
\le
C_{n,k}\sum_{j\in\mathbb{Z}_+}
E_{k-1}\left(f,B\left(x+ih,2^{-j+1}k|h|\right)\right),
\end{align}
where $C_{n,k}$ is a positive constant depending only on $n$ and $k$.
Fix $x\in\mathbb{R}^n$ such that \eqref{eq-pigeonhole1} holds and let
$C_1:=\frac{1-2^{-\epsilon}}{C_{n,k}}.$
Then, from \eqref{eq-pigeonhole1}, we infer that
\begin{align*}
&C_1\sum_{j\in\mathbb{Z}_+}2^{-j\epsilon}
\left|f(x+ih)-P^{(k-1)}_{B_1}(f)(x+ih)\right| \\
&\quad =
C_{n,k}
\left|f(x+ih)-P^{(k-1)}_{B_1}(f)(x+ih)\right|\le
\sum_{j\in\mathbb{Z}_+}
E_{k-1}\left(f,B\left(x+ih,2^{-j+1}k|h|\right)\right).
\end{align*}
Thus, there exists $j_{x,h}\in\mathbb{Z}_+$ such that
\begin{align*}
&C_1 2^{-j_{x,h}\epsilon}
\left|f(x+ih)-P^{(k-1)}_{B_1}(f)(x+ih)\right|\le
E_{k-1}\left(f,B\left(x+ih,2^{-j_{x,h}+1}k|h|\right)\right).
\end{align*}
Since $\epsilon\in(0,1)$, it follows that $(1-2^{-\epsilon})^{-1}\lesssim
\epsilon^{-1}$. Hence,
\begin{align}\label{eq-point-eps}
&\left|f(x+ih)-P^{(k-1)}_{B_1}(f)(x+ih)\right|^q
\lesssim
\epsilon^{-q}2^{j_{x,h}\epsilon q}
\left[
E_{k-1}\left(f,B\left(x+ih,2^{-j_{x,h}+1}k|h|\right)\right)
\right]^q .
\end{align}
In addition, applying Lemma \ref{lemCub}, we find that there exist
$C_{n}\in(1,\infty)$, depending only on $n$, and
$Q_{x,h}\in \mathbb{D}$ such that
\begin{align*}
B\left(x+ih,2^{-j_{x,h}+1}k|h|\right)
\subset Q_{x,h}\subset
B\left(x+ih,2^{-j_{x,h}+1}C_{n}k|h|\right).
\end{align*}
Then
$2^{j_{x,h}}l(Q_{x,h})\in[c_0|h|,c_1|h|],$
where $c_0:=2k\omega_n^{\frac1n}$ and
$c_1:=2C_{n}k\omega_n^{\frac1n}$ with $\omega_n$ denoting the
volume of the unit ball in $\mathbb{R}^n$. These further imply
\begin{align}\label{pig-e1}
(x,h)\in
\begin{cases}
\displaystyle
Q_{x,h}\times\left[\left(2^{j_{x,h}}Q_{x,h}\right)-\{x\}\right]
& \text{if } i=0,\\
\displaystyle
\left(2^{j_{x,h}}Q_{x,h}\right)\times
\frac{Q_{x,h}-\{x\}}{i}
& \text{if } i\in\{1,\dots,k\}.
\end{cases}
\end{align}
Moreover, from the comparability of $Q_{x,h}$ and
$B(x+ih,2^{-j_{x,h}+1}k|h|)$, it follows that
\begin{align*}
E_{k-1}\left(f,B\left(x+ih,2^{-j_{x,h}+1}k|h|\right)\right)
\lesssim E_{k-1}(f,Q_{x,h}).
\end{align*}
This, together with \eqref{eq-point-eps} and \eqref{pig-e1}, further implies that
\begin{align*}
&\left|f(x+ih)-P^{(k-1)}_{B_1}(f)(x+ih)\right|^q
\lesssim
\epsilon^{-q}
\sum_{j\in\mathbb{Z}_+}2^{j\epsilon q}
\sum_{\genfrac{}{}{0pt}{}{Q\in\mathbb{D}}{2^jl(Q)
\in[c_0|h|,c_1|h|]}}
[E_{k-1}(f,Q)]^qF_i(x,h).
\end{align*}
Combining this and the preceding estimate for $|\Delta_h^kf(x)|^q$,
we further obtain \eqref{lem-point-e1}, which completes the proof of
Lemma \ref{lem-point}.
\end{proof}

We shall use the following elementary consequence of H\"older's inequality.
For the reader's convenience, we give its proof.

\begin{lemma}\label{lem-discrete}
Let $I$ be a finite index set with $\#I=N$, and let
$\{a_\nu\}_{\nu\in I}$ and $\{b_\nu\}_{\nu\in I}$ be nonnegative
sequences. If $\theta_1,\theta_2\in (0,\infty)$ and
$\theta:=\theta_1+\theta_2$, then
\begin{align*}
\sum_{\nu\in I}a_\nu^{\theta_2}b_\nu^{\theta_1}
\le
N^{(1-\theta)_+}
\left(\sum_{\nu\in I}a_\nu\right)^{\theta_2}
\left(\sum_{\nu\in I}b_\nu\right)^{\theta_1}.
\end{align*}
\end{lemma}

\begin{proof}
By H\"older's inequality and the elementary estimate
$
\sum_{\nu\in I} c_\nu^\alpha
\le
N^{(1-\alpha)_+}
\left(\sum_{\nu\in I} c_\nu\right)^\alpha
$
for any $\alpha\in(0,\infty)$ and any nonnegative sequence
$\{c_\nu\}_{\nu\in I}$, we obtain
\begin{align*}
\sum_{\nu\in I}a_\nu^{\theta_2}b_\nu^{\theta_1}
&\le
\left(\sum_{\nu\in I}a_\nu^\theta\right)^{\frac{\theta_2}{\theta}}
\left(\sum_{\nu\in I}b_\nu^\theta\right)^{\frac{\theta_1}{\theta}} \\
&\le
N^{(1-\theta)_+\frac{\theta_2}{\theta}}
N^{(1-\theta)_+\frac{\theta_1}{\theta}}
\left(\sum_{\nu\in I}a_\nu\right)^{\theta_2}
\left(\sum_{\nu\in I}b_\nu\right)^{\theta_1} \\
&=
N^{(1-\theta)_+}
\left(\sum_{\nu\in I}a_\nu\right)^{\theta_2}
\left(\sum_{\nu\in I}b_\nu\right)^{\theta_1}.
\end{align*}
This completes the proof of Lemma \ref{lem-discrete}.
\end{proof}

Next, we are ready to prove Proposition \ref{proPoin}.

\begin{proof}[Proof of Proposition \ref{proPoin}]
Let $f\in W^{k,1}_{\rm loc}$ and $r\in (0,\infty)$.
We first show (i). Define
\begin{align*}
\delta_0:=s_0k-n\left(\frac{1-s_0}{p_1}+\frac{s_0}{p_2}-\frac1q\right)_+.
\end{align*}
Then $\delta_0>0$. Let $\epsilon:=\frac{\min\{1,\delta_0\}}{2}$.
Applying Lemma \ref{lem-point}, we find that, for almost every $x\in \mathbb{R}^n$,
\begin{align}\label{eq-p00}
\int_{|h|\ge r}\frac{|\Delta_h^k f(x)|^q}{|h|^{n+skq}}\,dh
&\lesssim\sum_{i=0}^{k}\sum_{j\in\mathbb{Z}_+}2^{j\epsilon q}
\sum_{\genfrac{}{}{0pt}{}{Q\in\mathbb{D}}{l(Q)\ge c_02^{-j}r}}
\left|2^jQ\right|^{-\frac{n+skq}{n}}
[E_{k-1}(f,Q)]^q \int_{\mathbb{R}^n}F_i(x,h)\,dh\notag\\
&=: \sum_{i=0}^{k} I_{i,r}(x),
\end{align}
where $c_0$ and $F_i$ are as in Lemma \ref{lem-point}. For any cube $Q$, define
\begin{align*}
T_m(f,Q):=
\begin{cases}
\displaystyle
\fint_Q |f(x)-f_Q|\,dx & \text{if } m=0,\\
\displaystyle
\fint_Q |\nabla^m f(x)|\,dx & \text{if } m=k.
\end{cases}
\end{align*}
From the Poincar\'{e} inequality (see, for instance,
\cite[Theorem 13.27]{Leo17}), we infer that, for any cube $Q$ and
$m\in\{0,k\}$,
\begin{align}\label{eq-TT}
E_{k-1}(f,Q)\lesssim |Q|^{\frac{m}{n}}T_m(f,Q).
\end{align}
Next, we fix $x\in\mathbb{R}^n$ satisfying \eqref{eq-p00} and estimate
$I_{i,r}(x)$ by considering the following two cases for $i$.

\emph{Case 1: $i=0$.} In this case, by the definition of $F_0$, we obtain
\begin{align}\label{eq-cas1}
I_{0,r}(x)
&\lesssim
\sum_{j\in\mathbb{Z}_+}2^{j\epsilon q}
\sum_{\genfrac{}{}{0pt}{}{Q\in\mathbb{D}}{l(Q)\ge c_02^{-j}r}}
\left|2^jQ\right|^{-\frac{skq}{n}}
[E_{k-1}(f,Q)]^q {\bf 1}_Q(x).
\end{align}
From \eqref{eq-TT}, it follows that, for any cube $Q$,
\begin{align*}
&\left|2^jQ\right|^{-\frac{skq}{n}}[E_{k-1}(f,Q)]^q {\bf 1}_Q(x)\\
&\quad\lesssim
\left|2^jQ\right|^{-\frac{skq}{n}}
[T_0(f,Q)]^{q(1-s_0)}
\left[|Q|^{\frac{k}{n}}T_k(f,Q)\right]^{qs_0}{\bf 1}_Q(x)\\
&\quad\le
2^{-js_0kq}\left|2^jQ\right|^{\frac{kq(s_0-s)}{n}}
[M^{\sharp}(f)(x)]^{q(1-s_0)}
\left[M\left(|\nabla^k f|\right)(x)\right]^{qs_0}{\bf 1}_Q(x).
\end{align*}
This, combined with \eqref{eq-cas1}, implies that
\begin{align}\label{eq-cas11}
I_{0,r}(x)
&\lesssim [M^{\sharp}(f)(x)]^{q(1-s_0)}
\left[M\left(|\nabla^k f|\right)(x)\right]^{qs_0}\notag\\
&\quad\times
\sum_{j\in\mathbb{Z}_+}2^{jq(\epsilon-s_0k)}
\sum_{\genfrac{}{}{0pt}{}{Q\in\mathbb{D}}{l(Q)\ge c_02^{-j}r}}
\left|2^jQ\right|^{\frac{kq(s_0-s)}{n}}{\bf 1}_Q(x).
\end{align}
For any $\alpha\in\{0,\frac13,\frac23\}^n$ and $j\in\mathbb{Z}_+$,
let $Q_\alpha\in\mathcal{D}^\alpha$ be the minimal cube containing $x$
such that $l(Q_\alpha)\ge c_02^{-j}r$. Hence, for any
$P\in\mathcal{D}^\alpha$ containing $x$ with $l(P)\ge c_02^{-j}r$, it
holds that $P\supset Q_\alpha$ and $|P|=2^{\nu n}|Q_\alpha|$ for some
$\nu\in\mathbb{Z}_+$. Therefore, from these and the assumption
$s\in(s_0,1)$, we deduce that
\begin{align}\label{eq-ddq}
&\sum_{\genfrac{}{}{0pt}{}{Q\in\mathbb{D}}{l(Q)\ge c_02^{-j}r}}
\left|2^jQ\right|^{\frac{kq(s_0-s)}{n}}{\bf 1}_Q(x)\notag\\
&\quad=
\sum_{\alpha\in\{0,\frac13,\frac23\}^n}
\sum_{\genfrac{}{}{0pt}{}{P\in\mathcal{D}^\alpha}{P\supset Q_\alpha}}
\left|2^jP\right|^{\frac{kq(s_0-s)}{n}}=
\sum_{\alpha\in\{0,\frac13,\frac23\}^n}
\sum_{\nu=0}^{\infty}
2^{\nu kq(s_0-s)}\left|2^jQ_\alpha\right|^{\frac{kq(s_0-s)}{n}}\notag\\
&\quad\lesssim \frac{r^{kq(s_0-s)}}{1-2^{kq(s_0-s)}}
\lesssim \frac{r^{kq(s_0-s)}}{s-s_0},
\end{align}
where the implicit positive constants depend only on $k$, $q$, and
$s_0$. Combining this, \eqref{eq-cas1}, \eqref{eq-cas11}, and
H\"older's inequality, we conclude that
\begin{align}\label{eq-Ic1}
I_{0,r}(x)
&\lesssim
\frac{r^{kq(s_0-s)}}{s-s_0}
[M^{\sharp}(f)(x)]^{q(1-s_0)}
\left[M\left(|\nabla^k f|\right)(x)\right]^{qs_0}\notag\\
&\le
\frac{r^{kq(s_0-s)}}{s-s_0}
[M_{p_1}^{\sharp}(f)(x)]^{q(1-s_0)}
\left[M_{p_2}\left(|\nabla^k f|\right)(x)\right]^{qs_0}.
\end{align}
This completes the estimation of $I_{0,r}(x)$.

\emph{Case 2: $i\in\{1,\ldots,k\}$.} In this case, by the definition of
$F_i$, we find that
\begin{align}\label{eq-III}
I_{i,r}(x)
&\lesssim
\sum_{j\in\mathbb{Z}_+}2^{j\epsilon q}
\sum_{\genfrac{}{}{0pt}{}{Q\in\mathbb{D}}{l(Q)\ge c_02^{-j}r}}
\left|2^jQ\right|^{-\frac{n+skq}{n}}|Q|
[E_{k-1}(f,Q)]^q {\bf 1}_{2^jQ}(x).
\end{align}
Using \eqref{eq-TT} and H\"older's inequality, we obtain, for any cube
$Q$,
\begin{align*}
&\left|2^jQ\right|^{-\frac{n+skq}{n}}|Q|[E_{k-1}(f,Q)]^q\\
&\quad\le
\left|2^jQ\right|^{-\frac{skq+n}{n}}|Q|
[T_0(f,Q)]^{q(1-s_0)}
\left[|Q|^{\frac{k}{n}}T_k(f,Q)\right]^{qs_0}\\
&\quad=
2^{-js_0kq-jn}\left|2^jQ\right|^{\frac{kq(s_0-s)}{n}}
[T_0(f,Q)]^{q(1-s_0)}[T_k(f,Q)]^{qs_0}\\
&\quad\le
2^{-js_0kq-jn}\left|2^jQ\right|^{\frac{kq(s_0-s)}{n}}
\left[\fint_Q |f(z)-f_Q|^{p_1}\,dz\right]^{\frac{q(1-s_0)}{p_1}}
\left[\fint_Q |\nabla^k f(z)|^{p_2}\,dz\right]^{\frac{qs_0}{p_2}}.
\end{align*}
This, combined with \eqref{eq-III}, further implies that
\begin{align}\label{eq-Hqj}
I_{i,r}(x)
&\le
\sum_{j\in\mathbb{Z}_+}2^{jq(\epsilon-s_0k)-jn}
\sum_{\genfrac{}{}{0pt}{}{Q\in\mathbb{D}}{l(Q)\ge c_02^{-j}r}}
\left|2^jQ\right|^{\frac{kq(s_0-s)}{n}}\notag\\
&\quad\times\left[\fint_Q |f(z)-f_Q|^{p_1}\,dz\right]^{\frac{q(1-s_0)}{p_1}}
\left[\fint_Q |\nabla^k f(z)|^{p_2}\,dz\right]^{\frac{qs_0}{p_2}}
{\bf 1}_{2^jQ}(x).
\end{align}
This estimate cannot be converted directly into a
product of maximal functions as in Case 1. Indeed, the point $x$ is only known to
belong to the dilated cube $2^jQ$, whereas the averages in
\eqref{eq-Hqj} are taken over the smaller cube $Q$. Thus, the averaging
cube may not contain $x$. To overcome this
possible mismatch, we dominate each dilated
cube $2^jQ$ using the shifted dyadic
grids again. We now introduce this domination procedure.
For any $\nu\in\mathbb{Z}$, let
$\mathbb{D}_\nu:=\{Q\in\mathbb{D}:l(Q)=2^{-\nu}\}$ and, for any cube
$Q$, let
\begin{align}\label{eq-dom}
\operatorname{Dom}(Q):=
\left\{P\in\mathbb{D}:P\supset Q,\ l(P)\in\left(\frac32l(Q),3l(Q)\right]\right\}.
\end{align}
By Lemma \ref{lemCub}(ii), we conclude that $\operatorname{Dom}(Q)$ is
nonempty for any cube $Q$ and, if $l(Q)=2^\nu$ for some
$\nu\in\mathbb{Z}$, then $l(J)=2l(Q)$ for any
$J\in\operatorname{Dom}(Q)$. From these, H\"older's inequality, and
Tonelli's theorem, we deduce that, for any $j\in\mathbb{Z}_+$,
\begin{align}\label{eq-Qvol}
&\sum_{\genfrac{}{}{0pt}{}{Q\in\mathbb{D}}{l(Q)\ge c_02^{-j}r}}
\left|2^jQ\right|^{\frac{kq(s_0-s)}{n}}
\left[\fint_Q |f(z)-f_Q|^{p_1}\,dz\right]^{\frac{q(1-s_0)}{p_1}}
\left[\fint_Q |\nabla^k f(z)|^{p_2}\,dz\right]^{\frac{qs_0}{p_2}}
{\bf 1}_{2^jQ}(x)\notag\\
&\quad\lesssim
\sum_{\genfrac{}{}{0pt}{}{Q\in\mathbb{D}}{l(Q)\ge c_02^{-j}r}}
\sum_{J\in\operatorname{Dom}(2^jQ)}
|J|^{\frac{kq(s_0-s)}{n}}
\left[\fint_Q |f(z)-f_Q|^{p_1}\,dz\right]^{\frac{q(1-s_0)}{p_1}}
\left[\fint_Q |\nabla^k f(z)|^{p_2}\,dz\right]^{\frac{qs_0}{p_2}}
{\bf1}_J(x)\notag\\
&\quad=
\sum_{\genfrac{}{}{0pt}{}{J\in\mathbb{D}}{l(J)\ge 2c_0r}}
\sum_{\genfrac{}{}{0pt}{}{Q\in\mathbb{D}}{J\in\operatorname{Dom}(2^jQ)}}
|J|^{\frac{kq(s_0-s)}{n}}
\left[\fint_Q |f(z)-f_Q|^{p_1}\,dz\right]^{\frac{q(1-s_0)}{p_1}}
\left[\fint_Q |\nabla^k f(z)|^{p_2}\,dz\right]^{\frac{qs_0}{p_2}}
{\bf1}_J(x).
\end{align}
Using H\"older's inequality, we conclude that, for any cube $Q$ with
$J\in\operatorname{Dom}(2^jQ)$,
\begin{align}\label{eq-hldd}
\int_Q |f(z)-f_Q|^{p_1}\,dz
&\lesssim
\int_Q |f(z)-f_J|^{p_1}\,dz+|Q||f_J-f_Q|^{p_1}
\lesssim \int_Q |f(z)-f_J|^{p_1}\,dz.
\end{align}
Observe that, for any $J\in\mathbb{D}$ and $j\in\mathbb{Z}_+$,
$$
\#\left\{Q\in\mathbb{D}: J\in\operatorname{Dom}(2^jQ)\right\}\sim 2^{jn}.
$$
From this, \eqref{eq-hldd}, and Lemma \ref{lem-discrete}, it follows
that, for any $J\in\mathbb{D}$ and $j\in\mathbb{Z}_+$,
\begin{align}\label{eq-tt}
&\sum_{\genfrac{}{}{0pt}{}{Q\in\mathbb{D}}{J\in\operatorname{Dom}(2^jQ)}}
\left[\fint_Q |f(z)-f_Q|^{p_1}\,dz\right]^{\frac{q(1-s_0)}{p_1}}
\left[\fint_Q |\nabla^k f(z)|^{p_2}\,dz\right]^{\frac{qs_0}{p_2}}\notag\\
&\quad\lesssim
2^{jn(1-\theta)_+}
\left[\sum_{\genfrac{}{}{0pt}{}{Q\in\mathbb{D}}{J\in\operatorname{Dom}(2^jQ)}}
\fint_Q |f(z)-f_J|^{p_1}\,dz\right]^{\frac{q(1-s_0)}{p_1}}
\left[\sum_{\genfrac{}{}{0pt}{}{Q\in\mathbb{D}}{J\in\operatorname{Dom}(2^jQ)}}
\fint_Q |\nabla^k f(z)|^{p_2}\,dz\right]^{\frac{qs_0}{p_2}}\notag\\
&\quad\le
2^{jn(1-\theta)_+}2^{jn\theta}
\left[\fint_J |f(z)-f_J|^{p_1}\,dz\right]^{\frac{q(1-s_0)}{p_1}}
\left[\fint_J |\nabla^k f(z)|^{p_2}\,dz\right]^{\frac{qs_0}{p_2}},
\end{align}
where $\theta=q(\frac{1-s_0}{p_1}+\frac{s_0}{p_2})$. Note that
$$
(1-\theta)_+ +\theta-1= (\theta-1)_+
= q\left(\frac{1-s_0}{p_1}+\frac{s_0}{p_2}-\frac1q\right)_+ .
$$
Using \eqref{eq-Hqj}, \eqref{eq-Qvol}, \eqref{eq-ddq}, \eqref{eq-tt},
and the choice of $\epsilon$, we find that
\begin{align}\label{eq-Ic2}
I_{i,r}(x)
&\lesssim
\sum_{j\in\mathbb{Z}_+}2^{jq(\epsilon-s_0k)+jn(\theta-1)_+}
\frac{r^{kq(s_0-s)}}{s-s_0}
[M_{p_1}^{\sharp}(f)(x)]^{q(1-s_0)}
\left[M_{p_2}\left(|\nabla^k f|\right)(x)\right]^{qs_0}\notag\\
&\lesssim
\frac{r^{kq(s_0-s)}}{s-s_0}
[M_{p_1}^{\sharp}(f)(x)]^{q(1-s_0)}
\left[M_{p_2}\left(|\nabla^k f|\right)(x)\right]^{qs_0}.
\end{align}
This then completes the estimate of
$I_{i,r}$ in this case.

Combining \eqref{eq-p00}, \eqref{eq-Ic1}, and \eqref{eq-Ic2}, we obtain
\eqref{eq-Pout}, which completes the proof of (i).

Next, we consider (ii). Choose
\begin{align*}
\widetilde\delta_0:=\widetilde{s_0}k-
n\left(\frac{1-\widetilde{s_0}}{p_1}+\frac{\widetilde{s_0}}{p_2}-\frac1q\right)_+>0
\end{align*}
and define $\widetilde{\epsilon}:=\frac{\min\{1,\widetilde\delta_0\}}2$. Using Lemma
\ref{lem-point}, we obtain, for almost every $x\in\mathbb{R}^n$,
\begin{align}\label{eq-p2}
\int_{|h|\le r}\frac{|\Delta_h^k f(x)|^q}{|h|^{n+skq}}\,dh
&\lesssim\sum_{i=0}^{k}\sum_{j\in\mathbb{Z}_+}2^{j\widetilde{\epsilon} q}
\sum_{\genfrac{}{}{0pt}{}{Q\in\mathbb{D}}{l(Q)\le c_12^{-j}r}}
\left|2^jQ\right|^{-\frac{n+skq}{n}}
[E_{k-1}(f,Q)]^q \int_{\mathbb{R}^n}F_i(x,h)\,dh,
\end{align}
where $c_1$ and $F_i$ are as in Lemma \ref{lem-point}. Observe that
\eqref{eq-p2} differs from \eqref{eq-p00} only
in the edge-length
restriction for the cubes in the innermost sum. Repeating the proof of
(i), with $s_0$ replaced by $\widetilde{s_0}$ and with the corresponding
geometric series summed in the opposite direction, we obtain the factor
$(\widetilde{s_0}-s)^{-1}$ and hence (ii).

Next, we prove (iii). Note that
$\sum_{i=0}^{k}(-1)^{k-i}\binom{k}{i}=0,$
which, together with \eqref{eq-sho-hd}, implies that
\begin{align*}
\Delta_h^k f(\cdot)
=
\sum_{i=1}^{k}(-1)^{k-i}\binom{k}{i}[f(\cdot+ih)-f(\cdot)].
\end{align*}
From this and Minkowski's inequality, it follows that, for almost every
$x\in\mathbb{R}^n$,
\begin{align*}
\left[\int_{|h|\ge r}\frac{|\Delta_h^k f(x)|^q}{|h|^{n+skq}}\,dh\right]^{\frac{1}{q}}
&\lesssim
\sum_{i=1}^{k}
\left[\sum_{j=0}^{\infty}
\int_{2^{j}r\le |h|\le 2^{j+1}r}
\frac{|f(x+ih)-f(x)|^q}{|h|^{n+skq}}\,dh\right]^{\frac{1}{q}}\\
&\lesssim
\sum_{i=1}^{k}
\left[\sum_{j=0}^{\infty}2^{-jskq}r^{-skq}
\fint_{|h|<2^{j+1}r}|f(x+ih)-f(x)|^q\,dh\right]^{\frac{1}{q}}\\
&\lesssim
s^{-\frac{1}{q}}r^{-sk}M_q(f-f(x))(x).
\end{align*}
This implies \eqref{eq-Pout2} and hence (iii),
which then completes the proof of Proposition \ref{proPoin}.
\end{proof}

Next, we give the proof of Theorem \ref{thmGNS}.
To this end, we need some necessary definitions.
The definition
of ball Banach function spaces is as follows,
which was introduced in \cite[Definition 2.1]{SHYY17}.

\begin{definition}\label{def-X}
A Banach space $X \subset L^0 $,
equipped with a norm
$\|\cdot\|_{X }$ which makes sense for all functions in
$L^0 $,
is called a \emph{ball Banach function space}
(for short, \emph{{\rm BBF} space})
if $X $ satisfies that
\begin{enumerate}[{\rm (i)}]
\item for any $f\in L^0 $,
if $\|f\|_{X } =0$,
then $f=0$ almost
everywhere;
\item if $f,g\in L^0 $ satisfy that $|g|\leq
|f|$
almost everywhere, then
$\|g\|_{X } \leq \|f\|_{X }$;
\item if a sequence
$\{f_m\}_{m\in{\mathbb{N}}}$ in $L^0 $
satisfies that
$0\leq f_m \uparrow f$ almost everywhere as $m\to\infty$,
then $\|f_m\|_{X } \uparrow \|f\|_{X }$
as $m\to\infty$;
\item for any ball $B$ in $\mathbb{R}^n$,
${\bf 1}_{B}\in {X }$;

\item for any ball $B$ in $\mathbb{R}^n$,
there exists a positive constant $C_{(B)}$,
depending on $B$,
such that, for any $f\in {X}$,
$$
\int_{B}|f(x)|\,dx \leq C_{(B)}\|f\|_{{X }}.
$$
\end{enumerate}
\end{definition}

\begin{remark}\label{rem-bbf}
Concerning the definition of BBF spaces,
we have the following remarks.
\begin{enumerate}[{\rm (i)}]

\item  From \cite[Proposition 1.2.36]{LYH22}, we deduce that
both (ii) and (iii) of Definition \ref{def-X}
imply that any ball Banach function space is complete.

\item If we replace any ball
$B$ in $\mathbb{R}^n$ by any bounded measurable set
$S$ in $\mathbb{R}^n$, then, by \cite[Remark 2.5(ii)]{YHYY22ams},
we obtain an equivalent formulation of BBF spaces.
Moreover, if we replace any ball $B$
by any measurable $E$ with $|E|<\infty$, then
we obtain the definition of \emph{Banach function spaces},
which was originally
introduced by Bennett and Sharpley in
\cite[Chapter 1, Definitions 1.1 and 1.3]{BS88}.
Using their definitions,
we easily find that
a Banach function space is always a ball Banach function space.
However,
the converse is not necessarily true
(see, for instance, \cite[p.\,9]{SHYY17}).

\item If we replace (iv)
by the \emph{saturation property}
that,
for any measurable set
$E$ in $\mathbb{R}^n$
with $|E|\in (0,\infty)$, there exists a measurable
set $F\subset E$ with $|F|\in (0,\infty)$ satisfying that
${\bf 1}_{F}\in X $,
then we obtain the definition of Banach function spaces in the terminology
of Lorist and Nieraeth \cite[p.\,251]{LN24im}. Moreover,
by \cite[Proposition 2.5]{ZYY23ccm}
(see also \cite[Proposition 4.21]{Nie23}),
we conclude that, if the normed vector space $X $ under
consideration satisfies the additional assumption that
the Hardy--Littlewood maximal
operator $\mathcal{M}$ is weakly bounded on one of its convexifications,
then the definition of Banach function spaces in \cite{LN24im}
coincides with the definition of ball
Banach function spaces. Thus, under this additional assumption, working with
ball Banach function spaces in the sense of Definition \ref{def-X} or Banach
function spaces in the sense of \cite{LN24im} would
yield exactly the same results.
\end{enumerate}
\end{remark}

\begin{definition}
Let $X$ be a {\rm BBF} space.
For any $p\in (0,\infty)$, the
\emph{$p$-convexification}
$X^p $ of $X $ is defined by setting
$X^p :=\{f\in L^0 :
|f|^p \in X \}$
and is equipped with the \emph{quasi-norm}
$\left\|f\right\|_{X^p }
:=\left\|\,|f|^p\right\|_{X }^{\frac{1}{p}}$
for any $f\in X^p $.
\end{definition}

The following definition of  Calder\'on--Lozanovski\u{\i} spaces can be found in \cite[p.\,123]{Cal64}.
For more related results, we refer to \cite{BS88,Nie23}.

\begin{definition}\label{def-CL}
Let $X$ and $Y$ be {\rm BBF} spaces. For any $s\in (0,1)$,
the \emph{Calder\'on--Lozanovski\u{\i} space} $X^{1-s}Y^{s}$
is defined by setting
\begin{align*}
X^{ 1-s}Y^{s}:=
\left\{h\in L^0: |h|\le |f|^{ 1-s}|g|^{ s}\ \text{for some}\
f\in X\ \text{and}\ g\in Y\right\},
\end{align*}
equipped with the \emph{norm}
\begin{align}\label{eq-nnn}
\|h\|_{X^{ 1-s}Y^{ s}} :=
\inf\left\{\|f\|^{ 1-s}_{X}\|g\|^{s}_{Y}: |h|\le |f|^{ 1-s}|g|^{s}\right\}.
\end{align}
\end{definition}

\begin{remark}\label{rem-cp}
Let $X$ and $Y$ be BBF spaces.
\begin{enumerate}[{\rm (i)}]
\item By \cite[p.\,123]{Cal64} and the definition of {\rm BBF} spaces, we conclude that,
for any $ s\in (0,1)$, $X^{ 1-s}Y^{ s}$ is also a {\rm BBF} space;
see also \cite[Section 2.3]{Nie26}.

\item From \eqref{eq-nnn}, we infer that, for any $s\in (0,1)$
and $f\in X\cap Y$,
\begin{align}\label{eq-UP}
\|f\|_{X^{1-s}Y^s} \leq  \|f\|^{1-s}_X\|f\|^{s}_Y.
\end{align}
\item By the definition, we easily find that, for any $s\in(0,1)$,
$(L^\infty)^{1-s}Y^s=Y^{\frac1s}$.
\end{enumerate}
\end{remark}

Now, we present the definition
of ball Banach Sobolev spaces, which was introduced
in \cite[Definition 2.6]{DGPYYZ24}.

\begin{definition}
Let $k\in {\mathbb{N}}$ and
$X $ be a {\rm BBF} space.
\begin{enumerate}[{\rm (i)}]
\item The \emph{inhomogeneous ball Banach Sobolev space}
$W^{k,X}$ is defined to be the set
of all $f\in X$ such that,
for any multi-index
$\alpha\in \mathbb{Z}_{+}^n$ with
$|\alpha|\le k$,
the $\alpha$-th weak partial derivative
$\partial^{\alpha} f$ of $f$ exists and
$\partial^{\alpha} f \in X$.
Moreover, the \emph{norm} $\|\cdot\|_{{W}^{k,X}}$ of $W^{k,X}$
is defined by setting, for any $f\in W^{k,X}$,
$\|f\|_{{W}^{k,X}}
:=\sum_{\alpha\in \mathbb{Z}_{+}^n,\, |\alpha|\le k}
\|\partial^{\alpha} f\|_{X}.$

\item 	The \emph{homogeneous ball Banach Sobolev space}
$\dot{W}^{k,X}$ is defined to be the set
of all $f\in L^{1}_{\rm loc}$ such that,
for any multi-index
$\alpha\in \mathbb{Z}_{+}^n$ with
$|\alpha|= k$,
the $\alpha$-th weak partial derivative
$\partial^{\alpha} f$ of $f$ exists and
$\partial^{\alpha} f \in X$.
Moreover, the \emph{seminorm} $\|\cdot\|_{\dot{W}^{k,X}}$
of $\dot{W}^{k,X}$
is defined by setting, for any $f\in \dot{W}^{k,X}$,
$\|f\|_{\dot{W}^{k,X}}
:=\sum_{\alpha\in \mathbb{Z}_{+}^n,\, |\alpha|= k}
\|\partial^{\alpha} f\|_{X}.$
\end{enumerate}
\end{definition}

To describe the boundedness of the Hardy--Littlewood
maximal operator $M$ on the convexification of {\rm BBF} spaces,
we need the following definition of the
lower generalized Boyd index, which was introduced in
\cite[Definition 2.1]{Ho12} and \cite[Definition 2.17]{SHYY17}.

\begin{definition}\label{def-Boyd}
Let $X$ be a {\rm BBF} space. The \emph{lower
generalized Boyd index} of $X$ is defined by setting
\begin{align*}
p_{X}:=\sup\left(\left\{p\in (1,\infty):
X^{\frac{1}{p}} \text{ is a {\rm BBF} space and } M\ \text{is bounded on $X^{\frac{1}{p}}$} \right\}\cup \{1\}\right).
\end{align*}
\end{definition}

\begin{remark}\label{rem-Boyd}
Let $X$ be a {\rm BBF} space with the
lower generalized Boyd index $p_X$.
\begin{enumerate}
\item[{\rm(i)}] As pointed out in \cite[p.\,98]{Ho12},
if $X$ is a rearrangement invariant Banach function space
(see, for instance, \cite[p.\,59, Definition 4.1]{BS88} for its definition),
then $p_X$ is exactly the reciprocal of the
well-known lower Boyd index of rearrangement invariant spaces
(see, for instance, \cite[p.\,149, Definition 5.12]{BS88}
for its definition).
\item[{\rm(ii)}] If $p_{X}\in (1,\infty)$, then
it holds that, for any $m\in \mathbb{Z}_+$,
\begin{align}\label{eq-wkpx}
\dot{W}^{m,X} \subset\bigcap_{1 \le p < p_X} W^{m,p}_{\rm loc}.
\end{align}
Indeed, in this case,  for any $p\in [1,p_X)$,
$X^{\frac{1}{p}}$ is a {\rm BBF} space.
By this and the definition
of {\rm BBF} spaces, we find that, for any $m\in \mathbb{Z}_+$ and $f\in \dot{W}^{m,X}$,
$|\nabla^m f|^p \in X^{\frac{1}{p}}\subset L^1_{\rm loc}$,
which, together with \cite[Sect. 1.1.2, Theorem]{ma2011},
further implies $f\in W^{m,p}_{\rm loc}$.
This completes the proof of \eqref{eq-wkpx}.
\end{enumerate}
\end{remark}

Using $M_{p}^{\sharp} (f)$, we can give the definition of
BMO spaces. For any $p\in [1,\infty)$, the \emph{${{\rm BMO}}_{p}$ space} is the set of
all $f\in L^p_{\rm loc}$ such that
\begin{align*}
\|f\|_{{\rm BMO}_{p}}:=
\left\| M_{p}^{\sharp}(f) \right\|_{L^\infty}<\infty.
\end{align*}
It is well known that,
for any $p\in [1,\infty)$,
${\rm BMO}_{p}={\rm BMO}_{1}=:{\rm BMO}$
(see, for instance, \cite[Corollary 6.12]{Duo01}).
Next, we turn to show Theorem \ref{thmGNS}.

\begin{proof}[Proof of Theorem \ref{thmGNS}]
Let $f\in \dot{W}^{k,Y}$.
Note that we always assume $p_Y>1$ in the present theorem.
From Remark \ref{rem-Boyd}(ii), it follows that
$f\in\bigcap_{1\le p<p_Y}W^{k,p}_{\rm loc}$.

We first prove (i). Assume further that $f\in X$.
For any $t\in [0,1]$, let
\begin{align*}
F(t):=n\left(\frac{1-t}{p_X}+\frac{t}{p_Y}-\frac{1}{q}\right)-tk.
\end{align*}
By the assumption $n(\frac{1}{p_Y}-\frac{1}{q})<k$, we have
$F(1)<0$. From this and the continuity of $F$, we infer that there
exists $s_0\in(0,1)$ such that $F(s_0)<0$. Since this inequality is
strict, we can choose $p_1\in(1,p_X)$ and $p_2\in(1,p_Y)$ satisfying
\begin{align*}
n\left(\frac{1-s_0}{p_1}+\frac{s_0}{p_2}-\frac{1}{q}\right)<s_0k
\quad\text{and}\quad
n\left(\frac{1}{p_2}-\frac{1}{q}\right)<k.
\end{align*}
Applying this and Theorem \ref{thm-ps}(i) with $\widetilde{s_0}:=1$,
we find that, for any $s\in(s_0,1)$ and almost every
$x\in\mathbb{R}^n$,
\begin{align*}
\mathfrak{D}^{s,k}_{q}(f)(x)
\lesssim [(s-s_0)(1-s)]^{-\frac{1}{q}}
\left[M^{\sharp}_{p_1}(f)(x)\right]^{1-s}
\left[M_{p_2}\left(\left|\nabla^k f\right|\right)(x)\right]^s.
\end{align*}
This, together with the definition of $X^{1-s}Y^s$,
$M_{p_1}^{\sharp}(f)\lesssim M_{p_1}(f)$, and the boundedness of $M$ on
both $X^{\frac{1}{p_1}}$ and $Y^{\frac{1}{p_2}}$, further implies that
\begin{align*}
\left\|\mathfrak{D}^{s,k}_{q}(f)\right\|_{X^{1-s}Y^s}
&\lesssim [(s-s_0)(1-s)]^{-\frac{1}{q}}
\left\|M_{p_1}^{\sharp}(f)\right\|_X^{1-s}
\left\|M_{p_2}\left(\left|\nabla^k f\right|\right)\right\|_Y^s \\
&\lesssim [(s-s_0)(1-s)]^{-\frac{1}{q}}
\left\|f\right\|_X^{1-s}
\left\|\nabla^k f\right\|_Y^s.
\end{align*}
Thus, \eqref{eq-gns1} holds.

Next, we show \eqref{eq-gns2}. Assume that $1\le q<p_X$.
Choose $p_1\in(q,p_X)$ and $p_2\in(1,p_Y)$ satisfying that
$n(\frac{1}{p_2}-\frac{1}{q})<k$. Since
$|f(x)|\le M_{p_1}(f)(x)$ for almost every $x\in\mathbb{R}^n$, we have
\begin{align*}
M_{p_1}(f-f(x))(x)
\le M_{p_1}(f)(x)+|f(x)|
\le 2M_{p_1}(f)(x).
\end{align*}
From this and Theorem \ref{thm-ps}(ii), we infer that, for any
$s\in(0,1)$ and almost every $x\in\mathbb{R}^n$,
\begin{align*}
\mathfrak{D}^{s,k}_q(f)(x)
\lesssim \frac{1}{[s(1-s)]^{\frac{1}{q}}}
\left[M_{p_1}(f)(x)\right]^{1-s}
\left[M_{p_2}\left(\left|\nabla^k f\right|\right)(x)\right]^s.
\end{align*}
This, together with the definition of $X^{1-s}Y^s$ and the boundedness
of $M$ on both $X^{\frac{1}{p_1}}$ and $Y^{\frac{1}{p_2}}$, further
implies \eqref{eq-gns2} and hence (i).

Finally, we prove (ii). Assume further that $f\in \operatorname{BMO}$.
Repeating the choice of the parameters in the proof of \eqref{eq-gns1}
with $X=L^\infty$ and $p_X=\infty$, we can choose
$s_0\in(0,1)$, $p_1\in(1,\infty)$, and $p_2\in(1,p_Y)$ such that
\begin{align*}
n\left(\frac{1-s_0}{p_1}+\frac{s_0}{p_2}-\frac{1}{q}\right)<s_0k
\quad\text{and}\quad
n\left(\frac{1}{p_2}-\frac{1}{q}\right)<k.
\end{align*}
Applying Theorem \ref{thm-ps}(i) with $\widetilde{s_0}:=1$, we obtain,
for any $s\in(s_0,1)$ and almost every $x\in\mathbb{R}^n$,
\begin{align*}
\mathfrak{D}^{s,k}_q(f)(x)
\lesssim [(s-s_0)(1-s)]^{-\frac{1}{q}}
\left[M_{p_1}^{\sharp}(f)(x)\right]^{1-s}
\left[M_{p_2}\left(\left|\nabla^k f\right|\right)(x)\right]^s.
\end{align*}
This, combined with Remark \ref{rem-cp}(iii), the boundedness of $M$ on
$Y^{\frac{1}{p_2}}$, and the fact
$\|M_{p_1}^{\sharp}(\cdot)\|_{L^\infty}\sim
\|\cdot\|_{\operatorname{BMO}}$, further implies that
\begin{align*}
\left\|\mathfrak{D}^{s,k}_q(f)\right\|_{Y^{\frac{1}{s}}}
&=\left\|\mathfrak{D}^{s,k}_q(f)\right\|_{(L^\infty)^{1-s}Y^s}
\lesssim [(s-s_0)(1-s)]^{-\frac{1}{q}}
\left\|f\right\|_{\operatorname{BMO}}^{1-s}
\left\|\nabla^k f\right\|_Y^s.
\end{align*}
This shows (ii), which completes the proof of Theorem \ref{thmGNS}.
\end{proof}

Finally, as pointed out in Remark \ref{rem-1.2}(ii),
we prove that estimate \eqref{eq-tss2} cannot hold uniformly in
$s\in(0,1)$ with $M_{p_1}(f-f(x))(x)$ replaced by
$M_{p_1}^{\sharp}(f)(x)$.
\begin{proposition}\label{pro-counter}
Let $k\in \mathbb{N}$ and $p_1,p_2,q\in[1,\infty)$. There is no
constant $C\in(0,\infty)$, depending only on $n$, $k$, $p_1$, $p_2$, and
$q$, such that, for any $s\in(0,1)$ and $f\in W^{k,1}_{\mathrm{loc}}$
and for almost every $x\in\mathbb{R}^n$,
\begin{align}\label{eq-false1}
\mathfrak{D}^{s,k}_q(f)(x)
\leq
\frac{C}{[s(1-s)]^{\frac1q}}
\left[M_{p_1}^{\sharp}(f)(x)\right]^{1-s}
\left[M_{p_2}(|\nabla^k f|)(x)\right]^s.
\end{align}
\end{proposition}

\begin{proof}
For any $x\in\mathbb{R}^n$, let
$g(x):=\log(1+|x|^2)$. Then $g\in C^\infty\cap\operatorname{BMO}$ and
$g\in W^{k,1}_{\mathrm{loc}}$. Moreover, $|\nabla^k g|\in L^\infty$.
Thus, for any $x\in\mathbb{R}^n$,
\begin{align}\label{eq-counter-max}
M_{p_1}^{\sharp}(g)(x)\lesssim 1
\quad\text{and}\quad
M_{p_2}(|\nabla^k g|)(x)\lesssim 1.
\end{align}

We next estimate $\mathfrak{D}^{s,k}_q(g)$. Fix
$x\in B({\bf0},1)$ and $\omega\in\mathbb{S}^{n-1}$. For any
$r\geq2$ and $j\in\{1,\ldots,k\}$, we have
\begin{align*}
\left|g(x+jr\omega)-2\log r\right|
&=
\left|
\log\left(r^{-2}+\left|j\omega+\frac{x}{r}\right|^2\right)
\right|
\lesssim 1.
\end{align*}
From this and the boundedness of $g$ on $B({\bf0},1)$, it follows that
\begin{align*}
\left|
\Delta_{r\omega}^k g(x)
-2\log r\sum_{j=1}^k(-1)^{k-j}\binom{k}{j}
\right|
\lesssim 1.
\end{align*}
Noting that $\sum_{j=1}^k(-1)^{k-j}\binom{k}{j}=(-1)^{k+1}\ne0$,
we conclude that there exists $R\in[2,\infty)$,
depending only on $k$, such that, for any
$r\ge R$,
$|\Delta^k_{r\omega}g(x)|\gtrsim \log r.$
Using this and polar coordinates, we find that, for any
$s\in(0,1)$ such that $e^{\frac1{skq}}>R$
and for any $x\in B({\bf0},1)$,
\begin{align}\label{eq-counter-D-lower}
\left[\mathfrak{D}^{s,k}_q(g)(x)\right]^q
&\geq
\int_{\mathbb{S}^{n-1}}\int_R^\infty
\frac{|\Delta^k_{r\omega}g(x)|^q}{r^{1+skq}}\,dr\,d\mathcal{H}^{n-1}(\omega)
\notag\\
&\gtrsim
\int_R^\infty\frac{(\log r)^q}{r^{1+skq}}\,dr
\geq
\int_{e^{\frac1{skq}}}^{e^{\frac2{skq}}}
\frac{(\log r)^q}{r^{1+skq}}\,dr \notag\\
&\gtrsim
s^{-q}
\int_{e^{\frac1{skq}}}^{e^{\frac2{skq}}}\frac{dr}{r^{1+skq}}
\gtrsim s^{-q-1}.
\end{align}
On the other hand, from \eqref{eq-false1} and
\eqref{eq-counter-max}, we deduce that, for any $s\in(0,1)$ and almost every
$x\in B({\bf0},1)$,
$\mathfrak{D}^{s,k}_q(g)(x)
\lesssim [s(1-s)]^{-\frac1q}.$
Combining this and \eqref{eq-counter-D-lower}, we
obtain, for all sufficiently small $s\in(0,1)$,
$s^{-1-\frac1q}\lesssim s^{-\frac1q},$
which is a contradiction as $s\to0^+$. This completes the proof of
Proposition \ref{pro-counter}.
\end{proof}

\begin{remark}\label{rem-HLYY-assumption}
We give the detailed comparison between the assumptions used in the
present article and those in \cite[Theorems 4.1 and 5.5]{HLYY-ineq}.
Consider the following two assumptions:
\begin{enumerate}[{\rm(a)}]
\item $p_X\in(1,\infty)$ and
$n(\frac1{p_X}-\frac1q)<k$;

\item there exists some $p_0\in(1,\infty)$ such that
$n(\frac1{p_0}-\frac1q)<k$, $X^{\frac1{p_0}}$ is a {\rm BBF} space,
and the Hardy--Littlewood maximal operator $M$ is bounded on
$(X^{\frac1{p_0}})'$.
\end{enumerate}
Assumption {\rm(b)} is used in the corresponding non-endpoint results
of \cite[Theorems 4.1 and 5.5]{HLYY-ineq}. We show that {\rm(b)}
implies {\rm(a)}.

Indeed, assume that {\rm(b)} holds. By
\cite[Lemma 3.10]{HLYY-ineq}, for any measurable function $h$,
\begin{align}\label{eq-rubio-representation}
\|h\|_X
\sim
\sup_{\|g\|_{(X^{1/p_0})'}\le1}
\|h\|_{L^{p_0}_{R_g}},
\end{align}
where $R_g\in A_1$ and $[R_g]_{A_1}$ is bounded uniformly in $g$ by
a positive constant depending only on
$\|M\|_{(X^{1/p_0})'\to(X^{1/p_0})'}$.

Let $p\in(1,p_0)$. From \eqref{eq-rubio-representation} and the
definition of the convexification, it follows that
\begin{align}\label{eq-rubio-convexification}
\|h\|_{X^{\frac1p}}
&=
\big\||h|^{\frac1p}\big\|_X^p
\sim
\sup_{\|g\|_{(X^{1/p_0})'}\le1}
\big\||h|^{\frac1p}\big\|_{L^{p_0}_{R_g}}^p =
\sup_{\|g\|_{(X^{1/p_0})'}\le1}
\|h\|_{L^{\frac{p_0}{p}}_{R_g}}.
\end{align}
Since $\frac{p_0}{p}>1$, the right-hand side of
\eqref{eq-rubio-convexification} is a supremum of Banach function
norms. Hence, $X^{\frac1p}$ is a {\rm BBF} space, up to an equivalent
norm.

Moreover, since $R_g\in A_1\subset A_{\frac{p_0}{p}}$ uniformly in
$g$, the boundedness of $M$ on weighted Lebesgue spaces and
\eqref{eq-rubio-convexification} imply that
\begin{align*}
\|M(h)\|_{X^{\frac1p}}
&\lesssim
\sup_{\|g\|_{(X^{1/p_0})'}\le1}
\|M(h)\|_{L^{\frac{p_0}{p}}_{R_g}}\lesssim
\sup_{\|g\|_{(X^{1/p_0})'}\le1}
\|h\|_{L^{\frac{p_0}{p}}_{R_g}}
\lesssim
\|h\|_{X^{\frac1p}}.
\end{align*}
Therefore, $p\le p_X$ for every $p\in(1,p_0)$. By the arbitrariness
of $p$, we obtain $p_0\le p_X$, and hence
$n(\frac1{p_X}-\frac1q)
\le n(\frac1{p_0}-\frac1q)<k$.
Thus, assumption {\rm(b)} implies assumption {\rm(a)}.

Consequently, in the global non-endpoint diagonal setting, both
Theorem \ref{thmGNS} and Theorem \ref{thmBBM} require a weaker assumption than the corresponding one in
\cite[Theorems 4.1 and 5.5]{HLYY-ineq}.
\end{remark}

\section{Proofs of Theorems \ref{thmBBM} and \ref{thmMS}}\label{sec-asy}

In this section, we give the proofs of Theorems \ref{thmBBM} and \ref{thmMS}, respectively, in Subsection \ref{secBBM} and Subsection \ref{secMS}.

\subsection{Proof of Theorem \ref{thmBBM}}\label{secBBM}

In this subsection, we prove Theorem \ref{thmBBM}. To this end, we first recall the notion of absolutely continuous norm of a {\rm BBF} space.

\begin{definition}\label{def-ac}
A  {\rm BBF} space $X$ is said to have an \emph{absolutely
continuous norm} if, for any $f\in X$ and any sequence
$\{E_j\}_{j\in\mathbb{N}}$ of measurable sets in $\mathbb{R}^n$
satisfying that ${\bf 1}_{E_j}\to0$ almost everywhere as $j\to\infty$,
one has $\|f{\bf 1}_{E_j}\|_X\to0$ as $j\to\infty$ (see, for instance,
\cite[Definition 3.2]{WYY20} or \cite[Definition 3.1]{BS88}).
\end{definition}

As in the proof of Theorem \ref{thm-ps}, we establish this BBM formula
by splitting the domain of integration in the definition of
$\mathfrak{D}^{s,k}_q(f)$ into inner and outer parts. For any
$s\in(0,1)$, $k\in\mathbb{N}$, $q\in[1,\infty)$, $f\in L^0$, and
$x\in\mathbb{R}^n$, define
\begin{align*}
\mathfrak{D}^{s,k}_{q,{\rm in}}(f)(x)&:=
\left[
\int_{|h|\le1}\frac{|\Delta_h^k f(x)|^q}{|h|^{n+skq}}\,dh
\right]^{\frac1q}\quad\text{and}\quad
\mathfrak{D}^{s,k}_{q,{\rm out}}(f)(x):=
\left[
\int_{|h|\ge1}\frac{|\Delta_h^k f(x)|^q}{|h|^{n+skq}}\,dh
\right]^{\frac1q}.
\end{align*}

First, we deal with the inner part as follows.

\begin{proposition}\label{proIN}
Let all the notation and assumptions be the same as in Theorem \ref{thmBBM}.
Then, for any $f\in X\cap W^{k,Y}$,
\begin{align}\label{eq-in1}
\lim_{s\to1^-}(1-s)^{\frac1q}
\left\|\mathfrak{D}^{s,k}_{q,{\rm in}}(f)\right\|_{X^{1-s}Y^s}
=(kq)^{-\frac1q}\left\|\mathbb{D}^k_q(f)\right\|_Y.
\end{align}
\end{proposition}

To show Proposition \ref{proIN}, we need two lemmas.

\begin{lemma}\label{lem-equa}
Let $X$ and $Y$ be {\rm BBF} spaces. For any $f\in X\cap Y$,
$\lim_{s\to1^-}\|f\|_{X^{1-s}Y^s}=\|f\|_Y.$
\end{lemma}

\begin{proof}
Let $f\in X\cap Y$. By \eqref{eq-UP}, we find that
\begin{align}\label{eq-UP1}
\limsup_{s\to1^-}\|f\|_{X^{1-s}Y^s}\le\|f\|_Y.
\end{align}
Next, we prove the converse inequality
\begin{align}\label{eq-LP}
\liminf_{s\to1^-}\|f\|_{X^{1-s}Y^s}\ge\|f\|_Y.
\end{align}
For this purpose, we need the concept of the \emph{associate space} (also
called the \emph{K\"othe dual}) of a {\rm BBF} space $X$, which is defined
to be the set of all $f\in L^0$ such that
\begin{align*}
\|f\|_{X'}:=
\sup_{\genfrac{}{}{0pt}{}{g\in X}{\|g\|_X=1}}\|fg\|_{L^1}<\infty
\end{align*}
(see, for instance, \cite[Chapter 1, Section 2]{BS88} or
\cite[p.\,9]{SHYY17}). From the Lozanovski\u{\i} duality theorem (see, for
instance, \cite[Theorem 7.2]{CNS03}), we infer that
$(X^{1-s}Y^s)'=(X')^{1-s}(Y')^s$ with the same norm.
Consequently, the norm in $X^{1-s}Y^s$ admits the following dual
representation:
\begin{align*}
\|f\|_{X^{1-s}Y^s}
=
\sup\left\{
\int_{\mathbb{R}^n}|f(x)w(x)|\,dx:
\|w\|_{(X')^{1-s}(Y')^s}\le1
\right\}.
\end{align*}
Note that,
for any fixed $w\in X'\cap Y'$ with $w\not\equiv0$,
by Remark \ref{rem-cp}(ii) we have
$\|w\|_{(X')^{1-s}(Y')^s}\le \|w\|_{X'}^{1-s}\|w\|_{Y'}^s$. Thus, for any
$s\in(0,1)$,
\begin{align*}
\|f\|_{X^{1-s}Y^s}
\ge
\frac{\int_{\mathbb{R}^n}|f(x)w(x)|\,dx}
{\|w\|_{(X')^{1-s}(Y')^s}}
\ge
\frac{\int_{\mathbb{R}^n}|f(x)w(x)|\,dx}
{\|w\|_{X'}^{1-s}\|w\|_{Y'}^s},
\end{align*}
which implies
\begin{align*}
\liminf_{s\to1^-}\|f\|_{X^{1-s}Y^s}
\ge
\frac{\int_{\mathbb{R}^n}|f(x)w(x)|\,dx}{\|w\|_{Y'}}.
\end{align*}
Taking the supremum over such $w$, we obtain
\begin{align}\label{up}
\liminf_{s\to1^-}\|f\|_{X^{1-s}Y^s}
\ge
\sup_{\genfrac{}{}{0pt}{}{w\in X'\cap Y'}{w\ne0}}
\frac{\int_{\mathbb{R}^n}|f(x)w(x)|\,dx}{\|w\|_{Y'}}.
\end{align}
Using the Lorentz--Luxemburg lemma for {\rm BBF} spaces (see, for
instance, \cite[Lemma 2.6]{ZYYW21}), we find that
\begin{align*}
\sup_{\genfrac{}{}{0pt}{}{w\in Y'}{w\ne0}}
\frac{\int_{\mathbb{R}^n}|f(x)w(x)|\,dx}{\|w\|_{Y'}}
=\|f\|_Y.
\end{align*}
We next show that the same supremum can be taken over $X'\cap Y'$. For
any $w\in Y'$ and $N,R\in\mathbb{N}$, let
$w_{N,R}:=\min\{|w|,N\}{\bf 1}_{B({\bf 0},R)}$. Then
$w_{N,R}\in Y'$. Moreover, since $w_{N,R}$ is bounded and has bounded
support, from the definition of {\rm BBF}
spaces, we deduce that $w_{N,R}\in X'$. Thus, $w_{N,R}\in X'\cap Y'$. Letting
first $N\to\infty$ and then $R\to\infty$, and using the monotone
convergence theorem, we conclude that
\begin{align*}
\sup_{\genfrac{}{}{0pt}{}{w\in X'\cap Y'}{w\ne0}}
\frac{\int_{\mathbb{R}^n}|f(x)w(x)|\,dx}{\|w\|_{Y'}}
=\|f\|_Y.
\end{align*}
This, together with \eqref{up}, further implies \eqref{eq-LP}.
Combining this and \eqref{eq-UP1}, we conclude that
\begin{align*}
\lim_{s\to1^-}\|f\|_{X^{1-s}Y^s}=\|f\|_Y,
\end{align*}
which completes the proof of Lemma \ref{lem-equa}.
\end{proof}

\begin{lemma}\label{lem-con}
Let $k\in\mathbb{N}$, $q\in[1,\infty)$, and $X,Y$ be {\rm BBF} spaces.
Assume that $Y$ has an absolutely continuous norm. Then, for any
$f\in C_{\rm c}^\infty$,
\begin{align}\label{eq-ccw}
\lim_{s\to1^-}(1-s)^{\frac1q}
\left\|\mathfrak{D}^{s,k}_{q,{\rm in}}(f)\right\|_{X^{1-s}Y^s}
=(kq)^{-\frac1q}\left\|\mathbb{D}^k_q(f)\right\|_Y.
\end{align}
\end{lemma}

\begin{proof}
Let $f\in C_{\rm c}^\infty$ and, for any $s\in(0,1)$ and
$x\in\mathbb{R}^n$, define
$g_s(x):=(1-s)^{\frac1q}\mathfrak{D}^{s,k}_{q,{\rm in}}(f)(x)$ and
$G(x):=(kq)^{-\frac1q}\mathbb{D}^k_q(f)(x)$.
From \cite[Lemma 5.7]{HLYY-ineq}
and
its corrected
multi-index form discussed in Remark \ref{rem-imp}(iii), we infer that, for any $x\in\mathbb{R}^n$,
\begin{align}\label{pointwise-limit}
\lim_{s\to1^-}g_s(x)=G(x).
\end{align}
Since $f\in C_{\rm c}^\infty$, there exists a compact set
$K\subset\mathbb{R}^n$ such that, for any $|h|\le1$,
$\operatorname{supp}(\Delta_h^k f)\subset K$. Moreover,
$|\Delta_h^k f(x)|\lesssim |h|^k$ for any $x\in\mathbb{R}^n$ and
$|h|\le1$. Thus, for any $s\in(0,1)$,
\begin{align*}
g_s(x)
&\lesssim
(1-s)^{\frac1q}{\bf 1}_{K}(x)
\left[\int_0^1 r^{k(1-s)q-1}\,dr\right]^{\frac1q}
\lesssim {\bf 1}_{K}(x).
\end{align*}
By this, the dominated convergence theorem for $Y$ (which is
equivalent to that $Y$ has an absolutely continuous norm; see
\cite[p.\,260]{LN24im}), and \eqref{pointwise-limit}, we conclude that
\begin{align}\label{Y-norm-conv}
\lim_{s\to1^-}\|g_s-G\|_Y=0.
\end{align}
Note that, for any $s\in(0,1)$, $|g_s-G|\lesssim {\bf 1}_K$ with the
implicit positive constant independent of $s$. From this, the definition
of $X^{1-s}Y^s$, and \eqref{Y-norm-conv}, it follows that
\begin{align*}
\left|\|g_s\|_{X^{1-s}Y^s}-\|G\|_{X^{1-s}Y^s}\right|
&\le \|g_s-G\|_{X^{1-s}Y^s}
\le \|g_s-G\|_X^{1-s}\|g_s-G\|_Y^s\notag\\[-2mm]
&\lesssim \|{\bf 1}_K\|_X^{1-s}\|g_s-G\|_Y^s\to0
\end{align*}
as $s\to1^-$. Applying this and Lemma \ref{lem-equa}, we obtain
\eqref{eq-ccw}. This completes the proof of Lemma \ref{lem-con}.
\end{proof}

For any given $r\in(0,\infty)$, the \emph{centered ball average operator}
$\mathcal{B}_r$ is defined by setting, for any $f\in L^1_{\rm loc}$ and
$x\in\mathbb{R}^n$,
\begin{align*}
\mathcal{B}_r(f)(x):=
\frac1{|B(x,r)|}\int_{B(x,r)}|f(y)|\,dy.
\end{align*}
Note that, if $1\le p<p_X<\infty$ holds, then $M$ is bounded on
$X^{\frac1p}$, which further implies that the centered ball average
operators $\{\mathcal{B}_r\}_{r\in(0,\infty)}$ are uniformly bounded on
$X^{\frac1p}$.

We next turn to prove Proposition \ref{proIN}.

\begin{proof}[Proof of Proposition \ref{proIN}]
Let $f\in X\cap W^{k,Y}$. Since $p_X,p_Y\in(1,\infty)$, it holds that
$\{\mathcal{B}_r\}_{r\in(0,\infty)}$ are uniformly bounded on both $X$ and
$Y$. From this, the assumption that both $X$ and
$Y$ have absolutely continuous norms, and the proof of
\cite[Corollary 3.10]{DGPYYZ24} (see also
\cite[Theorem 2.16]{HLYYZ-bsvy}), we infer that, for any fixed
$\epsilon\in(0,1)$, there exists $g\in C_{\rm c}^\infty$ such that
\begin{align}\label{eq-ddense}
\|f-g\|_X<\epsilon
\quad\text{and}\quad
\|f-g\|_{W^{k,Y}}<\epsilon.
\end{align}
Note that, in \eqref{eq-ddense},
we use the same approximating function
in both spaces. This is possible because the approximation in \cite[Corollary 3.10]{DGPYYZ24}
is constructed by truncation followed by mollification and is therefore independent of
the underlying function space.
Then, for any $s\in(0,1)$,
\begin{align}\label{eq-iiii}
&\left|
(1-s)^{\frac1q}
\left\|\mathfrak{D}^{s,k}_{q,{\rm in}}(f)\right\|_{X^{1-s}Y^s}
-(kq)^{-\frac1q}\left\|\mathbb{D}^k_q(f)\right\|_Y
\right|
\le {\rm I}(s)+{\rm II}(s)+{\rm III},
\end{align}
where
\begin{align*}
{\rm I}(s)&:=(1-s)^{\frac1q}
\left|
\left\|\mathfrak{D}^{s,k}_{q,{\rm in}}(f)\right\|_{X^{1-s}Y^s}
-
\left\|\mathfrak{D}^{s,k}_{q,{\rm in}}(g)\right\|_{X^{1-s}Y^s}
\right|,\\
{\rm II}(s)&:=
\left|
(1-s)^{\frac1q}
\left\|\mathfrak{D}^{s,k}_{q,{\rm in}}(g)\right\|_{X^{1-s}Y^s}
-(kq)^{-\frac1q}\left\|\mathbb{D}^k_q(g)\right\|_Y
\right|,
\end{align*}
and
\begin{align*}
{\rm III}:=(kq)^{-\frac1q}
\left|
\left\|\mathbb{D}^k_q(g)\right\|_Y
-
\left\|\mathbb{D}^k_q(f)\right\|_Y
\right|.
\end{align*}
Noting that $\mathfrak{D}^{s,k}_{q,{\rm in}}(f-g)
\le \mathfrak{D}^{s,k}_q(f-g)$, using the assumption
$n(\frac1{p_Y}-\frac1q)<k$ and Theorem \ref{thmGNS}(i), we conclude
that there exist $s_0\in(0,1)$ and a positive constant $C_0$, independent
of $s$ and $f$, such that, for any $s\in(s_0,1)$,
\begin{align*}
{\rm I}(s)
\le
C_0(s-s_0)^{-\frac1q}
\|f-g\|_X^{1-s}
\|\nabla^k(f-g)\|_Y^s.
\end{align*}
This, combined with \eqref{eq-ddense}, implies
\begin{align}\label{eq-Is}
\limsup_{s\to1^-}{\rm I}(s)\le C_0 (1-s_0)^{-\frac1q}\epsilon.
\end{align}
By Lemma \ref{lem-con} and $g\in C_{\rm c}^\infty$, we obtain
\begin{align}\label{eq-iis}
\lim_{s\to1^-}{\rm II}(s)=0.
\end{align}
Furthermore, from \eqref{eq-ssim}, it follows that there exists a
positive constant $C_1$, depending only on $n$, $k$, and $q$, such that
\begin{align*}
{\rm III}\le C_1\|\nabla^k(f-g)\|_Y\le C_1\epsilon.
\end{align*}
Combining this, \eqref{eq-iiii}, \eqref{eq-Is}, and \eqref{eq-iis}, we
conclude that
\begin{align*}
&\limsup_{s\to1^-}
\left|
(1-s)^{\frac1q}
\left\|\mathfrak{D}^{s,k}_{q,{\rm in}}(f)\right\|_{X^{1-s}Y^s}
-(kq)^{-\frac1q}\left\|\mathbb{D}^k_q(f)\right\|_Y
\right|
\le \left[C_0(1-s_0)^{-\frac1q}+C_1\right]\epsilon.
\end{align*}
Since $\epsilon$ is arbitrary, we deduce \eqref{eq-in1}, which completes
the proof of Proposition \ref{proIN}.
\end{proof}

Now, we turn to the outer part $\mathfrak{D}^{s,k}_{q,{\rm out}}(f)$.

\begin{proposition}\label{proOut}
Let $X,Y$ be {\rm BBF} spaces. Assume that $p_X,p_Y\in(1,\infty)$. Let
$k\in\mathbb{N}$ and $q\in[1,\infty)$ satisfy
$n(\frac1{p_Y}-\frac1q)<k$. Then, for any $f\in X\cap W^{k,Y}$,
\begin{align}\label{eq-out0}
\lim_{s\to1^-}(1-s)^{\frac1q}
\left\|\mathfrak{D}^{s,k}_{q,{\rm out}}(f)\right\|_{X^{1-s}Y^s}=0.
\end{align}
\end{proposition}

\begin{proof}
Let $f\in X\cap W^{k,Y}$. By the assumption
$n(\frac1{p_Y}-\frac1q)<k$, we find that there exists $s_0\in(0,1)$ such
that
\begin{align*}
n\left(\frac{1-s_0}{\min\{p_X,p_Y\}}+\frac{s_0}{p_Y}-\frac1q\right)<s_0k.
\end{align*}
Choose $p_1\in(1,\min\{p_X,p_Y\})$ and $p_2\in(1,p_Y)$ satisfying that
\begin{align*}
n\left(\frac{1-s_0}{p_1}+\frac{s_0}{p_2}-\frac1q\right)<s_0k.
\end{align*}
Using this, \eqref{eq-Pout} with $r:=1$, and
$M_{p_1}^\sharp(f)\lesssim M_{p_1}(f)$, we find that, for any
$s\in(s_0,1)$,
\begin{align}\label{eq-out1}
&(1-s)^{\frac1q}
\left\|\mathfrak{D}^{s,k}_{q,{\rm out}}(f)\right\|_{X^{1-s}Y^s}
\lesssim
\frac{(1-s)^{\frac1q}}{(s-s_0)^{\frac1q}}
\left\|
\left[M_{p_1}(f)\right]^{1-s_0}
\left[M_{p_2}(|\nabla^k f|)\right]^{s_0}
\right\|_{X^{1-s}Y^s}.
\end{align}
Observe that, for any $s\in(s_0,1)$,
\begin{align*}
X^{1-s}Y^s
=
\left(X^{1-s_0}Y^{s_0}\right)^{\frac{1-s}{1-s_0}}
Y^{\frac{s-s_0}{1-s_0}}.
\end{align*}
From this, \eqref{eq-UP}, and the boundedness of $M$ on
$X^{\frac1{p_1}}$ and $Y^{\frac1{p_2}}$, it follows that
\begin{align*}
&\left\|
\left[M_{p_1}(f)\right]^{1-s_0}
\left[M_{p_2}(|\nabla^k f|)\right]^{s_0}
\right\|_{X^{1-s}Y^s}
\notag\\
&\quad\le
\left\|
\left[M_{p_1}(f)\right]^{1-s_0}
\left[M_{p_2}(|\nabla^k f|)\right]^{s_0}
\right\|_{X^{1-s_0}Y^{s_0}}^{\frac{1-s}{1-s_0}}
\left\|
\left[M_{p_1}(f)\right]^{1-s_0}
\left[M_{p_2}(|\nabla^k f|)\right]^{s_0}
\right\|_{Y}^{\frac{s-s_0}{1-s_0}}
\notag\\
&\quad\le
\left[
\|M_{p_1}(f)\|_X^{1-s_0}
\|M_{p_2}(|\nabla^k f|)\|_Y^{s_0}
\right]^{\frac{1-s}{1-s_0}}
\left[
\|M_{p_1}(f)\|_Y^{1-s_0}
\|M_{p_2}(|\nabla^k f|)\|_Y^{s_0}
\right]^{\frac{s-s_0}{1-s_0}}
\notag\\
&\quad\lesssim
\left(
\|f\|_X^{1-s_0}\|\nabla^k f\|_Y^{s_0}
\right)^{\frac{1-s}{1-s_0}}
\left(
\|f\|_Y^{1-s_0}\|\nabla^k f\|_Y^{s_0}
\right)^{\frac{s-s_0}{1-s_0}}
\lesssim \|f\|_X+\|f\|_{W^{k,Y}}.
\end{align*}
By this and \eqref{eq-out1}, we obtain \eqref{eq-out0}. This completes
the proof of Proposition \ref{proOut}.
\end{proof}

Finally, applying Propositions \ref{proIN} and \ref{proOut}, we give the
proof of Theorem \ref{thmBBM}.

\begin{proof}[Proof of Theorem \ref{thmBBM}]
Let $f\in X\cap W^{k,Y}$. Since
\begin{align*}
\mathfrak{D}^{s,k}_{q}(f)
\le
\mathfrak{D}^{s,k}_{q,{\rm in}}(f)
+
\mathfrak{D}^{s,k}_{q,{\rm out}}(f)
\end{align*}
and
\begin{align*}
\mathfrak{D}^{s,k}_{q,{\rm in}}(f)
\le
\mathfrak{D}^{s,k}_{q}(f)
+
\mathfrak{D}^{s,k}_{q,{\rm out}}(f),
\end{align*}
Propositions \ref{proIN} and \ref{proOut} imply that
$(1-s)^{\frac1q}\|\mathfrak{D}^{s,k}_{q}(f)\|_{X^{1-s}Y^s}$ and
$(1-s)^{\frac1q}\|\mathfrak{D}^{s,k}_{q,{\rm in}}(f)\|_{X^{1-s}Y^s}$
have the same limit as $s\to1^-$. Applying Proposition \ref{proIN}, we
obtain Theorem \ref{thmBBM}.
\end{proof}

\subsection{Proof of Theorem \ref{thmMS}}
\label{secMS}
In this subsection, we give the proof of
Theorem \ref{thmMS}.
We first extend the structure of Triebel--Lizorkin space from the Lebesgue space $L^p$ to
the {\rm BBF} space $X$.

\begin{definition}
Let $k\in\mathbb{N}$, $\sigma\in(0,k)$, $q\in [1,\infty)$, and
$X$ be a {\rm BBF} space. The \emph{Triebel--Lizorkin type space
$F^{\sigma,k}_{X,q}$} is defined to be the set of all $f\in X$ such that
\begin{align*}
\|f\|_{F^{\sigma,k}_{X,q}}:=
\|f\|_X+ \left\|\left[
\int_{\mathbb{R}^n}\frac{|\Delta^k_h f(\cdot)|^q}{|h|^{n+\sigma q}}\,dh
\right]^{\frac{1}{q}}\right\|_X<\infty.
\end{align*}
\end{definition}

\begin{remark}
Let $X:= L^p$ with $p\in (0,\infty)$.
If $\sigma\in (\frac{n}{\min\{p,q\}},k)$, then
the space $F^{\sigma,k}_{L^p,q}$ in this case reduces to the classical Triebel--Lizorkin space
$F^{\sigma}_{p,q}$ (see \cite[p.\,101]{Tri83});
if $k=1$ and $p\in [1,\infty)$, then the space $F^{\sigma,1}_{L^p,p}$ in this case reduces to the first
order fractional Sobolev space $W^{\sigma,p}$
(see \cite{Leo23}). More embedding results about
Triebel--Lizorkin type spaces can be found in \cite{BLL21}.
\end{remark}

To show Theorem \ref{thmMS}, we need the following upper bound estimate.
Its proof is similar to that of \cite[Lemma 2.25]{PYYZ24}, with
$L^p_{\omega}$ replaced by $X$. For completeness, we include some details.

\begin{lemma}\label{lem-s0out}
Let $k\in \mathbb{N}$, $q\in[1,\infty)$, and $X$ be a {\rm BBF}
space. Assume that $X^{\frac{1}{q}}$ is a {\rm BBF} space and that the centered ball average operators
$\{\mathcal{B}_r\}_{r\in(0,\infty)}$ are uniformly bounded on
$X^{\frac{1}{q}}$. Then, for any $s\in(0,1)$ and $f\in X$,
\begin{align}\label{eq-s0out}
s^{\frac{1}{q}}
\left\|
\left[
\int_{|h|\geq1}
\frac{|\Delta_h^k f(\cdot)|^q}{|h|^{n+skq}}\,dh
\right]^{\frac{1}{q}}
\right\|_X
\lesssim
\|f\|_X,
\end{align}
where the implicit positive constant is independent of $s$ and $f$.
\end{lemma}

\begin{proof}
By decomposing $\{h\in\mathbb{R}^n:|h|\geq1\}$ into dyadic annuli,
we find that, for every $x\in\mathbb{R}^n$,
\begin{align*}
\int_{|h|\geq1}
\frac{|\Delta_h^k f(x)|^q}{|h|^{n+skq}}\,dh
&\lesssim
\sum_{m\in\mathbb{Z}_+}
2^{-m(n+skq)}
\sum_{j=0}^k
\int_{|h|<2^{m+1}}|f(x+jh)|^q\,dh\\
&\lesssim
\sum_{m\in\mathbb{Z}_+}
2^{-mskq}
\left[
|f(x)|^q+
\sum_{j=1}^k
\mathcal{B}_{j2^{m+1}}(|f|^q)(x)
\right].
\end{align*}
Applying the triangle inequality in $X^{\frac{1}{q}}$ to finite partial sums, and then using the Fatou property and the uniform boundedness of the centered ball average operators on $X^{\frac{1}{q}}$, we obtain
\begin{align*}
&\left\|
\int_{|h|\geq1}
\frac{|\Delta_h^k f(\cdot)|^q}{|h|^{n+skq}}\,dh
\right\|_{X^{\frac{1}{q}}}\\
&\quad\lesssim
\sum_{m\in\mathbb{Z}_+}
2^{-mskq}
\left[
\big\||f|^q\big\|_{X^{\frac{1}{q}}}
+
\sum_{j=1}^k
\left\|
\mathcal{B}_{j2^{m+1}}(|f|^q)
\right\|_{X^{\frac{1}{q}}}
\right]\\
&\quad\lesssim
\left(
\sum_{m\in\mathbb{Z}_+}2^{-mskq}
\right)
\big\||f|^q\big\|_{X^{\frac{1}{q}}}\lesssim
s^{-1}\|f\|_X^q.
\end{align*}
Since
\begin{align*}
\left\|
\left[
\int_{|h|\geq1}
\frac{|\Delta_h^k f(\cdot)|^q}{|h|^{n+skq}}\,dh
\right]^{\frac{1}{q}}
\right\|_X^q
=
\left\|
\int_{|h|\geq1}
\frac{|\Delta_h^k f(\cdot)|^q}{|h|^{n+skq}}\,dh
\right\|_{X^{\frac{1}{q}}},
\end{align*}
we obtain \eqref{eq-s0out}, which completes the proof of Lemma \ref{lem-s0out}.
\end{proof}

Next, we prove Theorem \ref{thmMS}.

\begin{proof}[Proof of Theorem \ref{thmMS}]
Let $f\in [\bigcup_{\sigma\in (0,k)}{F}^{\sigma,k}_{X,q}]\cap
[\bigcup_{\sigma\in (0,k)}{F}^{\sigma,k}_{Y,q}]$ with compact support. Then there exists $R_0\in(0,\infty)$ such that, for any $|h|\ge R_0$,
the sets $\{{\rm supp}\,f(\cdot+jh)\}_{j=0}^{k}$ are pairwise
disjoint. Hence, for all sufficiently small $s\in(0,1)$ and all
$|h|\ge s^{-\frac1q}$, by \eqref{eq-sho-hd},
we find that
\begin{align*}
\left|\Delta_{h}^k
f(\cdot)\right|=
\sum_{j=0}^{k}\binom{k}{j}|f(\cdot+jh)|\ge | f(\cdot)|.
\end{align*}
Thus, for all sufficiently small $s\in (0,1)$,
\begin{align}\label{eq-linf1}
{s}^{\frac{1}{q}}\left\|
\left[\int_{|h|\ge {s}^{-\frac{1}{q}}}\frac{|\Delta^k_h f(\cdot)|^q}{|h|^{n+skq}}\,dh\right]^{\frac{1}{q}}
\right\|_{X^{1-s}Y^{s}}
\ge
{s}^{\frac{1}{q}}\left\|
\left[\int_{|h|\ge {s}^{-\frac{1}{q}}}\frac{| f(\cdot)|^q}{|h|^{n+skq}}\,dh\right]^{\frac{1}{q}}
\right\|_{X^{1-s}Y^{s}}.
\end{align}
Note that, for any $s\in (0,1)$,
\begin{align*}
\left[s\int_{|h|\ge {s}^{-\frac{1}{q}}}\frac{1}{|h|^{n+skq}}\,dh\right]^{\frac{1}{q}}
= \left(\frac{s^{sk}}{kq}\left|\mathbb{S}^{n-1}\right|\right)^{\frac{1}{q}}.
\end{align*}
Since $s^{sk}\to1$ as $s\to0^+$, from this, \eqref{eq-linf1}, and
Lemma \ref{lem-equa} applied to the pair $(Y,X)$ with $1-s$ in place of
$s$, we deduce that
\begin{align}\label{eq-XX1}
\liminf_{s\to 0^+}{s}^{\frac{1}{q}}
\left\|
\mathfrak{D}^{s,k}_q (f)
\right\|_{X^{1-s}Y^{s}}
\ge \left(\frac{1}{kq}\left|\mathbb{S}^{n-1}\right|\right)^{\frac{1}{q}}
\left\|f
\right\|_{X}.
\end{align}
On the other hand, applying \eqref{eq-s0out} and the assumption $f\in X\cap Y$, we conclude that
\begin{align}\label{eq-XX0}
&\limsup_{s\to 0^+}{s}^{\frac{1}{q}}\left\|\left[\int_{|h|\ge 1}\frac{|\Delta^k_h f(\cdot)|^q}{|h|^{n+skq}}\,dh\right]^{\frac{1}{q}}\right\|_{X^{1-s}Y^{s}}\notag\\
& \quad\le
\limsup_{s\to 0^+}\left[
{s}^{\frac{1}{q}}\left\|\left[\int_{|h|\ge 1}\frac{|\Delta^k_h f(\cdot)|^q}{|h|^{n+skq}}\,dh\right]^{\frac{1}{q}}\right\|_{X}\right]^{1-s}
\left[
{s}^{\frac{1}{q}}\left\|\left[\int_{|h|\ge 1}\frac{|\Delta^k_h f(\cdot)|^q}{|h|^{n+skq}}\,dh\right]^{\frac{1}{q}}\right\|_{Y}\right]^{s}
\notag\\
&\quad \lesssim \limsup_{s\to 0^+}\|f\|_X^{1-s} \|f\|_Y^{s}
=\|f\|_X.
\end{align}
Moreover, since
$f\in [\bigcup_{\sigma\in (0,k)}{F}^{\sigma,k}_{X,q}]
\cap [\bigcup_{\sigma\in (0,k)}{F}^{\sigma,k}_{Y,q}]$, we can choose
$\sigma_1,\sigma_2 \in(0,k)$ such that
$f\in {F}^{\sigma_1,k}_{X,q}\cap {F}^{\sigma_2,k}_{Y,q}$. By this, we find that, for any $s\in (0,\sigma_1/k)$,
\begin{align*}
\left\|\left[\int_{|h|\le  1 }\frac{|\Delta^k_h f(\cdot)|^q}{|h|^{n+s kq}}\,dh\right]^{\frac{1}{q}}\right\|_{X}
&\le
\left\|\left[\int_{|h|\le 1}\frac{|\Delta^k_h f(\cdot)|^q}{|h|^{n+\sigma_1 q}}\,dh\right]^{\frac{1}{q}}\right\|_{X}
\le  \|f\|_{{F}^{\sigma_1,k}_{X,q}}<\infty
\end{align*}
and, for any $s\in (0,\sigma_2/k)$,
\begin{align*}
\left\|\left[\int_{|h|\le 1}\frac{|\Delta^k_h f(\cdot)|^q}{|h|^{n+skq}}\,dh\right]^{\frac{1}{q}}\right\|_{Y}
\le  \|f\|_{{F}^{\sigma_2,k}_{Y,q}}<\infty.
\end{align*}
From these and \eqref{eq-UP}, it follows that
\begin{align*}
\limsup_{s\to 0^+}{s}^{\frac{1}{q}}\left\|\left[\int_{|h|\le 1}\frac{|\Delta^k_h f(\cdot)|^q}{|h|^{n+skq}}\,dh\right]^{\frac{1}{q}}\right\|_{X^{1-s}Y^{s}}
\le \limsup_{s\to 0^+} {s}^{\frac{1}{q}}
\|f\|_{{F}^{\sigma_1,k}_{X,q}}^{1-s}
\|f\|_{{F}^{\sigma_2,k}_{Y,q}}^{s}=0.
\end{align*}
This, together with \eqref{eq-XX0}, implies that
\begin{align*}
\limsup_{s\to 0^+}{s}^{\frac{1}{q}}
\left\|
\mathfrak{D}^{s,k}_q (f)
\right\|_{X^{1-s}Y^{s}}
\lesssim \|f\|_X.
\end{align*}
Combining this and \eqref{eq-XX1}, we obtain \eqref{eq-XX3}.
This then completes the proof of Theorem \ref{thmMS}.
\end{proof}

\section{Proof of Theorem \ref{thm-fs}}\label{secOP}

In this section, we aim to show Theorem \ref{thm-fs}.
We begin with some necessary definitions.
Let $f\in L^0$.
The \emph{distribution
function of $f$} is defined by setting, for any $\alpha\in (0,\infty)$,
\begin{align*}
d_f (\alpha):= |\{x\in \mathbb{R}^n:|f(x)|>\alpha\}|.
\end{align*}
The \emph{non-increasing rearrangement function of $f$}
is defined by setting, for any $t\in (0,\infty)$,
\begin{align}\label{eq-fstar}
f^* (t):= \inf\left\{\alpha\in (0,\infty):d_f (\alpha)\le t\right\}.
\end{align}
It is well known that, for any $f\in L^0$, $f^*$ is
a non-negative, non-increasing, and right-continuous function on $[0,\infty)$ (see \cite[p.\,41, Proposition 1.7]{BS88}). Recall that
a Banach function space $X$ [see Remark \ref{rem-bbf}(ii) for its definition]
is called a \emph{rearrangement invariant Banach function space} if,
for any $f\in L^0$ and $g\in X$ satisfying $d_f = d_g$,
it holds that $f\in X$ and $\|f\|_X =\|g\|_X$.

If $X$ is a rearrangement invariant Banach function space,
then, by the Luxemburg representation theorem (see, for instance, \cite[p.\,62, Theorem 4.10]{BS88}),
there exists a unique rearrangement invariant Banach function
space $\overline{X}$ on
$(0,\infty)$ such that, for any $f\in L^0$,
\begin{align}\label{eq-OX}
\|f\|_X =\|f^*\|_{\overline{X}}.
\end{align}
This result proves that
the norm of any rearrangement invariant space is essentially determined
by the distribution of its functions.
Moreover, from \cite[p.\,148, Proposition 5.11]{BS88}, we infer that the \emph{dilation operators} are bounded on rearrangement
invariant Banach function spaces; that is,
if $X$ is a rearrangement
invariant Banach function space, then,
for any $t\in (0,\infty)$ and $f\in X$,
\begin{align}\label{eq-bX}
\left\| f^*(\cdot/t)\right\|_{\overline{X}} \le
\max\{1,t\}\|f\|_X.
\end{align}

To show Theorem \ref{thm-fs}, we need the following conclusion, which proves
that, among all rearrangement invariant Banach function spaces, the Calder\'on--Lozanovski\u{\i} space
is the minimal target space for which the Gagliardo--Nirenberg inequalities hold.

\begin{theorem}\label{thmOP}
Let $X,Y$, and $B$ be rearrangement invariant Banach function spaces.
\begin{enumerate}[{\rm (i)}]
\item Let $k\in\mathbb{N}$, $s\in(0,1)$, and $q\in[1,\infty)$. If there exists a positive constant
$C$ such that the fractional
Gagliardo--Nirenberg inequality
\begin{align}\label{eq-GN}
\left\|\mathfrak{D}^{s,k}_q(f)\right\|_B
\le C
\|f\|_X^{1-s}\left\|\nabla^k f\right\|_Y^s
\end{align}
holds for any $f\in X\cap\dot{W}^{k,Y}$, then
$X^{1-s}Y^s\hookrightarrow B$.

\item Let $j,k\in\mathbb{N}$ with $1\le j<k$. If
there exists a positive constant
$C$ such that
the Gagliardo--Nirenberg
inequality
\begin{align}\label{eq-GN-integer}
\left\|\nabla^j f\right\|_B
\le C
\|f\|_X^{1-\frac{j}{k}}
\left\|\nabla^k f\right\|_Y^{\frac{j}{k}}
\end{align}
holds for any $f\in X\cap\dot{W}^{k,Y}$, then
$X^{1-\frac{j}{k}}Y^{\frac{j}{k}}\hookrightarrow B$.
\end{enumerate}
\end{theorem}

\begin{proof}
We first show (i). By the Fatou property and the fact that any non-negative
measurable functions can be approximated from below by an increasing sequence
of simple functions, we find that, to prove $X^{1-s}Y^s\hookrightarrow B$, it
suffices to show that there exists a positive constant $C$ such that, for any
non-negative simple functions, after taking a common refinement of their level
sets, of the form
\begin{align}\label{eq-uv}
u=\sum_{j=1}^N a_j{\bf 1}_{E_j}\quad\text{and}\quad
v=\sum_{j=1}^N b_j{\bf 1}_{E_j},
\end{align}
where $N\in\mathbb{N}$, $\{E_j\}_{j=1}^N$ are mutually disjoint sets of finite
measure, and $\{a_j\}_{j=1}^N$, $\{b_j\}_{j=1}^N$ are positive sequences, one has
\begin{align}\label{eq-gn1}
\left\|u^{1-s}v^s\right\|_B
\le C\|u\|_X^{1-s}\|v\|_Y^s.
\end{align}
Let $u,v$ be the same as in \eqref{eq-uv} and
$\phi\in C_{\rm c}^\infty(B({\bf 0},1))$ satisfy that
\begin{align}\label{eq-phi}
{\bf 1}_{B({\bf 0},\frac12)}\le\phi\le {\bf 1}_{B({\bf 0},1)}.
\end{align}
Then, for any $x\in B({\bf 0},\frac12)$, since
$|x+\ell h|\ge |h|-|x|>1$ for any $\ell\in\{1,\ldots,k\}$ and
$h\in\mathbb{R}^n$ with $|h|\ge2$, it holds that
\begin{align}\label{eq-phi3}
\mathfrak{D}^{s,k}_q(\phi)(x)
\ge
\left(\int_{|h|\ge2}\frac{1}{|h|^{n+skq}}\,dh\right)^{\frac1q}
=:c_{n,k,q,s},
\end{align}
where $c_{n,k,q,s}$ is a positive constant depending only on
$n$, $k$, $q$, and $s$.
For any fixed $R\in(0,\infty)$ and any $j\in\{1,\ldots,N\}$, let
$\lambda_j:=(a_j/b_j)^{\frac1k}$ and
\begin{align}\label{eq-mj}
M_j:=\left\lfloor \frac{|E_j|}{\omega_nR^n\lambda_j^n}\right\rfloor.
\end{align}
Let $I:=\{(j,m)\in\mathbb{N}\times\mathbb{N}:
1\le j\le N,\ 1\le m\le M_j\}$. Then the cardinality
$\# I=\sum_{j=1}^NM_j$ is finite. Fix $L\in(0,\infty)$ and choose points
$\{x_{j,m}\}_{j\in\{1,\ldots,N\},\,m\in\{1,\ldots,M_j\}}$ in $\mathbb{R}^n$
such that, for any $(j_1,m_1),(j_2,m_2)\in I$ with $(j_1,m_1)\ne (j_2,m_2)$, $|x_{j_1,m_1}-x_{j_2,m_2}|\ge L$.
For any $i=(j,m)\in I$, let $c_i:=a_j$, $A_i:=R\lambda_j$, and $x_i:=x_{j,m}$.
For any $R,L\in(0,\infty)$ and $x\in\mathbb{R}^n$, define
\begin{align}\label{eq-fR}
f_{R,L}(x)
:=\sum_{j=1}^Na_j\sum_{m=1}^{M_j}
\phi\left(\frac{x-x_{j,m}}{R\lambda_j}\right)
=\sum_{i\in I}c_i\phi\left(\frac{x-x_i}{A_i}\right)
=:\sum_{i\in I}\phi_i(x).
\end{align}
Then $f_{R,L}\in C_{\rm c}^\infty\subset W^{k,1}_{\rm loc}$. Let
$\rho:=\max_{i\in I}A_i$. We claim that, for any sufficiently small
$R\in(0,\infty)$ and any $L\in(2\rho,\infty)$,
\begin{align*}
\|f_{R,L}\|_X\sim\|u\|_X,
\quad
\left\|\nabla^k f_{R,L}\right\|_Y\lesssim R^{-k}\|v\|_Y,
\end{align*}
and
\begin{align}\label{eq-dsss}
\liminf_{L\to\infty}\left\|\mathfrak{D}^{s,k}_q(f_{R,L})\right\|_B
\gtrsim
R^{-sk}\left\|u^{1-s}v^s\right\|_B,
\end{align}
where the implicit positive constants in the claim are independent of $R$ and
$L$. If this claim holds, then $f_{R,L}\in X\cap\dot W^{k,Y}$. This, together
with the above claim and assumption \eqref{eq-GN}, implies \eqref{eq-gn1}
and hence $X^{1-s}Y^s\hookrightarrow B$. Thus, to complete the proof of (i),
we only need to prove the above claim.

We first show $\|f_{R,L}\|_X\sim\|u\|_X$. By a change of variables, we obtain,
for any $i\in I$ and $\alpha\in(0,\infty)$,
\begin{align}\label{eq-sing}
d_{\phi_i}(\alpha)
=A_i^n\left|\left\{y\in\mathbb{R}^n:\phi(y)>\frac{\alpha}{c_i}\right\}\right|
=A_i^nd_\phi\left(\frac{\alpha}{c_i}\right).
\end{align}
From $L\in(2\rho,\infty)$, we deduce that the supports of
$\{\phi_i\}_{i\in I}$ are pairwise disjoint. By this and \eqref{eq-sing}, we
conclude that, for any $\alpha\in(0,\infty)$,
\begin{align*}
d_{f_{R,L}}(\alpha)
=\sum_{i\in I}d_{\phi_i}(\alpha)
=\sum_{j=1}^NM_j(R\lambda_j)^nd_\phi\left(\frac{\alpha}{a_j}\right).
\end{align*}
From \eqref{eq-mj}, we deduce that
\begin{align*}
\frac{|E_j|}{\omega_n}-(R\lambda_j)^n<M_j(R\lambda_j)^n\le\frac{|E_j|}{\omega_n},
\end{align*}
with $\omega_n:=|B({\bf 0},1)|$. Moreover, using \eqref{eq-phi}, we conclude
that, for any $\alpha\in(0,\infty)$,
\begin{align*}
d_{{\bf 1}_{B({\bf 0},\frac12)}}(\alpha)
\le d_\phi(\alpha)
\le d_{{\bf 1}_{B({\bf 0},1)}}(\alpha).
\end{align*}
From these and a basic calculation, it follows that, for any $R$ satisfying
\begin{align}\label{eq-R}
R^n\le\min\left\{\frac{|E_j|}{2\omega_n\lambda_j^n}:j=1,\ldots,N\right\}
\end{align}
and for any $\alpha\in(0,\infty)$,
\begin{align*}
\frac{1}{2^{n+1}}d_u(\alpha)
&=\sum_{j=1}^N\frac{|E_j|}{2\omega_n}
d_{{\bf 1}_{B({\bf 0},\frac12)}}\left(\frac{\alpha}{a_j}\right)
<d_{f_{R,L}}(\alpha)\le d_u(\alpha).
\end{align*}
This, together with the definition of the non-increasing rearrangement
function [see \eqref{eq-fstar}], further implies that, for any
$t\in(0,\infty)$,
\begin{align}\label{eq-dff}
u^*(2^{n+1}t)\le f_{R,L}^*(t)\le u^*(t).
\end{align}
By this, \eqref{eq-OX}, and \eqref{eq-bX}, we find that, for any $R$
satisfying \eqref{eq-R} and
for any $L\in(2\rho,\infty)$,
\begin{align}\label{eq-fxx}
\|f_{R,L}\|_X=\|f_{R,L}^*\|_{\overline X}
\sim\|u^*\|_{\overline X}=\|u\|_X
\end{align}
with the positive equivalence constants independent of both $R$ and $L$.

Next, we prove $\left\|\nabla^k f_{R,L}\right\|_Y\lesssim R^{-k}\|v\|_Y$.
From \eqref{eq-fR} and $\lambda_j^k=a_j/b_j$, we infer that, for any
$x\in\mathbb{R}^n$,
\begin{align*}
\left|\nabla^k f_{R,L}(x)\right|
&\lesssim
\sum_{j=1}^N\frac{b_j}{R^k}
\sum_{m=1}^{M_j}{\bf 1}_{B(x_{j,m},R\lambda_j)}(x).
\end{align*}
Similar to the proof of \eqref{eq-dff}, we conclude that, for any
$R\in(0,\infty)$, $L\in(2\rho,\infty)$, and $t\in(0,\infty)$,
\begin{align*}
\left|\nabla^k f_{R,L}\right|^*(t)
\lesssim
\left(\frac{v}{R^k}\right)^*(t).
\end{align*}
Thus,
\begin{align}\label{eq-fYY}
\left\|\nabla^k f_{R,L}\right\|_Y\lesssim R^{-k}\|v\|_Y
\end{align}
with the implicit positive constant independent of both $R$ and $L$.

Finally, we show \eqref{eq-dsss}. For any $R,L\in(0,\infty)$ and
$x\in\mathbb{R}^n$, let
\begin{align*}
\Phi_{R,L}(x)
&:=\sum_{i\in I}\mathfrak{D}^{s,k}_q(\phi_i)(x)
{\bf 1}_{\bigcup_{i\in I}B(x_i,2\rho)}(x)
=\sum_{j=1}^N\sum_{m=1}^{M_j}
\frac{a_j^{1-s}b_j^s}{R^{sk}}
\mathfrak{D}^{s,k}_q\phi\left(\frac{x-x_{j,m}}{R\lambda_j}\right)
{\bf 1}_{\bigcup_{i\in I}B(x_i,2\rho)}(x)
\end{align*}
and
\begin{align}\label{eq-EEE}
\mathcal{E}_{R,L}(x)
:=\mathfrak{D}^{s,k}_q(f_{R,L})(x){\bf 1}_{\bigcup_{i\in I}B(x_i,2\rho)}(x)
-\Phi_{R,L}(x).
\end{align}
To prove \eqref{eq-dsss}, we first show that, for any $R\in(0,\infty)$,
$L\in(6k\rho,\infty)$, and $x\in\mathbb{R}^n$,
\begin{align}\label{eq-aal1}
\left|\mathcal{E}_{R,L}(x)\right|
\lesssim
L^{-\frac{n}{q}-sk}{\bf 1}_{\bigcup_{i\in I}B(x_i,2\rho)}(x),
\end{align}
where the implicit positive constant may depend on $R$, $u$, and $v$, but is
independent of $L$. Fix $R\in(0,\infty)$ and $L\in(6k\rho,\infty)$ and assume
$x\in B(x_{i_0},2\rho)$ for some $i_0\in I$. From ${\rm supp\,}(\phi)\subset B({\bf 0},1)$,
we infer that, for any $i\in I$, ${\rm supp\,}(\phi_i)\subset B(x_i,A_i)\subset B(x_i,\rho)$.
For any $i\in I$ with $i\ne i_0$, by this and the assumption $|x_i-x_{i_0}|\ge L$, we
have $|x-x_i|\ge L-2\rho>4\rho$ and hence
\begin{align}\label{eq-gix}
\phi_i(x)=0.
\end{align}
Moreover, when $|h|\le L/(2k)$, for any $\ell\in\{1,\ldots,k\}$ and
$i\ne i_0$,
\begin{align*}
|x+\ell h-x_i|
\ge |x_i-x_{i_0}|-|x-x_{i_0}|-\ell|h|
\ge L-2\rho-\frac L2>\rho.
\end{align*}
Thus, the terms generated by the bumps $\{\phi_i\}_{i\ne i_0}$ vanish in
$\Delta_h^kf_{R,L}(x)$ for $|h|\le L/(2k)$. Using this and Minkowski's inequality,
we find that
\begin{align}\label{eq-j1j2}
&\left|\mathfrak{D}^{s,k}_q(f_{R,L})(x)-\mathfrak{D}^{s,k}_q(\phi_{i_0})(x)\right|^q
\lesssim
\int_{|h|>\frac L{2k}}
\frac{\left|\sum_{i\ne i_0}\Delta_h^k\phi_i(x)\right|^q}{|h|^{n+skq}}\,dh.
\end{align}
By the support condition, \eqref{eq-gix}, the change of variables
$y=x+\ell h$ for $\ell\in\{1,\ldots,k\}$, and the finiteness of $I$, we obtain
\begin{align*}
&\int_{|h|>\frac L{2k}}
\frac{\left|\sum_{i\ne i_0}\Delta_h^k\phi_i(x)\right|^q}{|h|^{n+skq}}\,dh
\lesssim
L^{-n-skq}\sum_{i\ne i_0}|a_i|^qA_i^n
\lesssim L^{-n-skq},
\end{align*}
where the implicit positive constant may depend on $R$, $u$, and $v$, but is
independent of $L$. Applying this and \eqref{eq-j1j2}, we conclude that
\begin{align}\label{eq-gk}
\left|\mathfrak{D}^{s,k}_q(f_{R,L})(x)-\mathfrak{D}^{s,k}_q(\phi_{i_0})(x)\right|
\lesssim L^{-\frac{n}{q}-sk}.
\end{align}
Similarly, from \eqref{eq-gix}, we deduce that, for any $x\in B(x_{i_0},2\rho)$,
\begin{align*}
\sum_{i\ne i_0}\mathfrak{D}^{s,k}_q(\phi_i)(x)
\lesssim L^{-\frac{n}{q}-sk},
\end{align*}
where the implicit positive constant may depend on $R$, $u$, and $v$, but is
independent of $L$. By this and \eqref{eq-gk}, we obtain, for any
$x\in B(x_{i_0},2\rho)$,
\begin{align*}
\left|\mathfrak{D}^{s,k}_q(f_{R,L})(x)-\sum_{i\in I}\mathfrak{D}^{s,k}_q(\phi_i)(x)\right|
\lesssim L^{-\frac{n}{q}-sk}.
\end{align*}
This proves \eqref{eq-aal1}.

Consequently, by \eqref{eq-aal1}, we find that, for any fixed $R$ and
$t\in(0,\infty)$,
\begin{align*}
\mathcal{E}_{R,L}^*(t)
\lesssim
L^{-\frac{n}{q}-sk}{\bf 1}_{(0,\,\# I\omega_n(2\rho)^n)}(t),
\end{align*}
where the implicit positive constant is independent of $L$. From this and
\eqref{eq-OX}, it follows that, for any fixed $R$,
\begin{align*}
\|\mathcal{E}_{R,L}\|_B
=\|\mathcal{E}_{R,L}^*\|_{\overline B}
\lesssim
L^{-\frac{n}{q}-sk}\left\|{\bf 1}_{(0,\,\# I\omega_n(2\rho)^n)}\right\|_{\overline B}
\to0
\end{align*}
as $L\to\infty$. This, together with \eqref{eq-EEE}, further implies that
\begin{align}\label{eq-key}
\liminf_{L\to\infty}
\left\|\mathfrak{D}^{s,k}_q(f_{R,L}){\bf 1}_{\bigcup_{i\in I}B(x_i,2\rho)}\right\|_B
=
\liminf_{L\to\infty}\|\Phi_{R,L}\|_B.
\end{align}
Moreover, using \eqref{eq-phi3}, we obtain, for any $R,L\in(0,\infty)$ and
$x\in\mathbb{R}^n$,
\begin{align*}
\Phi_{R,L}(x)
&\ge c_{n,k,q,s}\sum_{j=1}^N\sum_{m=1}^{M_j}
\frac{a_j^{1-s}b_j^s}{R^{sk}}
{\bf 1}_{B({\bf 0},\frac12)}\left(\frac{x-x_{j,m}}{R\lambda_j}\right)
=:\Psi_{R,L}(x).
\end{align*}
Similar to the proof of \eqref{eq-dff}, we conclude that, for any $R$ satisfying
\eqref{eq-R}, $L\in(6k\rho,\infty)$, and $t\in(0,\infty)$,
\begin{align*}
\Phi_{R,L}^*(t)\ge\Psi_{R,L}^*(t)
\ge
\left(\frac{c_{n,k,q,s}u^{1-s}v^s}{R^{sk}}\right)^*(2^{n+1}t).
\end{align*}
This, combined with \eqref{eq-OX}, further implies that
\begin{align*}
\liminf_{L\to\infty}\|\Phi_{R,L}\|_B
=\liminf_{L\to\infty}\|\Phi_{R,L}^*\|_{\overline B}
\gtrsim
R^{-sk}\|(u^{1-s}v^s)^*\|_{\overline B}
=R^{-sk}\|u^{1-s}v^s\|_B
\end{align*}
with the implicit positive constant independent of $R$. From this, the lattice
property, and \eqref{eq-key}, we deduce \eqref{eq-dsss}. Combining this,
\eqref{eq-fxx}, and \eqref{eq-fYY}, we find that the above claim holds. This
completes the proof of (i).

Finally, we show (ii), whose proof is similar to (i) and even simpler.
Indeed, by an argument similar to that used above, we only need to prove that \eqref{eq-gn1} holds
with $s=\frac jk$ for all simple functions $u$ and $v$ as in \eqref{eq-uv}.
Choose $\eta\in C_{\rm c}^\infty(B({\bf 0},1))$ such that
$\eta=1$ on $B({\bf 0},\frac12)$ and, for any $x:=(x_1,\ldots,x_n)\in\mathbb{R}^n$, define
$\psi(x):=\frac{x_1^j}{j!}\eta(x).$
Then $|\nabla^j\psi|\ge {\bf 1}_{B({\bf 0},\frac12)}.$
For any $R,L\in(0,\infty)$, let
$g_{R,L}$ be as in \eqref{eq-fR} with
$\lambda_\ell:=a_\ell^{\frac1k}b_\ell^{-\frac1k}$ and $\phi$ replaced by $\psi$.
Then, for any $x\in\mathbb{R}^n$,
\begin{align*}
\left|\nabla^j g_{R,L}(x)\right|
&=
\frac1{R^j}
\sum_{\ell=1}^N a_\ell^{1-\frac{j}{k}}b_\ell^{\frac{j}{k}}
\sum_{m=1}^{M_\ell}
\left|\nabla^j\psi\left(
\frac{x-x_{\ell,m}}{R\lambda_\ell}
\right)\right|
\end{align*}
and
\begin{align*}
\left|\nabla^k g_{R,L}(x)\right|
&\lesssim
\frac1{R^k}
\sum_{\ell=1}^N b_\ell
\sum_{m=1}^{M_\ell}
{\bf 1}_{B(x_{\ell,m},R\lambda_\ell)}(x).
\end{align*}
Similar to the proofs of \eqref{eq-gn1} and \eqref{eq-fxx}, we conclude that,
for any sufficiently small $R\in(0,\infty)$ and any sufficiently
large $L\in(0,\infty)$,
\begin{align*}
\|g_{R,L}\|_X\sim\|u\|_X,
\quad
\left\|\nabla^k g_{R,L}\right\|_Y\lesssim R^{-k}\|v\|_Y,
\end{align*}
and
\begin{align*}
\left\|\nabla^j g_{R,L}\right\|_B
\gtrsim
R^{-j}\left\|u^{1-\frac jk}v^{\frac jk}\right\|_B,
\end{align*}
where the implicit positive constants are independent of both $R$ and $L$.
Applying this and \eqref{eq-GN-integer} to $g_{R,L}$, we obtain that
\eqref{eq-gn1} holds with $s=\frac jk$ and hence (ii). This
then completes the proof of Theorem \ref{thmOP}.
\end{proof}

With the help of Theorem \ref{thmOP}, we show Theorem \ref{thm-fs}.
\begin{proof}[Proof of Theorem \ref{thm-fs}]
The proof of the estimates is similar to that of Theorem \ref{thmGNS}: after
choosing suitable convexification indices, we apply Theorem \ref{thm-ps} and
then use the Calder\'on--Lozanovski\u{\i} construction together with the boundedness
of the Hardy--Littlewood maximal operator.

Let $f\in \dot W^{k,Y}$. Since $p_Y>1$ in both (i) and (ii), from Remark
\ref{rem-Boyd}(ii), it follows that
$f\in\bigcap_{1\le p<p_Y}W^{k,p}_{\rm loc}$.

We first prove (i). For any $t\in[0,1]$, define
$F(t):=n(\frac{1-t}{p_X}+\frac{t}{p_Y}-\frac1q)-tk$. By the assumption
$F(s)<0$ and the continuity of $F$, we find that there exist
$s_0,\widetilde{s_0}\in(0,1)$ such that $s_0<s<\widetilde{s_0}$,
$F(s_0)<0$, and $F(\widetilde{s_0})<0$. As in the proof of Theorem
\ref{thmGNS}, we can choose $p_1\in(1,p_X)$ and $p_2\in(1,p_Y)$ satisfying
that
\begin{align*}
n\left(\frac{1-s_0}{p_1}+\frac{s_0}{p_2}-\frac1q\right)<s_0k
\quad
\text{and}
\quad
n\left(\frac{1-\widetilde{s_0}}{p_1}
+\frac{\widetilde{s_0}}{p_2}-\frac1q\right)<\widetilde{s_0}k.
\end{align*}
Applying Theorem \ref{thm-ps}(i) with these parameters and arguing as in the
proof of Theorem \ref{thmGNS}(i), we obtain \eqref{eq-XYss}.

Next, for any $t\in[0,1]$, define
$G(t):=n(\frac{t}{p_Y}-\frac1q)-tk$. By the assumption $G(s)<0$ and the
continuity of $G$, we can choose $s_0,\widetilde{s_0}\in(0,1)$ such that
$s_0<s<\widetilde{s_0}$, $G(s_0)<0$, and $G(\widetilde{s_0})<0$. As above,
we can choose $p_1\in(1,\infty)$ large enough and $p_2\in(1,p_Y)$
sufficiently close to $p_Y$ such that
\begin{align*}
n\left(\frac{1-s_0}{p_1}+\frac{s_0}{p_2}-\frac1q\right)<s_0k
\quad
\text{and}
\quad
n\left(\frac{1-\widetilde{s_0}}{p_1}
+\frac{\widetilde{s_0}}{p_2}-\frac1q\right)<\widetilde{s_0}k.
\end{align*}
Applying Theorem \ref{thm-ps}(i) with these parameters and using the BMO
argument in the proof of Theorem \ref{thmGNS}(ii), Remark \ref{rem-cp}(iii),
and the boundedness of $M$ on $Y^{\frac1{p_2}}$, we obtain \eqref{eq-BMO-fs}.

It remains to show the optimality statements. The sufficiency of
\eqref{eq-OPP-ri} follows from \eqref{eq-XYss} and the embedding
$X^{1-s}Y^s\hookrightarrow B$. Conversely, if \eqref{eq-OPP-ri} holds, then
Theorem \ref{thmOP}(i) implies $X^{1-s}Y^s\hookrightarrow B$.

Similarly, the sufficiency of \eqref{eq-OPP-bmo} follows from
\eqref{eq-BMO-fs} and the embedding $Y^{\frac1s}\hookrightarrow B$. Conversely,
suppose that \eqref{eq-OPP-bmo} holds. Since
$L^\infty\hookrightarrow\operatorname{BMO}$, it follows that
\begin{align*}
\left\|\mathfrak D_q^{s,k}(f)\right\|_B
\lesssim
\|f\|_{L^\infty}^{1-s}
\left\|\nabla^k f\right\|_Y^s
\end{align*}
for any $f\in L^\infty\cap\dot W^{k,Y}$. Applying Theorem \ref{thmOP}(i) with
$X=L^\infty$ and using Remark \ref{rem-cp}(iii), we conclude that
$Y^{\frac1s}\hookrightarrow B$. This then completes the proof of
Theorem \ref{thm-fs}.
\end{proof}

\section{Sharpness of Assumptions on
\texorpdfstring{$q$}{q}
in Theorems \ref{thmGNS} and \ref{thm-fs}}\label{sec-sharp}

In this section, we establish the sharpness of the
assumptions on $q$ in
Theorems \ref{thmGNS} and \ref{thm-fs}. In Subsection \ref{sec-Es}, we prove that, for each
fixed $s\in(0,1)$, the assumptions on $q$ in both parts of Theorem
\ref{thm-fs} are sharp. In Subsection \ref{sec-cG},
we show that the assumption
$n(\frac{1}{p_Y}-\frac{1}{q})<k$
in Theorem \ref{thmGNS} is sharp. Finally, in Subsection
\ref{sec-cqpx}, even when this assumption is
satisfied, we prove that estimate \eqref{eq-gns2} may fail if $q\ge p_X$.

We first establish the following lower estimate
for difference operators, which is frequently used in this
section.

\begin{lemma}\label{lem-sharp-tail}
Let $k\in\mathbb{N}$ and $q\in[1,\infty)$. There exists a nonzero
function $\phi\in C_{\rm c}^{\infty}(B({\bf 0},2))$ such that, for any
$s\in(0,1)$ and $x\in\mathbb{R}^n$ with $|x|\ge4k$,
\begin{align}\label{eq-sharp-tail}
\mathfrak D_q^{s,k}(\phi)(x)
&\gtrsim
|x|^{-\frac nq-sk},
\end{align}
where the implicit positive constant is independent of $s$ and $x$.
\end{lemma}

\begin{proof}
Let $\phi\in C_{\rm c}^{\infty}(B({\bf 0},2))$ be nonnegative and
satisfy $\phi=1$ in $B({\bf 0},1)$. Fix $|x|\ge4k$ and define
\begin{align*}
H_x
:=
\left\{h\in\mathbb{R}^n:x+kh\in B({\bf 0},1)\right\}.
\end{align*}
For any $h\in H_x$, we have
\begin{align*}
\frac{|x|-1}{k}
\le |h|
\le
\frac{|x|+1}{k},
\end{align*}
and hence $|h|\sim|x|$. Moreover, for any
$j\in\{0,\ldots,k-1\}$,
\begin{align*}
|x+jh|
&=
\left|
\frac{k-j}{k}x+\frac{j}{k}(x+kh)
\right|\ge
\frac{k-j}{k}|x|-\frac{j}{k}|x+kh|\ge
\frac{|x|-(k-1)}{k}
>2.
\end{align*}
Thus, $\phi(x+jh)=0$ for $j\in\{0,\ldots,k-1\}$ and
$\phi(x+kh)=1$, which imply that
$|\Delta_h^k\phi(x)|=1.$
Since $|H_x|=k^{-n}|B({\bf 0},1)|$, we deduce that
\begin{align*}
\mathfrak D_q^{s,k}(\phi)(x)
&\ge
\left[
\int_{H_x}\frac{dh}{|h|^{n+skq}}
\right]^{\frac1q}
\gtrsim
|x|^{-\frac nq-sk},
\end{align*}
which shows \eqref{eq-sharp-tail} and completes the proof of Lemma
\ref{lem-sharp-tail}.
\end{proof}

\subsection{The Assumption in Theorem \ref{thm-fs} for Each
\texorpdfstring{$s$}{s}}\label{sec-Es}

\begin{proposition}\label{prop-sharp-fixed-s5}
Let $k\in\mathbb{N}$, $s\in(0,1)$,
$p_1\in(1,\infty]$, $p_2\in(1,\infty)$,
$p_s\in(1,\infty)$ satisfy
$\frac1{p_s}
=
\frac{1-s}{p_1}+\frac{s}{p_2},$
and
$q\in[1,\infty)$.
If
\begin{align}\label{eq-sharp-fixed-condition}
n\left(\frac1{p_s}-\frac1q\right)
\ge sk,
\end{align}
then there exists a nonzero function
$\phi\in C_{\rm c}^{\infty}$ such that
$\|\mathfrak D_q^{s,k}(\phi)\|_{L^{p_s}}
=
\infty.$
Consequently, the assumption on $q$ in Theorem
\ref{thm-fs}{\rm(i)} is sharp in general. Taking $p_1=\infty$ also
proves that the assumption on $q$ in Theorem
\ref{thm-fs}{\rm(ii)} is sharp in general.
\end{proposition}

\begin{proof}
Let $\phi$ be the function as in Lemma \ref{lem-sharp-tail}. From
\eqref{eq-sharp-tail}, we infer that
\begin{align*}
\left\|\mathfrak D_q^{s,k}(\phi)\right\|_{L^{p_s}}^{p_s}
&\gtrsim
\int_{|x|\ge4k}
|x|^{-p_s(\frac nq+sk)}\,dx\sim
\int_{4k}^{\infty}
r^{-1-p_s(sk+\frac nq-\frac n{p_s})}\,dr.
\end{align*}
Condition \eqref{eq-sharp-fixed-condition} is equivalent to
$
sk+\frac nq-\frac n{p_s}\le0.$
Therefore, the last integral diverges and hence
\begin{align*}
\left\|\mathfrak D_q^{s,k}(\phi)\right\|_{L^{p_s}}
=
\infty.
\end{align*}
On the other hand, since
$\phi\in C_{\rm c}^{\infty}$, it follows that
$
\|\phi\|_{L^{p_1}}
+
\|\nabla^k\phi\|_{L^{p_2}}
<\infty.$
Thus, the right-hand side of the estimate in Theorem
\ref{thm-fs}{\rm(i)} is finite, whereas its left-hand side is infinite.
This shows the sharpness assertion for Theorem
\ref{thm-fs}{\rm(i)}.

Finally, if $p_1=\infty$, then
$
\frac1{p_s}=\frac{s}{p_2}$
and $L^{p_s}=L^{\frac{p_2}{s}}.$
Moreover,
$\phi\in\operatorname{BMO}$ and
$\nabla^k\phi\in L^{p_2}$. Hence, the right-hand side of
the estimate in Theorem \ref{thm-fs}{\rm(ii)} is finite, while its
left-hand side is again infinite whenever
$n(\frac{s}{p_2}-\frac1q)\ge sk.$
This then completes the proof of Proposition
\ref{prop-sharp-fixed-s5}.
\end{proof}

\subsection{The Assumption in Theorem \ref{thmGNS} as \texorpdfstring{$s\to1^-$}{s to 1}
}\label{sec-cG}

\begin{proposition}\label{prop-sharp-fixed-s}
Let $k\in\mathbb{N}$,
$p_1\in(1,\infty]$, $p_2\in(1,\infty)$, and
$q\in[1,\infty)$.
For any  $s\in(0,1)$, let
$p_s\in(1,\infty)$ satisfy
$\frac1{p_s}
=
\frac{1-s}{p_1}+\frac{s}{p_2}$.
If
\begin{align*}
n\left(\frac1{p_2}-\frac1q\right)
&\ge k,
\end{align*}
then there exists a nonzero function
$\phi\in C_{\rm c}^{\infty}$ for which there do not exist
$C\in(0,\infty)$ and $s_0\in(0,1)$ such that, for every
$s\in(s_0,1)$,
\begin{align}\label{eq-sharp-near-one-fail}
\left\|\mathfrak D_q^{s,k}(\phi)\right\|_{L^{p_s}}
&\le
C(1-s)^{-\frac1q}
\|\phi\|_{L^{p_1}}^{1-s}
\left\|\nabla^k\phi\right\|_{L^{p_2}}^s.
\end{align}
In particular, when $p_1=\infty$, the same conclusion holds with
$\|\phi\|_{L^\infty}^{1-s}$ replaced by
$\|\phi\|_{\operatorname{BMO}}^{1-s}$.
\end{proposition}

\begin{proof}
Let $\phi$ be as in Lemma \ref{lem-sharp-tail}. For any $s\in(0,1)$,
let
\begin{align*}
\eta_s
&:=
sk+\frac nq-\frac n{p_s}.
\end{align*}
By \eqref{eq-sharp-tail}, for any $s\in(0,1)$,
\begin{align}\label{eq-sharp-near-one-integral}
\left\|\mathfrak D_q^{s,k}(\phi)\right\|_{L^{p_s}}^{p_s}
&\gtrsim
\int_{4k}^{\infty}r^{-1-p_s\eta_s}\,dr,
\end{align}
where the implicit positive constant is independent of $s$. We now consider
the following three cases.

\emph{Case 1: $n(\frac1{p_2}-\frac1q)>k$.}
In this case,
$
\eta_1
=
k+\frac nq-\frac n{p_2}<0.$
From the continuity of $s\mapsto\eta_s$, we deduce that
$\eta_s<0$ for every $s$ sufficiently close to $1$. This, together
with \eqref{eq-sharp-near-one-integral}, further implies that
\begin{align}\label{eq-sharp-near-one-growth-2}
\left\|\mathfrak D_q^{s,k}(\phi)\right\|_{L^{p_s}}
&=\infty
\end{align}
for every such $s$. On the other hand, since
$\phi\in C_{\rm c}^{\infty}$,
$
\|\phi\|_{L^{p_1}}^{1-s}
\|\nabla^k\phi\|_{L^{p_2}}^s
<\infty$
for every $s\in(0,1)$. Hence,
\eqref{eq-sharp-near-one-fail} cannot hold, which gives the desired
conclusion in this case.

\emph{Case 2: $n(\frac1{p_2}-\frac1q)=k$ and $p_1\le q$.}
In this case, for any $s\in(0,1)$,
\begin{align*}
\eta_s
&=
n(1-s)\left(\frac1q-\frac1{p_1}\right)
\le0.
\end{align*}
Combining this with \eqref{eq-sharp-near-one-integral}, we find that
\eqref{eq-sharp-near-one-growth-2} holds
for every $s\in(0,1)$. Since the right-hand side of
\eqref{eq-sharp-near-one-fail} is finite for this fixed function
$\phi$, the estimate cannot hold. This gives the desired conclusion
in this case.

\emph{Case 3: $n(\frac1{p_2}-\frac1q)=k$ and $p_1>q$.}
In this case, for any $s\in(0,1)$,
\begin{align}\label{eq-sharp-near-one-growth-1}
\eta_s
&=
n(1-s)\left(\frac1q-\frac1{p_1}\right)>0.
\end{align}
Thus, from \eqref{eq-sharp-near-one-integral}, it follows that, for any
$s\in(0,1)$,
\begin{align*}
\left\|\mathfrak D_q^{s,k}(\phi)\right\|_{L^{p_s}}^{p_s}
&\gtrsim
\frac{(4k)^{-p_s\eta_s}}{p_s\eta_s}.
\end{align*}
Since $p_s\to p_2$ as $s\to1^-$, from
\eqref{eq-sharp-near-one-growth-1}, we infer that
$p_s\eta_s
\sim 1-s$
for every $s$ sufficiently close to $1$. Hence,
\begin{align}\label{eq-sharp-near-one-growth}
\left\|\mathfrak D_q^{s,k}(\phi)\right\|_{L^{p_s}}
&\gtrsim
(1-s)^{-\frac1{p_s}}
\end{align}
for every $s$ sufficiently close to $1$.

On the other hand, since $\phi\in C_{\rm c}^{\infty}$ is fixed, it follows that
\begin{align*}
\sup_{s\in(0,1)}
\|\phi\|_{L^{p_1}}^{1-s}
\left\|\nabla^k\phi\right\|_{L^{p_2}}^s
&<\infty.
\end{align*}
Therefore, using \eqref{eq-sharp-near-one-growth} and
$
\lim_{s\to 1^-}\frac1{p_s}
=
\frac1{p_2}
=
\frac1q+\frac kn,$
we obtain
\begin{align*}
(1-s)^{\frac1q}
\frac{
\|\mathfrak D_q^{s,k}(\phi)\|_{L^{p_s}}
}{
\|\phi\|_{L^{p_1}}^{1-s}
\|\nabla^k\phi\|_{L^{p_2}}^s
}
&\gtrsim
(1-s)^{\frac1q-\frac1{p_s}}\to\infty
\end{align*}
as
$s\to1^-$. This proves the desired conclusion in this case.

Finally, when $p_1=\infty$, we have
$
\|\phi\|_{\operatorname{BMO}}
\lesssim
\|\phi\|_{L^\infty}.$
Thus, an estimate with the $\operatorname{BMO}$ norm would imply
\eqref{eq-sharp-near-one-fail} with the $L^\infty$ norm. The BMO
assertion hence follows, which then completes the proof of Proposition
\ref{prop-sharp-fixed-s}.
\end{proof}

\subsection{The Assumption
\texorpdfstring{$q<p_X$}{q less than pX} in Theorem \ref{thmGNS}}
\label{sec-cqpx}

We next show that the strict assumption $q<p_X$ associated with
\eqref{eq-gns2} cannot be relaxed in general. More precisely, we construct
explicit {\rm BBF} spaces for which \eqref{eq-gns2} fails, although the
assumption considered in the preceding subsection is satisfied.

In the endpoint example below, we use the multiplier-weight convention;
that is, for any $q\in(0,\infty)$, any $f\in L^0$, and any positive measurable $\omega$, define
\begin{align*}
\|f\|_{L^q(\omega)}
:=
\left[
\int_{\mathbb{R}^n}|f(x)|^q\omega(x)^q\,dx
\right]^{\frac1q}.
\end{align*}
Thus, in the standard notation for weighted Lebesgue spaces, the
corresponding weight is $v:=\omega^q$. The
Muckenhoupt classes are understood in their usual sense; see,
for instance, \cite[Chapter 7]{Gra14}.

\begin{proposition}\label{prop-endpoint-counterexample}
Let $k\in\mathbb{N}$, $q\in(1,\infty)$, and $r\in(1,q]$ satisfy
\begin{align}\label{eq-sharp-r-range}
n\left(\frac1r-\frac1q\right)<k.
\end{align}
Fix $\beta\in(0,\infty)$ and let
$\omega_\beta(x):=[\log(e+|x|)]^\beta$ for every
$x\in\mathbb{R}^n$. Define
\begin{align*}
X
:=
\begin{cases}
L^r& \text{if }r<q,\\
L^q(\omega_\beta)& \text{if }r=q.
\end{cases}
\end{align*}
Then $X$ is a {\rm BBF} space, $p_X=r$, and
\begin{align}\label{eq-sharp-dimension-holds}
n\left(\frac1{p_X}-\frac1q\right)<k.
\end{align}
However, there is no positive constant $C$, independent of $s$ and
$f$, such that
\begin{align}\label{eq-endpoint-fail}
\left\|\mathfrak D_q^{s,k}(f)\right\|_X
\le
C[s(1-s)]^{-\frac1q}
\|f\|_X^{1-s}
\left\|\nabla^k f\right\|_X^s
\end{align}
holds for every $s\in(0,1)$ and
$f\in X\cap\dot W^{k,X}$.
\end{proposition}

\begin{proof}
We first verify the assertions concerning $X$ and $p_X$.

Assume that $r\in(1,q)$. Then it is easy to prove that
$X=L^r$ is a {\rm BBF} space and
$p_X=r$. Hence, \eqref{eq-sharp-dimension-holds} follows immediately
from \eqref{eq-sharp-r-range}.

We next consider $r=q$. Let
$v_\beta:=\omega_\beta^q=[\log(e+|\cdot|)]^{\beta q}$.
By the standard criterion for logarithmic radial weights
(see, for instance,
\cite[Chapter 7]{Gra14}), we find that
$v_\beta\in\bigcap_{t\in(1,\infty)} A_t$. Consequently,
from \cite[Section 7.1]{SHYY17}, it follows that
$X=L^q(\omega_\beta)$ is a {\rm BBF} space.
For every $p\in(1,q)$, we have
$X^{\frac1p}=L^{\frac qp}(\omega_\beta^p)$. In the standard weighted
Lebesgue notation, the corresponding weight is
\begin{align*}
(\omega_\beta^p)^{\frac qp}
=
\omega_\beta^q
=
v_\beta
\in A_{\frac qp}.
\end{align*}
Thus, the Hardy--Littlewood maximal operator
$M$ is bounded on
$X^{\frac1p}$. If $p>q$, then $\frac{q}{p}<1$, and hence
$X^{\frac1p}=L^{\frac qp}(\omega_\beta^p)$ does not satisfy the
triangle inequality and therefore is not a {\rm BBF} space.
It follows that $p_X=q$. Therefore,
$n(\frac1{p_X}-\frac1q)=0<k$ and
\eqref{eq-sharp-dimension-holds} also holds in this case.

We now show the failure of \eqref{eq-endpoint-fail}. Assume first that
$r\in(1,q)$. Choose
\begin{align*}
0<s<
\frac nk\left(\frac1r-\frac1q\right).
\end{align*}
Such a choice is possible by \eqref{eq-sharp-r-range}, and it gives
$n(\frac1r-\frac1q)>sk$. Applying Proposition
\ref{prop-sharp-fixed-s5} with $p_1=p_2=r$, we find a nonzero function
$\phi\in C_{\rm c}^{\infty}$ such that
\begin{align*}
\left\|\mathfrak D_q^{s,k}(\phi)\right\|_{L^r}
=
\infty.
\end{align*}
On the other hand, both $\|\phi\|_{L^r}$ and
$\|\nabla^k\phi\|_{L^r}$ are finite. Thus, the left-hand side of
\eqref{eq-endpoint-fail} is infinite, while its right-hand side is
finite, and hence \eqref{eq-endpoint-fail} fails.

It remains to consider $r=q$. Let $\phi$ be as in Lemma
\ref{lem-sharp-tail}. Since $\phi$ and $\nabla^k\phi$ have compact
support, both belong to $X$. Moreover, from \eqref{eq-sharp-tail} and
$\rho^{-skq}\ge e^{-kq}$ when
$\rho\le e^{1/s}$, we deduce that,
for every sufficiently small $s\in(0,1)$,
\begin{align*}
\left\|\mathfrak D_q^{s,k}(\phi)\right\|_X^q
&\gtrsim
\int_{4k}^{\infty}
\rho^{-1-skq}
[\log(e+\rho)]^{\beta q}\,d\rho\gtrsim
\int_{4k}^{e^{1/s}}
\frac{(\log\rho)^{\beta q}}{\rho}\,d\rho
\gtrsim
s^{-\beta q-1}.
\end{align*}
On the other hand,
$\|\phi\|_X^{1-s}\|\nabla^k\phi\|_X^s$ is bounded uniformly with
respect to $s\in(0,1)$. Therefore,
\begin{align*}
[s(1-s)]^{\frac1q}
\frac{
\left\|\mathfrak D_q^{s,k}(\phi)\right\|_X
}{
\|\phi\|_X^{1-s}
\|\nabla^k\phi\|_X^s
}
\gtrsim
s^{-\beta}(1-s)^{\frac1q}
\longrightarrow
\infty
\end{align*}
as $s\to0^+$. Hence, no constant independent of $s$ and $f$ can make
\eqref{eq-endpoint-fail} hold when $r=q$.
This then completes the proof of Proposition
\ref{prop-endpoint-counterexample}.
\end{proof}

\section{Applications to Specific Function Spaces}\label{sec-app}

We now give concrete applications of the preceding results to several
classical and modern function space scales. In Subsection
\ref{sec-6.1}, we consider weighted Lebesgue, Morrey, and
Bourgain--Morrey spaces, verify the relevant assumptions,
and summarize the resulting fractional Gagliardo--Nirenberg, BBM, and
MS type conclusions. In Subsection \ref{sec-lorentz}, we specialize
our results to Lorentz and Orlicz spaces; besides obtaining explicit
off-diagonal inequalities and endpoint asymptotic formulas, we
characterize the optimal rearrangement invariant target
spaces for both the
fractional and the corresponding integer order inequalities. In
Subsection \ref{sec-bmo}, we exploit the Lebesgue structure of the
derivative space to establish a sharpened BMO endpoint estimate with
improved dependence on the parameters. Finally, in Subsection
\ref{sec-morespace}, we indicate further applications to variable
Lebesgue, mixed-norm Lebesgue, Orlicz-slice, and generalized Herz
spaces, highlighting both the recovery of known diagonal results and
the new off-diagonal consequences.

\subsection{Weighted Lebesgue, Morrey, and Bourgain--Morrey Spaces}
\label{sec-6.1}

In this subsection, we apply the preceding results to weighted Lebesgue,
Morrey, and Bourgain--Morrey type spaces. We first recall their
definitions and the structural properties needed to verify the
hypotheses of Theorems \ref{thmGNS}, \ref{thm-fs}, \ref{thmBBM}, and
\ref{thmMS}, and then summarize the resulting inequalities and
asymptotic estimates. Since these spaces are generally not rearrangement
invariant, the application of Theorem \ref{thm-fs} here concerns its norm
inequalities rather than the characterization of optimal rearrangement
invariant target spaces.

{\bf Weighted Lebesgue spaces.} \quad A weight is nonnegative locally integrable function on $\mathbb{R}^n$ which is
positive almost everywhere. In this section, we use the multiplier-weight convention,
i.\,e., for any $p\in(0,\infty)$, any weight $\omega$, and any $f\in L^0$,
\begin{align*}
\|f\|_{L^p(\omega)}
:=
\left[
\int_{\mathbb{R}^n}|f(x)|^p\omega(x)^p\,dx
\right]^{\frac1p}.
\end{align*}
Thus, in the standard notation for weighted Lebesgue spaces, the
corresponding weight is $w:=\omega^p$. The
Muckenhoupt classes are understood in their usual sense; see
\cite[Chapter 7]{Gra14}. In particular, for $t\in(1,\infty)$,
$w\in A_t$ if and only if
\begin{align*}
\sup_Q
\left[\fint_Qw(x)\,dx\right]
\left[\fint_Qw(x)^{-\frac1{t-1}}\,dx\right]^{t-1}
<\infty,
\end{align*}
where the supremum is taken over all cubes $Q\subset\mathbb{R}^n$.
For any $w\in A_\infty:=\bigcup_{t\in(1,\infty)}A_t$, define
\begin{align*}
I(w):=\inf\{t\in[1,\infty):w\in A_t\}.
\end{align*}
Let $p\in(1,\infty)$ and assume $\omega^p\in A_p$. Then
$L^p(\omega)$ is a {\rm BBF} space with an absolutely continuous norm;
see, for instance, \cite[Section 7.1]{SHYY17}. Moreover, for any
$\tau\in(1,p)$,
$[L^p(\omega)]^{\frac1\tau}
=
L^{\frac p\tau}(\omega^\tau).$
Thus, by the Muckenhoupt theorem
\cite[Theorem 7.1.9]{Gra14} and the definition of $I(\omega^p)$, we find that
\begin{align*}
p_{L^p(\omega)}
=
\frac{p}{I(\omega^p)}.
\end{align*}
If $p_i\in(1,\infty)$ and
$\omega_i^{p_i}\in A_{p_i}$ for $i\in\{1,2\}$,
then, from \cite[Lemma 5]{KLM14} or
\cite[Proposition 1]{KM03}, it follows that
\begin{align}\label{eq-weighted-product}
[L^{p_1}(\omega_1)]^{1-s}[L^{p_2}(\omega_2)]^s
=
L^{p_s}(\omega_s)
\end{align}
with equivalent norms, where
\begin{align*}
\frac1{p_s}
=
\frac{1-s}{p_1}+\frac{s}{p_2}\quad\text{and}\quad
\omega_s
=
\omega_1^{1-s}\omega_2^s.
\end{align*}

{\bf Morrey spaces.}\quad
Let $p\in[1,\infty)$ and $u\in[p,\infty)$. The \emph{Morrey space}
$\mathcal{M}_p^u$ is defined to be the set of all
$f\in L^p_{\rm loc}$ such that
\begin{align*}
\|f\|_{\mathcal{M}_p^u}
:=
\sup_Q
|Q|^{\frac1u-\frac1p}
\left[
\int_Q|f(x)|^p\,dx
\right]^{\frac1p}
<\infty,
\end{align*}
where the supremum is taken over all cubes $Q\subset\mathbb{R}^n$.
These spaces were introduced by Morrey \cite{Mor38}; see also the monographs
\cite{a15,SDGD20I,SDGD20II}. If $p>1$, then
$\mathcal{M}_p^u$ is a {\rm BBF} space and
$p_{\mathcal{M}_p^u}=p$. When $p=u$, it reduces to $L^p$; when
$p<u$, its norm is not absolutely continuous.

{\bf Bourgain--Morrey spaces.}\quad
Let the \emph{symbol $\mathcal{D}$} denote the family of all dyadic cubes in $\mathbb{R}^n$.
For $p,u,r\in(0,\infty]$, the \emph{Bourgain--Morrey} space
$\mathcal{M}_{p,r}^u(\mathbb{R}^n)$ is defined to be the set of all
$f\in L^p_{\rm loc}$ such that
\begin{align*}
\|f\|_{\mathcal{M}_{p,r}^u}
:=
\left[
\sum_{Q\in\mathcal{D}}
\left(
|Q|^{\frac1u-\frac1p}
\|f{\bf 1}_Q\|_{L^p}
\right)^r
\right]^{\frac1r}
<\infty
\end{align*}
with the usual modification made when $r=\infty$; see
\cite{Bou91,Mas-arXiv,HNSH22}. This space is nontrivial only in the ranges
$p<u<r<\infty$ and $p\le u\le r=\infty$. Moreover,
$\mathcal{M}_{p,\infty}^u=\mathcal{M}_p^u$ with equivalent norms.
When $p\in (1,\infty]$, $\mathcal{M}_{p,r}^u$ is a {\rm BBF} space and
$p_{\mathcal{M}_{p,r}^u}=p$. If $r<\infty$, its norm is absolutely
continuous, whereas the case $r=\infty$ contains the genuine Morrey
spaces and does not have this property in general.

Now, we briefly comment on the
Calder\'on--Lozanovski\u{\i} products of these Morrey type spaces. Let
$i\in\{1,2\}$, $p_i\in[1,\infty)$, and $u_i\in[p_i,\infty)$.
Assume that $p_s$ and $u_s$ satisfy
\begin{align*}
\frac1{p_s}
=
\frac{1-s}{p_1}+\frac{s}{p_2}
\quad\text{and}\quad
\frac1{u_s}
=
\frac{1-s}{u_1}+\frac{s}{u_2}.
\end{align*}
For Morrey spaces, it is known that
\begin{align*}
(\mathcal M_{p_1}^{u_1})^{1-s}
(\mathcal M_{p_2}^{u_2})^s
\hookrightarrow
\mathcal M_{p_s}^{u_s}.
\end{align*}
Moreover, this embedding is an equality, with equivalent norms, when
$p_1/u_1=p_2/u_2$, while it is strict when
$p_1/u_1\ne p_2/u_2$; see
\cite[Theorem 2.5]{YSY16} and the references therein.
In addition, as a special case of \cite[Theorem 4.3]{ZSTYY23}, the
Calder\'on--Lozanovski\u{\i} products of Bourgain--Morrey spaces can be
identified explicitly under suitable compatibility conditions on the
parameters. In the general case an explicit description of
these products is still unknown.

Finally, we also verify the assumption on the centered ball average operators
appearing in Theorem \ref{thmMS} for these spaces. In the weighted Lebesgue case, if
$
q
<
\frac{p}{I(\omega^p)},$
then $\frac pq>I(\omega^p)$ and hence $\omega^p\in A_{\frac pq}$. Since
$
[L^p(\omega)]^{\frac1q}
=
L^{\frac pq}(\omega^q),$
the Muckenhoupt theorem \cite[Theorem 7.1.9]{Gra14} and the pointwise estimate
$\mathcal B_r(f)\le M(f)$ imply that the centered ball average operators
are uniformly bounded on $[L^p(\omega)]^{\frac1q}$.
For Morrey spaces, if $q\in[1,p]$, then
$
[\mathcal M_p^u]^{\frac1q}
=
\mathcal M_{\frac pq}^{\frac uq}.$
The translation invariance of the Morrey norm and Minkowski's inequality
imply that, for any appropriate $f$ and any $r\in(0,\infty)$,
\begin{align*}
\|\mathcal B_r(f)\|_{\mathcal M_{\frac pq}^{\frac uq}}
\le
\|f\|_{\mathcal M_{\frac pq}^{\frac uq}}.
\end{align*}
For Bourgain--Morrey spaces, if $q<p$, then
$
[\mathcal M_{p,r}^u]^{\frac1q}
=
\mathcal M_{\frac pq,\frac rq}^{\frac uq}.$
Combining this
and the boundedness of the Hardy--Littlewood maximal operator on
Bourgain--Morrey spaces (see \cite{HNSH22,HLY23}), we further obtain
the required uniform boundedness of the
centered ball average operators. Therefore, the applications of Theorem
\ref{thmMS} recorded below are understood under these corresponding
conditions for both spaces involved.

{\bf Conclusions.}\quad
To simplify the presentation, we use the following notation in the table below:
the \emph{symbol} $\mathbf{G}$ denotes
that the estimates \eqref{eq-gns1} and
\eqref{eq-gns2} hold,
the \emph{symbol} $\mathbf{F}$ denotes that norm inequality
\eqref{eq-XYss} holds,
the \emph{symbol} $\mathbf{B}$ denotes that Theorem \ref{thmBBM} holds, and that
the \emph{symbol} $\mathbf{S}$ denotes that Theorem \ref{thmMS} holds. All entries are understood
under the hypotheses of the corresponding theorem, including the
conditions on the centered ball average operators verified above.
Since the spaces in this subsection are generally not rearrangement
invariant, the \emph{symbol} $\mathbf{F}$ refers only to
that \eqref{eq-XYss} holds, and not to
that the
optimal-target assertion in Theorem \ref{thm-fs} holds.
For $i\in\{1,2\}$, let $X_i$ denote the space containing the $i$-th
factor, with $X_1$ containing $f$ and $X_2$ containing
$\nabla^k f$. We also write
$\mathcal C_s(X_1,X_2):=X_1^{1-s}X_2^s$.

\sbox0{\ref{eq-weighted-product}}
\begin{center}
\small
\setlength{\tabcolsep}{4pt}
\renewcommand{\arraystretch}{1.25}
\begin{tabularx}{\textwidth}{|>{\raggedright\arraybackslash}p{0.19\textwidth}|>{\raggedright\arraybackslash}p{0.13\textwidth}|>{\raggedright\arraybackslash}p{0.18\textwidth}|>{\raggedright\arraybackslash}X|}
\hline
\textbf{Scale} & \textbf{Lower index} &
\textbf{Target in $\mathbf{F}$} & \textbf{Conclusions} \\
\hline
Weighted Lebesgue
& $p_{X_i}=\frac{p_i}{I(\omega_i^{p_i})}$
& $L^{p_s}(\omega_s)$, by \eqref{eq-weighted-product}
& $\mathbf{G}$, $\mathbf{F}$, $\mathbf{B}$, and $\mathbf{S}$ \\
\hline
Morrey
& $p_{X_i}=p_i$
& $\mathcal C_s(\mathcal M_{p_1}^{u_1},\mathcal M_{p_2}^{u_2})$
& $\mathbf{G}$, $\mathbf{F}$, and $\mathbf{S}$; $\mathbf{B}$ only when the relevant norms are absolutely continuous \\
\hline
Bourgain--Morrey
& $p_{X_i}=p_i$
& $\mathcal C_s(\mathcal M_{p_1,r_1}^{u_1},\mathcal M_{p_2,r_2}^{u_2})$
& $\mathbf{G}$, $\mathbf{F}$, and $\mathbf{S}$; $\mathbf{B}$ when $r_1,r_2<\infty$ \\
\hline
\end{tabularx}
\end{center}

\begin{remark}
\begin{enumerate}[{\rm(i)}]
\item When $p_1=p_2$ and $\omega_1=\omega_2$, the weighted
Gagliardo--Nirenberg and BBM type results reduce to, in the common range
and with $\Omega=\mathbb{R}^n$, the corresponding conclusions in
\cite[Theorems 1.6 and 1.8(i)]{HLYY-ineq}. In the diagonal first order
case, the weighted MS type estimate also improves the corresponding
weighted Lebesgue space case of \cite[Theorem 2.16(i)]{PYYZ24}. The off-diagonal
weighted Gagliardo--Nirenberg, BBM, and MS type results, as well as their
higher order MS extensions, appear to be new.

\item For Bourgain--Morrey spaces, the diagonal
Gagliardo--Nirenberg result reduces to the corresponding conclusion in
\cite[Theorem 7.3]{HLYY-ineq}. When the relevant norms are absolutely
continuous, the diagonal BBM type result also reduces to the corresponding
part of that theorem. In the diagonal first order case, the MS type
estimate improves the corresponding Morrey and Bourgain--Morrey cases of
\cite[Theorem 2.16(i)]{PYYZ24}. The off-diagonal
Gagliardo--Nirenberg, BBM, and MS type results, together with the
higher order MS extensions, appear to be new.
\end{enumerate}
\end{remark}

\subsection{Lorentz and Orlicz Spaces}\label{sec-lorentz}

In this subsection, we consider Lorentz and Orlicz spaces, which are two classical
scales of rearrangement invariant Banach function spaces. Their lower
indices and Calder\'on--Lozanovski\u{\i} products can be described explicitly,
so the fractional inequalities and endpoint asymptotic results established
above admit concrete formulations in these settings. Moreover, Theorems
\ref{thm-fs} and \ref{thmOP} prove that the resulting
Calder\'on--Lozanovski\u{\i} spaces are precisely the optimal rearrangement
invariant targets. We also give the corresponding integer order
characterizations. In the Lorentz case, the latter gives an answer
to \cite[Question 2.6]{LRS23}.

{\bf Lorentz spaces.}\quad
For $r\in(0,\infty)$ and $\mu\in(0,\infty]$, the \emph{Lorentz space}
$L^{r,\mu}$ is defined to be the set of all $f\in L^0$ such that
\begin{align*}
\|f\|_{L^{r,\mu}}
:=
\left\{
\int_0^\infty
\left[t^{\frac1r}f^*(t)\right]^\mu
\frac{dt}{t}
\right\}^{\frac1\mu}
<\infty,
\end{align*}
with the usual modification made when $\mu=\infty$, where $f^*$ denotes the
non-increasing rearrangement of $f$ in \eqref{eq-fstar}; see \cite{Lor50,Lor51}. When
$r,\mu\in(1,\infty)$, this functional is equivalent to a Banach norm,
$L^{r,\mu}$ has an absolutely continuous norm, and
$p_{L^{r,\mu}}=\min\{r,\mu\}$.

Let $r_1,\mu_1,r_2,\mu_2\in(1,\infty)$ and $s\in(0,1)$. Assume that
$r,\mu\in(1,\infty)$ satisfy
\begin{align*}
\frac1r
=
\frac{1-s}{r_1}+\frac{s}{r_2}
\quad\text{and}\quad
\frac1\mu
=
\frac{1-s}{\mu_1}+\frac{s}{\mu_2}.
\end{align*}
Then
\begin{align}\label{eq-lor-product}
[L^{r_1,\mu_1}]^{1-s}[L^{r_2,\mu_2}]^s
=
L^{r,\mu}
\end{align}
with equivalent norms.

We now record the consequences of the main results in this scale. Let
$X:=L^{r_1,\mu_1}$ and $Y:=L^{r_2,\mu_2}$. Then
$p_X=\min\{r_1,\mu_1\}$ and $p_Y=\min\{r_2,\mu_2\}$. Therefore, if
\begin{align*}
n\left(
\frac{1}{\min\{r_2,\mu_2\}}-\frac1q
\right)
<k,
\end{align*}
Theorem \ref{thmGNS}{\rm(i)} gives estimate \eqref{eq-gns1} with
target $L^{r,\mu}$. If, in addition,
$q<\min\{r_1,\mu_1\}$, then estimate \eqref{eq-gns2} also holds with
the same target.

The BMO endpoint estimate in Theorem \ref{thmGNS}{\rm(ii)} has target
$
Y^{\frac1s}
=
L^{\frac{r_2}{s},\frac{\mu_2}{s}}.
$
For each fixed $s\in(0,1)$, Theorem \ref{thm-fs}{\rm(i)} gives
\eqref{eq-XYss} with target $L^{r,\mu}$ whenever
\begin{align}\label{eq-lor-fixed-condition}
n\left(
\frac{1-s}{\min\{r_1,\mu_1\}}
+
\frac{s}{\min\{r_2,\mu_2\}}
-
\frac1q
\right)
<sk.
\end{align}
Similarly, Theorem \ref{thm-fs}{\rm(ii)} gives the BMO endpoint
estimate with target
$L^{\frac{r_2}{s},\frac{\mu_2}{s}}$ whenever
\begin{align*}
n\left(
\frac{s}{\min\{r_2,\mu_2\}}
-
\frac1q
\right)
<sk,
\end{align*}
and this target is optimal among rearrangement invariant Banach
function spaces.

Since the Lorentz spaces under consideration have absolutely continuous
norms, Theorem \ref{thmBBM} also applies whenever
$n(\frac{1}{\min\{r_2,\mu_2\}}-\frac1q)<k$. Moreover, if
\begin{align*}
q
<
\min\{r_1,\mu_1,r_2,\mu_2\},
\end{align*}
then the relevant convexifications are Banach Lorentz spaces and the
centered ball average operators are uniformly bounded on them.
Consequently, Theorem \ref{thmMS} applies as well.

The preceding conclusions show that $L^{r,\mu}$ is an admissible target.
Theorem \ref{thm-fs}, together with Theorem \ref{thmOP}{\rm(i)}, proves
more generally that an inequality with an arbitrary rearrangement
invariant target $B$ holds if and only if
$L^{r,\mu}\hookrightarrow B$. For targets within the Lorentz scale, this
characterization takes the following explicit form.

\begin{theorem}\label{op-lor}
Let $k\in\mathbb{N}$, $s\in(0,1)$, and
$r_1,\mu_1,r_2,\mu_2,\widetilde{r},\widetilde{\mu}\in(1,\infty)$, and define $r$ and $\mu$ as above.
Let $q\in[1,\infty)$ satisfy \eqref{eq-lor-fixed-condition}. Then
there exists a positive constant $C$ such that
the
inequality
\begin{align}\label{eq-lor-opt}
\left\|\mathfrak D_q^{s,k}(f)\right\|_{L^{\widetilde r,\widetilde\mu}}
\le C
\|f\|_{L^{r_1,\mu_1}}^{1-s}
\left\|\nabla^k f\right\|_{L^{r_2,\mu_2}}^s
\end{align}
holds for every
$f\in L^{r_1,\mu_1}\cap\dot W^{k,L^{r_2,\mu_2}}$ if and only if
$\widetilde r=r$ and $\widetilde\mu\ge\mu$.
\end{theorem}

\begin{proof}
If $\widetilde r=r$ and $\widetilde\mu\ge\mu$, the sufficiency follows
from \eqref{eq-lor-product}, Theorem \ref{thm-fs}, and the monotonicity
of Lorentz spaces with respect to the second index. Conversely, assume
that \eqref{eq-lor-opt} holds. A scaling argument gives
$\widetilde r=r$. By Theorem \ref{thmOP}{\rm(i)},
$L^{r,\mu}\hookrightarrow L^{r,\widetilde\mu}$, which is equivalent to
$\widetilde\mu\ge\mu$. This completes the proof of Theorem
\ref{op-lor}.
\end{proof}

The same conclusion holds in the corresponding integer order inequality.
In particular, the following result includes the unresolved Lorentz case
in \cite[Question 2.6]{LRS23}.

\begin{theorem}\label{op-lor-integer}
Let $1\le j<k$,
$r_1,\mu_1,r_2,\mu_2\in(1,\infty)$, and $r,\mu\in(1,\infty)$ satisfy
\begin{align*}
\frac1r
=
\left(1-\frac jk\right)\frac1{r_1}
+
\frac jk\frac1{r_2}
\quad\text{and}\quad
\frac1\mu
=
\left(1-\frac jk\right)\frac1{\mu_1}
+
\frac jk\frac1{\mu_2}.
\end{align*}
Then there exists a positive constant $C$ such that
\begin{align*}
\|\nabla^j f\|_{L^{\widetilde r,\widetilde\mu}}
\le C
\|f\|_{L^{r_1,\mu_1}}^{1-\frac jk}
\left\|\nabla^k f\right\|_{L^{r_2,\mu_2}}^{\frac jk}
\end{align*}
holds for every
$f\in L^{r_1,\mu_1}\cap\dot W^{k,L^{r_2,\mu_2}}$ if and only if
$\widetilde r=r$ and $\widetilde\mu\ge\mu$.
\end{theorem}

\begin{proof}
The sufficiency follows from the integer order Gagliardo--Nirenberg
inequality for rearrangement invariant Banach function spaces in
\cite[Theorem 1.1]{LRS25} and the Calder\'on--Lozanovski\u{\i} product
formula \eqref{eq-lor-product} with $s=j/k$. For the necessity, a
scaling argument gives $\widetilde r=r$, while Theorem
\ref{thmOP}{\rm(ii)} gives
$L^{r,\mu}\hookrightarrow L^{r,\widetilde\mu}$. Hence
$\widetilde\mu\ge\mu$, which completes the proof of Theorem
\ref{op-lor-integer}.
\end{proof}

{\bf Orlicz spaces.}\quad
A \emph{Young function} is a convex and non-decreasing function
$\Phi:[0,\infty)\to[0,\infty]$ satisfying $\Phi(0)=0$ and
$\lim_{t\to\infty}\Phi(t)=\infty$. The \emph{Orlicz space} $L^\Phi$ is
defined to be the set of all $f\in L^0$ such that
\begin{align*}
\|f\|_{L^\Phi}
:=
\inf\left\{
\lambda\in (0,\infty):
\int_{\mathbb{R}^n}
\Phi\left(\frac{|f(x)|}{\lambda}\right)\,dx
\le1
\right\}<\infty.
\end{align*}
We use the critical lower and upper type indices
\begin{align*}
r^-_\Phi
&:=
\sup\left\{
\alpha\in  (0,\infty):
\Phi(\lambda t)\lesssim\lambda^\alpha\Phi(t)
\text{ for }\lambda\in(0,1],\ t\in  (0,\infty)
\right\}
\end{align*}
and
\begin{align*}
r^+_\Phi
&:=
\inf\left\{
\alpha\in  (0,\infty):
\Phi(\lambda t)\lesssim\lambda^\alpha\Phi(t)
\text{ for }\lambda\in[1,\infty),\ t\in  (0,\infty)
\right\}.
\end{align*}
If $1<r^-_\Phi\le r^+_\Phi<\infty$, then $L^\Phi$ is a
rearrangement invariant {\rm BBF} space with an absolutely continuous
norm and $p_{L^\Phi}=r^-_\Phi$; see
\cite{DFMN21,NS14,RR02}.

Let $\Phi_1$ and $\Phi_2$ satisfy these assumptions. For
$s\in(0,1)$, let $\Phi_s$ be a Young function satisfying that,
for any $t\in(0,\infty)$,
\begin{align*}
\Phi_s^{-1}(t)
\sim
\left[\Phi_1^{-1}(t)\right]^{1-s}
\left[\Phi_2^{-1}(t)\right]^s
\end{align*}
with the positive equivalence constants independent of $s$.
Then
\begin{align}\label{eq-orlicz-product}
\left(L^{\Phi_1}\right)^{1-s}\left(L^{\Phi_2}\right)^s
=
L^{\Phi_s}
\end{align}
with equivalent norms.

We next state the consequences of the main results in the Orlicz scale.
Let $X:=L^{\Phi_1}$ and $Y:=L^{\Phi_2}$. Then
$p_X=r^-_{\Phi_1}$ and $p_Y=r^-_{\Phi_2}$. Hence, if
$n(
\frac1{r^-_{\Phi_2}}-\frac1q
)
<k,$
Theorem \ref{thmGNS}{\rm(i)} gives estimate \eqref{eq-gns1} with
target $L^{\Phi_s}$. If, in addition, $q<r^-_{\Phi_1}$, then
\eqref{eq-gns2} also holds with the same target.

For the BMO endpoint, let $\Psi_s$ be a Young function satisfying that,
for any $t\in(0,\infty)$,
\begin{align*}
\Psi_s^{-1}(t)
\sim
\left[\Phi_2^{-1}(t)\right]^s
\end{align*}
with the positive equivalence constants independent of $s$.
Then
$(L^\infty)^{1-s}(L^{\Phi_2})^s
=
L^{\Psi_s}$
with equivalent norms. Thus, Theorem \ref{thmGNS}{\rm(ii)} gives the
corresponding BMO endpoint estimate with target $L^{\Psi_s}$.

For each fixed $s\in(0,1)$, Theorem \ref{thm-fs}{\rm(i)} gives
\eqref{eq-XYss} with target $L^{\Phi_s}$ whenever
\begin{align}\label{eq-orlicz-fixed-condition}
n\left(
\frac{1-s}{r^-_{\Phi_1}}
+
\frac{s}{r^-_{\Phi_2}}
-
\frac1q
\right)
<sk.
\end{align}
Theorem \ref{thm-fs}{\rm(ii)} gives the corresponding BMO endpoint
estimate with target $L^{\Psi_s}$ whenever
\begin{align*}
n\left(
\frac{s}{r^-_{\Phi_2}}
-
\frac1q
\right)
<sk,
\end{align*}
and this target is optimal among rearrangement invariant Banach function
spaces.

Theorem \ref{thmBBM} also applies because the Orlicz spaces under
consideration have absolutely continuous norms. Moreover, if
$q<\min\{r^-_{\Phi_1},r^-_{\Phi_2}\}$, then the relevant
convexifications are Banach Orlicz spaces on which the centered ball
average operators are uniformly bounded. Therefore, Theorem
\ref{thmMS} applies under this condition.

The next theorem shows that $L^{\Phi_s}$ is not merely an admissible
target: it characterizes all rearrangement invariant targets for which
the corresponding inequality holds.

\begin{theorem}\label{op-orlicz}
Let $k\in\mathbb{N}$, $s\in(0,1)$, and $q\in[1,\infty)$, and assume that
\eqref{eq-orlicz-fixed-condition} holds. Let $B$ be a rearrangement invariant
Banach function space. Then
there exists a positive constant $C$ such that
\begin{align*}
\left\|\mathfrak D_q^{s,k}(f)\right\|_B
\le C
\|f\|_{L^{\Phi_1}}^{1-s}
\left\|\nabla^k f\right\|_{L^{\Phi_2}}^s
\end{align*}
holds for every
$f\in L^{\Phi_1}\cap\dot W^{k,L^{\Phi_2}}$ if and only if
$L^{\Phi_s}\hookrightarrow B$.
\end{theorem}

\begin{proof}
The sufficiency follows from \eqref{eq-orlicz-product}, Theorem
\ref{thm-fs}, and the assumed embedding. The necessity follows from
Theorem \ref{thmOP}{\rm(i)}, which completes the proof of Theorem
\ref{op-orlicz}.
\end{proof}

For $1\le j<k$, let $\Phi_{j,k}$ be a Young function satisfying that,
for any $t\in(0,\infty)$,
\begin{align*}
\Phi_{j,k}^{-1}(t)
\sim
[\Phi_1^{-1}(t)]^{1-\frac jk}
[\Phi_2^{-1}(t)]^{\frac jk}
\end{align*}
with the positive equivalence constants independent of $j$ and $k$.
The following integer order characterization also holds.

\begin{theorem}\label{op-orlicz-integer}
Let $B$ be a rearrangement invariant Banach function space. Then there exists a positive constant $C$ such that
\begin{align*}
\|\nabla^j f\|_B
\le C
\|f\|_{L^{\Phi_1}}^{1-\frac jk}
\left\|\nabla^k f\right\|_{L^{\Phi_2}}^{\frac jk}
\end{align*}
holds for every
$f\in L^{\Phi_1}\cap\dot W^{k,L^{\Phi_2}}$ if and only if
$L^{\Phi_{j,k}}\hookrightarrow B$.
\end{theorem}

\begin{proof}
The sufficiency follows from \cite[Theorem 1.1]{LRS25} and the
Calder\'on--Lozanovski\u{\i} product formula for Orlicz spaces. The necessity
follows from Theorem \ref{thmOP}{\rm(ii)}, which completes the proof of
Theorem \ref{op-orlicz-integer}.
\end{proof}

\begin{remark}
\begin{enumerate}
\item[{\rm(i)}]
When $r_1=r_2$ and $\mu_1=\mu_2$, the above conclusions for Lorentz spaces
corresponding to Theorems \ref{thmGNS} and \ref{thmBBM} reduce to, in the
common range and with $\Omega=\mathbb{R}^n$, the corresponding results
in \cite[Theorem 7.9]{HLYY-ineq}. In the diagonal first order case, the
MS type estimate for Lorentz spaces improves the corresponding Orlicz space case
of \cite[Theorem 2.16(i)]{PYYZ24}. The off-diagonal fractional
inequalities, the BMO endpoint estimates, the off-diagonal BBM and MS
type results, and the optimal target characterization in Theorem
\ref{op-lor} appear to be new. Moreover, Theorem
\ref{op-lor-integer} with $j=1$ and $k=2$ gives an answer to
\cite[Question 2.6]{LRS23}.

\item[{\rm(ii)}]
When $\Phi_1=\Phi_2$, the above conclusions for Orlicz spaces corresponding to
Theorems \ref{thmGNS} and \ref{thmBBM} reduce to, in the common range and
with $\Omega=\mathbb{R}^n$, the corresponding results in
\cite[Theorem 7.11(i)]{HLYY-ineq}. In the diagonal first order case, the
MS type estimate for Orlicz spaces improves the corresponding Orlicz space case
of \cite[Theorem 2.16(i)]{PYYZ24}. The off-diagonal fractional
inequalities, the BMO endpoint estimates, the off-diagonal BBM and MS
type results, and the optimal target characterizations in Theorems
\ref{op-orlicz} and \ref{op-orlicz-integer} appear to be new.
\end{enumerate}
\end{remark}

\subsection{A Sharpened BMO--Lebesgue Estimate}
\label{sec-bmo}

When the outer space is an ordinary Lebesgue space, the order of integration can be used to
obtain a sharper BMO endpoint estimate than a direct application of Theorem \ref{thmGNS}.
This gives the following consequence.

\begin{theorem}\label{thm-L}
Let $k\in\mathbb N$ and $p\in(1,\infty)$. Then, for any
$s\in(\frac1p,1)$ and $f\in\operatorname{BMO}\cap\dot W^{k,sp}$,
\begin{align}\label{eq-lebg}
\int_{\mathbb R^n}\int_{\mathbb R^n}
\frac{|\Delta_h^k f(x)|^p}{|h|^{n+skp}}\,dh\,dx
\lesssim
\frac{1}{(1-s)(sp-1)}
\|f\|_{\operatorname{BMO}}^{(1-s)p}
\int_{\mathbb R^n}|\nabla^k f(x)|^{sp}\,dx,
\end{align}
where the implicit positive constant is independent of $s$ and $f$.
\end{theorem}

\begin{proof}
Let $f\in {\rm BMO}\cap\dot{W}^{k,sp}$.
If $\|f\|_{\operatorname{BMO}}=0$, then $f$ is constant almost everywhere and the conclusion
is immediate. Thus, we may assume that $\|f\|_{\operatorname{BMO}}\in (0,\infty)$.
By an argument used in the proof of Lemma~\ref{lem-point},
we find that, for any $h\in\mathbb{R}^n\setminus\{{\bf0}\}$ and almost every $x\in \mathbb{R}^n$,
\begin{align}
|\Delta_h^k f(x)|
\lesssim
\sum_{i=0}^k\sum_{j=0}^{\infty}
E_{k-1}\bigl(
f,B(x+ih,c_k2^{-j}|h|)
\bigr),
\label{eq-67-telescoping}
\end{align}
where $c_k\in (0,\infty)$ depends only on $k$.
For any $z\in\mathbb{R}^n$ and $r\in (0,\infty)$,
define
\begin{align*}
H(z,r):=
\sum_{j=0}^{\infty}
E_{k-1}\bigl(f,B(z,c_k2^{-j}r)\bigr).
\end{align*}
From \eqref{eq-67-telescoping} and the change of variables $z=x+ih$, it follows that
\begin{align}
\int_{\mathbb{R}^n}\int_{\mathbb{R}^n}
\frac{|\Delta_h^k f(x)|^p}{|h|^{n+skp}}
\,dh\,dx
&\lesssim
\int_{\mathbb{R}^n}\int_{\mathbb{R}^n}
\frac{H(z,|h|)^p}{|h|^{n+skp}}
\,dh\,dz
\nonumber\\
&\lesssim
\int_{\mathbb{R}^n}\int_0^\infty
H(z,r)^p\frac{dr}{r^{1+skp}}\,dz.
\label{eq-67-reduction}
\end{align}
Applying the higher-order Poincar\'e inequalities, we obtain
\begin{align*}
H(z,r)
\lesssim
\sum_{j=0}^{\infty}
\min\bigl\{
\|f\|_{\operatorname{BMO}},2^{-jk}r^k
M\bigl(|\nabla^k f|\bigr)(z)
\bigr\}.
\end{align*}
For any $z\in\mathbb{R}^n$, let
$\rho(z):=
[\frac{\|f\|_{\operatorname{BMO}}}{M\bigl(|\nabla^k f|\bigr)(z)}]^{\frac{1}{k}}.$
Then
\begin{align*}
H(z,r)
\lesssim
\begin{cases}
r^kM\bigl(|\nabla^k f|\bigr)(z),&0<r\leq\rho(z),\\[2mm]
\|f\|_{\operatorname{BMO}}\left[1+\log\displaystyle\frac{r}{\rho(z)}\right],
&r>\rho(z).
\end{cases}
\end{align*}
Note that
\begin{align*}
\int_0^{\rho(z)}
H(z,r)^p\frac{dr}{r^{1+skp}}
&\lesssim
[M\bigl(|\nabla^k f|\bigr)(z)]^p
\int_0^{\rho(z)}
r^{kp-skp-1}\,dr\\
&=
\frac{
\|f\|_{\operatorname{BMO}}^{(1-s)p}[M\bigl(|\nabla^k f|\bigr)(z)]^{sp}
}{kp(1-s)},
\end{align*}
and
\begin{align*}
\int_{\rho(z)}^\infty
H(z,r)^p\frac{dr}{r^{1+skp}}
&\lesssim
\|f\|_{\operatorname{BMO}}^p\rho(z)^{-skp}
\int_1^\infty
(1+\log t)^pt^{-1-skp}\,dt\notag\\
&\lesssim
\|f\|_{\operatorname{BMO}}^{(1-s)p}[M\bigl(|\nabla^k f|\bigr)(z)]^{sp}.
\end{align*}
These estimates further imply that
\begin{align*}
\int_0^\infty
H(z,r)^p\frac{dr}{r^{1+skp}}
\lesssim
\frac{\|f\|_{\operatorname{BMO}}^{(1-s)p}}{1-s}[M\bigl(|\nabla^k f|\bigr)(z)]^{sp}.
\end{align*}
From this and \eqref{eq-67-reduction}, we deduce that
\begin{align*}
\int_{\mathbb{R}^n}\int_{\mathbb{R}^n}
\frac{|\Delta_h^k f(x)|^p}{|h|^{n+skp}}
\,dh\,dx
\lesssim
\frac{\|f\|_{\operatorname{BMO}}^{(1-s)p}}{1-s}
\int_{\mathbb{R}^n}
\bigl[M(|\nabla^k f|)(x)\bigr]^{sp}\,dx,
\end{align*}
which, together with the quantitative maximal inequality
\begin{align*}
\int_{\mathbb{R}^n}
\bigl[M(|\nabla^k f|)(x)\bigr]^{sp}\,dx
\lesssim
\frac{1}{sp-1}
\int_{\mathbb{R}^n}|\nabla^k f(x)|^{sp}\,dx,
\end{align*}
further implies \eqref{eq-lebg}.
This completes the proof of Theorem \ref{thm-L}.
\end{proof}

\begin{remark}
Theorem \ref{thm-L} when $k=1$ reduces to \cite[Theorem 1]{V23} with
$\Omega=\mathbb R^n$. The higher order results are new.
\end{remark}

\subsection{Further Function Spaces}\label{sec-morespace}

The preceding results can also be applied to several other function
space scales. For variable Lebesgue spaces, we refer to
\cite{CF13,DHR09,KR91,NS12} for the basic theory, to
\cite[Theorem 3.16]{CF13} for the boundedness of maximal operators,
and to \cite[Proposition 3.3]{WYY22} for
Calder\'on--Lozanovski\u{\i} products. Mixed-norm Lebesgue spaces originate
from \cite{Hor60,BP61}, and the structural and operator theoretic
properties needed here are treated in
\cite{CG20,CGN17,CGN19,HY21,HLYY19}. For Orlicz-slice spaces, see
\cite{ZYYW19,AM19,AP17,Ho19,Ho21,Ho22,Ho23,KNTYY07}, and for local
and global generalized Herz spaces, see
\cite{herz,RS20,GLY98,HY99,HS25,HWYY23,LY96,LYH22,ZYZ22}.

In the diagonal case, the resulting Gagliardo--Nirenberg inequalities
and BBM type conclusions reduces to, in the common parameter ranges, the
corresponding results in \cite{HLYY-ineq}. In the diagonal first order
case, the MS type conclusions improve the corresponding cases
of \cite[Theorem 2.16(i)]{PYYZ24}. The off-diagonal fractional
inequalities, the BMO endpoint estimates, and the corresponding
off-diagonal BBM and MS type results appear to be new
also in these function
space scales.

\bigskip

\noindent Pingxu Hu, Yinqin Li, Dachun Yang and
Wen Yuan.

\medskip

\noindent Laboratory of Mathematics and Complex Systems
(Ministry of Education of China),
School of Mathematical Sciences, Institute for Advanced Study,
Beijing Normal University,
Beijing 100875, The People's Republic of China

\smallskip

\noindent {\it E-mails}: \texttt{pingxuhu@mail.bnu.edu.cn} (P. Hu)

\noindent\phantom{\it E-mails }
\texttt{yinqli@mail.bnu.edu.cn} (Y. Li)

\noindent\phantom{\it E-mails }
\texttt{dcyang@bnu.edu.cn} (D. Yang)

\noindent\phantom {\it E-mails }
\texttt{wenyuan@bnu.edu.cn} (W. Yuan)

\end{document}